\documentclass[1 [leqno,11pt]{amsart}
\usepackage{amsmath,mathrsfs}
\usepackage{amssymb, amsmath,fourier}
\usepackage{graphicx}

\def\inte#1{
\displaystyle\mathop{#1\kern0pt}^\circ }

\let\pa=\partial

\let\f=\frac

\def\pa{\partial}

\def\fF{\mathfrak{F}}

\def\virgp{\raise 2pt\hbox{,}}
\def\cdotpv{\raise 2pt\hbox{;}}

\def\eqdefa{\buildrel\hbox{\footnotesize def}\over =}

\def\C{\mathop{\mathbb C\kern 0pt}\nolimits}
\def\DD{\mathop{\mathbb D\kern 0pt}\nolimits}
\def\EE{\mathop{{\mathbb E \kern 0pt}}\nolimits}
\def\K{\mathop{\mathbb K\kern 0pt}\nolimits}
\def\N{\mathop{\mathbb N\kern 0pt}\nolimits}
\def\Q{\mathop{\mathbb Q\kern 0pt}\nolimits}
\def\R{\mathop{\mathbb R\kern 0pt}\nolimits}
\def\SS{\mathop{\mathbb S\kern 0pt}\nolimits}
\def\NN{\mathop{\mathbb N\kern 0pt}\nolimits}
\def\ZZ{\mathop{\mathbb Z\kern 0pt}\nolimits}
\def\TT{\mathop{\mathbb T\kern 0pt}\nolimits}
\def\P{\mathop{\mathbb P\kern 0pt}\nolimits}

\newcommand{\Z}{{\ZZ}}

\def\na{\nabla}

\newcommand{\beq}{\begin{equation}}
\newcommand{\eeq}{\end{equation}}
\newcommand{\ben}{\begin{eqnarray}}
\newcommand{\een}{\end{eqnarray}}
\newcommand{\beno}{\begin{eqnarray*}}
\newcommand{\eeno}{\end{eqnarray*}}

\newtheorem{defi}{Definition}[section]
\newtheorem{thm}{Theorem}[section]
\newtheorem{lem}{Lemma}[section]
\newtheorem{rmk}{Remark}[section]
\newtheorem{col}{Corollary}[section]
\newtheorem{prop}{Proposition}[section]

\begin{document}
\title[Sharp bounds for Boltzmann and Landau  operators]
{Sharp bounds for Boltzmann and Landau   collision operators\medskip \\Estimations précises pour les opérateurs de collision de Boltzmann et de Landau}

\author[L.B. He]{Ling-Bing He}
\address[L.B. He]{Department of Mathematical Sciences, Tsinghua University\\
Beijing 100084,  P. R.  China.} \email{lbhe@math.tsinghua.edu.cn}

\maketitle
\begin{abstract}   
 The aim of the work is to provide a stable method to get sharp bounds for Boltzmann and Landau operators in weighted Sobolev spaces and in anisotropic spaces.  The results and proofs have the following main features and innovations:
\medskip

\noindent $\bullet$ All the sharp bounds are given for the original Boltzmann and Landau operators. 
The sharpness means the lower and upper bounds for the operators are consistent with the behavior of  the linearized operators. Moreover, we make clear the difference between the bounds for the original operators and those for the linearized ones. It will be useful for the well-posedness of the original equations.

\noindent $\bullet$ 
According to the Bobylev's formula, we introduce two types of dyadic decompositions performed in both phase and frequency spaces  to make full use of the interaction and the cancellation. It allows us to see clearly which part of the operator behaves like a Laplace type operator and which part is dominated by the anisotropic structure. It is the key point  to get the sharp bounds in weighted Sobolev spaces and in anisotropic spaces.
 
\noindent $\bullet$   
Based on the geometric structure of the elastic collision, we make a geometric decomposition to capture the anisotropic structure of the  collision operator. More precisely, we make it explicit that the fractional Laplace-Beltrami operator  really exists in the structure of the collision operator. It enables us to derive the sharp bounds  in anisotropic spaces and then complete the entropy dissipation estimates. 
   
\noindent $\bullet$   
The structures mentioned above are  so stable that we can apply them to the rescaled Boltzmann collision operator in the process of the grazing collisions limit. Then we get the sharp bounds for the Landau collision operator by passing to the limit.  We remark that our analysis used here will shed light on the investigation of the asymptotics from Boltzmann equation to Landau equation.

\medskip

 \noindent Résumé. L'objectif de ce travail est de fournir une méthode robuste pour obtenir des estimations précises pour les opérateurs de Boltzmann et de Landau dans des espaces de Sobolev à poids et des espaces anisotropes. Les résultats et leur démonstration font ressortir les innovations suivantes :
\medskip

 \noindent $\bullet$ Toutes les estimations précises concernent les opérateurs originaux de Boltzmann et de Landau. Le mot `précis' se réfère au fait que les estimations sont cohérentes avec le comportement des opérateurs linéarisés correspondants. Ceci est utile pour étudier le caractère bien posé des équations originales.

\noindent $\bullet$ En accord avec la formule de Bobylev, on introduit deux types de décomposition dyadique, dans l'espace des phases et dans celui des fréquences, afin d'utiliser au maximum les annulations. Cela nous permet de voir clairement quelle partie de l'opérateur se comporte comme un opérateur de type Laplacien, et quelle partie est dominée par la structure anisotrope. 

\noindent $\bullet$ En se basant sur la structure géométrique des collisions élastiques, on fait une décomposition géométrique pour capturer la structure anisotrope de l'opérateur de collision. Plus précisément, on explicite le fait que l'opérateur de Laplace-Beltrami apparaît bien dans l'opérateur de collision. Cela nous permet d'obtenir des estimations précises dans des espaces anisotropes et de finaliser les estimations sur la dissipation d'entropie.

\noindent $\bullet$ Les structures mentionnées ci-dessus sont si robustes qu'on peut les retrouver dans la limite des collisions rasantes. On obtient ainsi des estimations précises pour le noyau de collision de Landau en passant à la limite. On remarque que la présente analyse éclaire le passage à la limite de l'équation de Boltzmann vers celle de Landau.

\smallskip


Mots-clefs: les  équations de Boltzmann et de Landau; la structure anisotrope; la limite des collisions rasantes.

Class. math.: 	35Q20;35A23;35Q62.

\end{abstract}

\tableofcontents

\section{Introduction}
 The aim of the present work is  to provide a stable method to give a complete description of the behavior of the Boltzmann and Landau collision operators. We remark that it is related closely to the derivation of the Landau equation from the Boltzmann equation and also the asymptotics of the Boltzmann equation from  short-range interactions to  long-range interactions.
 
 We first recall that the Boltzmann equation reads:
\begin{eqnarray}\label{homb}
\partial _t f+v\cdot \na_x f  =Q (f ,f ),
 \label{bol}
\end{eqnarray}
where $f(t,x,v)\geq 0$ is a distribution
function   of  colliding particles which, at time
$t\geq 0$  and position $x\in  \TT^3$, move with velocity
$v\in\R^3$.  We remark that the Boltzmann equation is one of the fundamental equations of mathematical physics and is a cornerstone of statistical physics.

The Boltzmann collision operator $Q $ is a bilinear
operator which acts only on the velocity variable $v$, that is,
\beno Q(g,f)(v)\eqdefa
\int_{\R^3}\int_{\SS^{2}}B(v-v_*,\sigma)(g'_*f'-g_*f)d\sigma dv_*.
\eeno Here we use the standard shorthand  $f=f(t,x,v)$, $g_*=g(t,x, v_*)$,
$f'=f(t,x,v')$, $g'_*=g(t,x,v'_*)$  where $(v,v_*)$ and $(v',v_*')$ are the velocities of particles before and after the collision. Here $v'$ and $v_*'$ are given by
\begin{eqnarray}\label{e3}
v'=\frac{v+v_{*}}{2}+\frac{|v-v_{*}|}{2}\sigma\ ,\ \ \
v'_{*}=\frac{v+v_{*}}{2}-\frac{|v-v_{*}|}{2}\sigma\
,\qquad\sigma\in\SS^{2}.
\end{eqnarray}
The representation is consistent with the physical laws of the elastic collision: \beno
v+v_*&=&v'+v_*',\\ |v|^2+|v_*|^2&=&|v'|^2+|v'_*|^2. \eeno

In the definition of $Q$, $B $  is  called the Boltzmann collision kernel. It is always
  assumed that $B\ge0$ and that $B$ depends only on $|v-v_{*}|$ and $\frac{v-v_{*}}{|v-v_{*}|}\cdot\sigma$. Usually, we introduce the angle variable $\theta$ through
  $\cos\theta=\frac{v-v_{*}}{|v-v_{*}|}\cdot\sigma$.  Without loss of generality, we may
assume that $B(v-v_{*},\sigma)$ is supported in the set
 $0\leq\theta\leq\frac{\pi}{2}$ , i.e,
$\frac{v-v_{*}}{|v-v_{*}|}\cdot\sigma\ge0
$.   Otherwise, $B$ can be replaced by its symmetrized form:\ben\label{a4}
\bar{B} (v-v_{*},\sigma)=[B (v-v_{*},\sigma)+B(v-v_{*},-\sigma)]\mathbf{1}_{\{\frac{v-v_{*}}{|v-v_{*}|}\cdot\sigma\ge0\}}.\een
Here, $\mathbf{1}_A$ is the characteristic function of the set $A$.  In this paper, we consider the  collision kernel
satisfying the following assumptions:

 \smallskip
 
{\bf (A1).} The kernel $B(v-v_{*},\sigma)$ takes a product form 
\ben \label{a1} B(v-v_{*},\sigma)=\Phi(|v-v_*|)b(\cos\theta),\een
 where both $\Phi$ and $b$ are nonnegative functions.
 
 \smallskip
 
{\bf (A2).} The angular function $b(t)$  
satisfies  for $\theta\in [0, \pi/2]$, \begin{eqnarray}\label{a2} K\theta^{-1-2s} \le\sin\theta
b(\cos\theta)\le K^{-1}
 \theta^{-1-2s},  \quad
\mbox{with}\   0<s<1,\ K>0.
\end{eqnarray}

\smallskip

{\bf (A3).} The kinetic factor $\Phi$ takes the form \ben\label{a3}
\Phi(|v-v_*|)=|v-v_{*}|^{\gamma},\een
 where the parameter $\gamma$   verifies that  $\gamma+2s>-1$ and $\gamma\le 2$.

\begin{rmk} For inverse repulsive potential, it hold that
$\gamma=\frac{p-5}{p-1}$ and $s=\frac{1}{p-1}$ with $p>2$. It is easy to check
that $\gamma+4s=1$ which makes sense of the assumption
$\gamma+2s>-1$. Generally, the case $\gamma>0$, the case $\gamma=0$, and the case
$\gamma<0$ correspond to so-called hard, maxwellian, and soft
potentials respectively.
\end{rmk}

\begin{rmk} If we replace the assumption \eqref{a2} by
\ben\label{abc2}  K\theta^{-1-2s}\bigg(1-\psi\big(\f{\sin(\theta/2)}{\epsilon}\big)\bigg) \le\sin\theta
b(\cos\theta)\le K^{-1}
 \theta^{-1-2s}\bigg(1-\psi\big(\f{\sin(\theta/2)}{\epsilon}\big)\bigg),  \een
where $\psi$ is a  non-negative and smooth function  defined in \eqref{defpsivarphi}, then the mathematical problem of the asymptotics of the Boltzmann equation from short-range interactions to long-range interactions  can be formulated by the limit in which the parameter $\epsilon$ in \eqref{abc2} goes to zero. We remark that for fixed $\epsilon$, \eqref{abc2} corresponds to the famous Grad's cut off assumption for the kernel $B$.
\end{rmk}

The solutions of the Boltzmann equation \eqref{homb} enjoy the fundamental properties of the conservation of  mass,  momentum and the kinetic energy, that is, for all $t\ge0$,
\beno  \iint_{\TT^3\times\R^3} f(t,x,v)\phi(v)dvdx=\iint_{\TT^3\times\R^3} f(0,x, v)\phi(v)dvdx,\quad \phi(v)=1,v,|v|^2.\eeno
Moreover, if the entropy $H(f)$ is defined by  $$H(f)(t)\eqdefa \iint_{\TT^3\times\R^3} f \ln f  dvdx,$$ then the celebrated   Boltzmann's $H$-theorem predicts  that  the entropy is decreasing over time, which formally is
\beno \f{d}{dt}H(f)(t)= \iint_{\TT^3\times\R^3} Q (f ,f )\ln f   dvdx\le0.\eeno

  Before introducing the Landau equation, let us give the definition of the grazing collision limit. We introduce the rescaled Boltzmann's kernel $B^\epsilon$ which verifies the assumption   \smallskip
 
{\bf (B1)}.  The rescaled Boltzmann's kernel $B^\epsilon$ takes the simple product form \beno
  B^\epsilon(v-v_{*},\sigma)=\Phi(|v-v_*|)b^\epsilon(\cos\theta),\eeno where the kinetic factor $\Phi$ satisfies \eqref{a3} and
the angular function $b^\epsilon(t)$
satisfies  for $\theta\in [0, \pi/2], $\ben\label{ab2} \sin\theta
b^\epsilon(\cos\theta)=K'\epsilon^{2s-2}\psi\big(\f{\sin(\theta/2)}{\epsilon}\big)\theta^{-1-2s}, \een
with  $0<s<1$. Here $K'$ is a positive constant and the function $\psi$ is defined in \eqref{defpsivarphi}.
\smallskip
 
 The assumption \eqref{ab2} means that the deviation angles  between relative velocities  before  and after collisions are restricted to be less than $\epsilon$. Mathematically the grazing collision limit is defined by the process in which the parameter $\epsilon$ goes to zero.   
Thanks to  the full Taylor expansion,
by taking the limit  $\epsilon\rightarrow0$, the Boltzmann collision operator $Q^\epsilon$ with rescaled kernel $B^\epsilon$ will be reduced to a new operator, namely the Landau collision operator  $Q_L$,  defined by
\beno  Q_L(g,h)\eqdefa\nabla_v\cdot \bigg\{\int_{\R^3}a(v-v_*)[g(v_*)\nabla _v h(v)-\nabla _v g(v_*)h(v)]dv_*\bigg\}.\eeno
Here the nonnegative matrix $a$ is given by
\begin{equation}
a_{ij}(v)=\Lambda\left(\delta _{ij}-\frac{v_i v_j}{|v|^2}\right)|v|^{\gamma
+2}, \qquad \gamma \in [-3,1],  \label{e2}
\end{equation}
where $\Lambda$ is a positive constant and can be calculated by
\beno  \Lambda =\f{\pi}8 K'\int_0^{\pi/2} \psi(\theta)\theta^{1-2s}d\theta.\eeno 
Then the Landau equation can be written by
\ben\label{landau}
\pa_t f+v\cdot\na_x f =Q_L(f,f).
\een
 We remark that the equation was proposed by Landau in 1936 to model the   behavior
of a dilute plasma interacting through binary collisions.
We also mention that the Landau equation possesses all the properties known for the Boltzmann
equation, namely the conservation of mass, momentum and energy and the $H$-theorem.

\subsection{Motivation  and  short review   of the problem} 
The justification of the derivation of the Landau equation in the sense that the solution to the Boltzmann equation with rescaled kernel (see assumption {\bf  (B1)}) will converge to the solution to the corresponding Landau equation in the grazing collision limit  had been proved by several authors. We refer readers to \cite{villani1} and the  references therein to check details.  However, the physical problem of the justification is   formulated as a higher-order correction to the limit. In other words, we should establish some kind of the asymptotic formula to the solutions in the limit process.
Suppose that $f^\epsilon_B$ and $f_L$ represent the solutions to Boltzmann and Landau equations. In \cite{he}, for the homogeneous case,  that is, the solution does not depend on the position variable $x$, the author showed that the following asymptotic formula \ben\label{asymformula} f^\epsilon_B= f_L+O(\epsilon)
\een
holds globally or locally in Sobolev spaces
for almost all physical potentials except for Coulomb potential.  It shows that the Landau equation is a good approximation to the Boltzmann equation when the parameter $\epsilon$ is small enough. It  gives the validity of the Landau equation. However it is very difficult to extend the similar result to the inhomogeneous case even in the close-to-equilibrium setting. The main obstruction is  lack of a complete description of the asymptotic behavior of collision operator in the limit process since the behavior of the operator is very sensitive to the parameter $\epsilon$.   The same problem happens when we study the asymptotics of the Boltzmann equation from short-range interactions to long-range interactions under the assumption \eqref{abc2}.  We remark that the investigation of the second asymptotics is related closely to the construction of approximate solutions to the equation with long-range interactions and also to the jump phenomenon of the spectral gap of the linearized collision operator (see \eqref{linearizedcop} for the definition) for soft potentials ($-2s\le \gamma<0$) when $\epsilon$ goes to zero.

Motivated by these two asymptotic problems, in the present work,  we try to find out some stable structure inside the Boltzmann collision operator to obtain the sharp bounds and then extend them to the Landau operator via the grazing collision limit. Before going further, let us give a short review on the estimates of the Boltzmann collision operator. For simplicity, we only address the estimates for the maxwellian molecular case, that is, $\gamma=0$.  

In what follows, we assume that all the functions depend only  on the $v$ variable recalling that the collision operator $Q$ acts only on the $v$ variable. We will use the duality method to get the lower and upper bounds of the operator. The inner product of $f$ and $g$ over $\R^3_v$, namely $\langle f, g\rangle_v$,  is defined by  \beno \langle f, g\rangle_v\eqdefa \int_{\R^3} f(v)g(v)dv. \eeno
 Then by change of variables(see \cite{advw}),  we have
 \beno  \langle Q(g, f), f\rangle_v =\iint_{v,v_*\in\R^3,\sigma\in\SS^2} B(|v-v_*|,\sigma)g_*f(f'-f)d\sigma dv_*dv. \eeno
  It is easy to check 
\ben\label{QDEC}  \langle Q(g, f), f\rangle_v&=&-\underbrace{\f12 \iint_{v,v_*\in\R^3,\sigma\in\SS^2} B(|v-v_*|,\sigma)g_*(f' -f)^2d\sigma dv_*dv}_{\mathcal{E}_g(f)}\notag\\&&+\f12
\iint_{v,v_*\in\R^3,\sigma\in\SS^2} B(|v-v_*|,\sigma)g_*(f'^2-f^2)d\sigma dv_*dv. \een 
In \cite{advw}, the authors gave the first  coercivity estimate of $\mathcal{E}_g(f)$   which can be stated as
\beno \mathcal{E}_g(f)\ge C_g\|f\|_{H^s}^2-\|g\|_{L^1}\|f\|_{L^2}^2,\eeno
where $C_g$ is a constant depending on the lower bound of $\|g\|_{L^1}$ and the upper bounds of $\|g\|_{L^1_2}$ and $\|g\|_{L\log L}$ (see the definitions in Section 1.2). It gives a positive answer to the conjecture that the Boltzmann collision operator behaves like a fractional Laplace operator, that is,
\ben\label{behavqb} -Q(g, \cdot)\sim C_g (-\triangle_v)^s+LOT. \een
This conjecture was further confirmed by the upper bound for the collision operator. Mathematically, it reads
that if  $a,b\in\R$ with $a+b=2s$, then \ben\label{uboundWSP} |\langle Q(g, h), f\rangle_v |\lesssim \|g\|_{L^1_{2s}}\|h\|_{H^a_{s}}\|f\|_{H^b_{s}}, \een
which was proved in \cite{amuxy4}. The weighted Sobolev space $H^m_l$ is defined in Section 1.2. This upper bound is sharp in the sense that we have the freedom of choosing derivatives for functions $h$ and $f$. For the general potentials,
we refer readers to \cite{amuxy4,chenhe1,he1}  on the lower and upper bounds in weighted Sobolev spaces.

Combining the lower and upper bounds, one may find \ben\label{boundWSP} C_g\|f\|_{H^s}^2-\|g\|_{L^1}\|f\|_{L^2}^2 \le\langle -Q(g, f), f\rangle_v \lesssim \|g\|_{L^1_{2s}} \|f\|_{H^s_{s}}^2.\een We remark that additional condition for $f$ is imposed in the upper bound. The reason lies in the fact that some anisotropic structure is  hidden in the operator which  can not be observed in weighted Sobolev spaces.  Indeed, in \cite{chuh},   the authors show that the linearized collision operator   $\mathcal{L}_B$, which is
 defined by \ben\label{linearizedcop} \mathcal{L}_B f\eqdefa -\mu^{-\f12}\big(Q(\mu,\mu^{\f12}f)+Q(\mu^{\f12}f,\mu)\big),\quad \mu=\f1{(2\pi)^{3/2}} e^{-\f{|v|^2}2}, \een   is a self-adjoint operator and has explicit eigenvalues and eigenfunctions. In particular, the eigenfunction $E(v)$ takes the form  
 \beno   E(v)=f(|v|^2)Y(\sigma),\eeno where $v=|v|\sigma$, $f$ is a radial function and $Y$ is a real spherical harmonic. In \cite{villani2},  Villani proved that
 \beno  Q_L(f,f)= 3\na\cdot(\na f+vf)-\big(P_{ij}(f)\pa_{ij} f+\na\cdot (vf)\big)+\triangle_{\SS^2} f,\eeno
where $P_{ij}(f)=\int_{\R^3} fv_iv_jdv$ and   $(-\triangle_{\SS^2})$ is the Laplace-Beltrami operator on the unit sphere.  The special form of  the eigenfunction of $\mathcal{L}_B$ and the Laplace-Beltrami operator in the expression of $Q_L$ indicate that there should be some anisotropic structure inside the operator.

 Mathematically the first attempt to capture the anisotropic structure of  the operator were due to \cite{amuxy1} and \cite{grstr1}(see also \cite{amuxy2} and \cite{grstr2}). In fact,  to describe the behavior of the operator, they introduce two types of anisotropic norms which are defined by
\ben\label{be1}  \VERT f\VERT^2\eqdefa\|f\|_{L^2_{s}}^2+\iint_{v,v_*\in\R^3,\sigma\in\SS^2} b(\cos\theta)\mu_*(f'-f)^2 d\sigma dv_* dv\een 
and
\ben\label{be2}  \|f\|_{N^{s}}^2\eqdefa\|f\|_{L^2_{s}}^2+\underbrace{\iint_{v,v'\in\R^3}\langle v\rangle^{s+1/2}\langle v'\rangle^{s+1/2} \f{|f-f'|^2}{d(v,v')^2}\mathrm{1}_{d(v,v')\le1}dvdv'}_{\mathcal{A}(f)},\een 
where $d(v,v')=\sqrt{|v-v'|^2+\f14(|v|^2-|v'|^2)^2}$ and $\langle v\rangle=\sqrt{1+|v|^2}$. We remark the second term in the righthand side of the definition \eqref{be1} is exactly the term $\mathcal{E}_\mu(f)$ defined in \eqref{QDEC}. Then the estimates can be stated as
\ben\label{be5} \VERT f \VERT^2-\|f\|_{L^2}^2\lesssim \langle \mathcal{L}_B f,f \rangle_v\lesssim   \VERT f \VERT^2.\een Moreover,  the upper bound can be generalized to  the original collision operator $Q$: \beno  |\langle Q(g,h), f\rangle_v|\lesssim \|g\|_{L^2} \VERT h \VERT \VERT f \VERT.\eeno 
We emphasize that these two anisotropic norms are crucial to prove  global small solutions in the perturbation framework. Since they  are given in an implicit way, it does not help to understand the anisotropic property of the operator.   In the grazing collision limit, for $\mathcal{E}_\mu(f)$, on one hand,  it is stable  since it is given in an implicit way.  On the other hand, we  have no idea on the limit of this quantity. The similar problem occurs   for $\mathcal{A}(f)$.

Very recently, two groups gave explicit description of the anisotropic behavior of the linearized operator. Both of them started with the same point, that is, the well understanding of the behavior of the linearized Landau operator $\mathcal{L}_L$. In fact, they proved that
\beno  \mathcal{L}_L\sim (-\triangle+|v|^2/4)+(-\triangle_{\SS^2})\sim (-\triangle+|v|^2/4)+|D_v\times v|^2,\eeno
recalling that $-\triangle_{\SS^2}=\sum\limits_{1\le i<j\le 3}(v_i\pa_j-v_j\pa_i)^2$.
In \cite{ahl}, the authors show that
\ben\label{be3}  \langle \mathcal{L}_B f,f \rangle_v+\|f\|_{L^2}^2\sim \|f\|_{L^2_s}^2+\|f\|_{H^s}^2+\||D_v\times v|^sf\|_{L^2}^2,\een 
where $|D_v\times v|^s$ is a pseudo-differential operator with the symbol $|\xi\times v|^s$.   Thanks to \cite{chuh}, by comparing the eigenvalues between Boltzmann and Landau collision operators, in \cite{lmpx}, the authors show that
\ben\label{blconnection} \mathcal{L}_B\sim \mathcal{L}_L^s\een and
\ben \label{be4}    \langle \mathcal{L}_B f,f \rangle_v+\|f\|_{L^2}^2\sim \|f\|_{L^2_s}^2+\|f\|_{H^s}^2+\|(-\triangle_{\SS^2})^{s/2}f\|_{L^2}^2.\een   We remark that the strategy to obtain   \eqref{be3} and \eqref{be4}  depends heavily on the linearized structure, for instance, the symmetric property of the operator and the fine properties of the Maxellian state $\mu$. Therefore it  cannot be generalized to the original collision operator and to the rescaled operator under the assumption \eqref{abc2} or \eqref{ab2}.  
\medskip

The short review can be summarized as follows:
\begin{enumerate}
\item The previous results on the description of the behavior of the operator  are given in an implicit way or in an unstable way. It means that the  anisotropic  structure  is still mysterious and  not captured well.
\item  The upper bound of the collision operator is far away from the sharpness. For instance, recalling \eqref{blconnection}, the typical upper bound for the operator should be in the form 
\ben\label{sharpform}  |\langle Q(g,h), f\rangle_v|\lesssim  C(g) \|(\mathcal{L}_L^{a/2}+1) h\|_{L^2}\|(\mathcal{L}_L^{b/2} +1)f\|_{L^2},\een
where $a,b\in[0,2s]$ with $a+b=2s$.  
\item For the linearized collision operator $\mathcal{L}^\epsilon_B$ with rescaled kernel $B^\epsilon$ under the assumption \eqref{abc2} or \eqref{ab2}, we have no available results on the complete  description  of the operator.
we also have no idea on the sharp bounds of the original collision operator $Q^\epsilon$ in the process of the limit.

\end{enumerate}

We end this subsection by the remark that points $(1)$ and $(2)$ are related closely to the Cauchy problem for the original equation \eqref{bol}. And the point $(3)$ is related to  the investigation of two types of asymptotics mentioned before.

\subsection{Notations and main results} Before stating our main results,  we first introduce the function spaces 
which will be used throughout the paper.  

\begin{enumerate}
\item For
any integer $N\geq 0$, we define the Sobolev space $H^N $ by
\begin{equation*} H^N \eqdefa\bigg\{f( v)| \sum_{|\alpha | \leq N}\|
\partial _v^{\alpha }f\|_{L^2 }<+\infty\bigg\},
\end{equation*} where the multi-index $\alpha =(\alpha _1,\alpha _2,\alpha _3)$ with
$|\alpha |=\alpha _1+\alpha _2+\alpha _3$ and $\pa^\alpha_v=\pa^{\alpha_1}_{v_1}\pa^{\alpha_2}_{v_1}\pa^{\alpha_2}_{v_3}$.
\item For   real numbers $m, l $, we define the weighted Sobolev space $H^{m}_l$ by
\begin{equation*}
H^{m}_l \eqdefa\bigg\{f(v)|  \|f\|_{H^m_l}=\| \langle D\rangle^m \langle \cdot\rangle^l f
  \|_{L^2}<+\infty\bigg\},
\end{equation*}
where  $\langle v\rangle\eqdefa(1+|v|^2)^{\f12}$. $a(D)$ is a   pseudo-differential operator with the symbol
$a(\xi)$ and it is defined as
\beno  \big(a(D)f\big)(x)\eqdefa\f1{(2\pi)^3}\int_{\R^3}\int_{\R^3} e^{i(x-y)\xi}a(\xi)f(y)dyd\xi.\eeno
\item The general weighted Sobolev space $W^{N,p}_l$  with $p\in [1, \infty)$ is defined as 
\beno 
W^{N,p}_l\eqdefa \bigg\{f(v)|\|f\|_{W^{N,p}_l}=\sum_{|\alpha|\le N} \bigg(\int_{\R^3}  |\partial^\alpha_v f(v)|^p\langle v\rangle^{lp}dv \bigg)^{1/p}<\infty
\bigg\}.
\eeno
In particular, if $N=0$, we  introduce the weighted $L^p_l$ space defined as   
\beno 
L^p_l\eqdefa \bigg\{f(v)| \|f\|_{L^p_l}=\bigg(\int_{\R^3} |f(v)|^p\langle v
\rangle^{l p}dv\bigg)^{\f1{p}}<\infty \bigg\}. 
\eeno
\item The $L\log L$ space  is defined as
\beno L\log L\eqdefa \bigg\{f(v)|  \|f\|_{L\log L}=\int_{\R^3}
|f|\log (1+|f|)dv<\infty \bigg\}.\eeno
\end{enumerate}

  Next we list some notations which will be used in the paper. We write $a\lesssim b$ to indicate that  there is a
uniform constant $C,$ which may be different on different lines,
such that $a\leq Cb$.  We use the notation $a\sim b$ whenever $a\lesssim b$ and $b\lesssim
a$. The notation $a^+$ means the maximum value of $a$ and $0$. The weight function $W_l$ is defined by  $W_l(v)\eqdefa\langle v\rangle^l$ where $\langle v\rangle=\sqrt{1+|v|^2}$. 
Suppose $A$  and $B$ are two operators. Then the commutator $[A,B]$ between $A$ and $B$ is defined as follows: \beno [A, B]\eqdefa AB-BA.\eeno
We denote $C(\lambda_1,\lambda_2,\cdots, \lambda_n)$ by a constant depending on   parameters $\lambda_1,\lambda_2,\cdots, \lambda_n$.

\bigskip

Our first result is on the sharp upper bounds of the Boltzmann collision operator in weighted Sobolev spaces.

\begin{thm}\label{thmub}  Let $w_1,w_2\in \R, a,b\in [0,2s]$ with $w_1+w_2=\gamma+2s$ and  $a+b=2s$.  Then for  smooth functions $g, h$ and
$f$, we have
\begin{enumerate}
\item if $\gamma+2s>0$,     \ben\label{u1} |\langle Q(g,h),f\rangle_v|\lesssim
(\|g\|_{L^1_{\gamma+2s+(-w_1)^++(-w_2)^+}}+\|g\|_{L^2})\|h\|_{H^{a}_{w_1}}\|f\|_{H^{b }_{w_2}},
 \een
 \item if $\gamma+2s= 0$,     \ben\label{u2} |\langle Q(g,h),f\rangle_v|\lesssim
(\|g\|_{L^1_{w_3}}+\|g\|_{L^2})\|h\|_{H^{a}_{w_1}}\|f\|_{H^{b }_{w_2}},
 \een where $w_3=\max\{\delta,  (-w_1)^++(-w_2)^+\}$ with $\delta>0$ which is sufficiently small,
 \item if $-1<\gamma+2s<0$,     \ben\label{u3} |\langle Q(g,h),f\rangle_v|\lesssim
(\|g\|_{L^1_{w_4}}+\|g\|_{L^2_{-(\gamma+2s)}})\|h\|_{H^{a}_{w_1}}\|f\|_{H^{b }_{w_2}},
 \een where  $w_4=\max\{-(\gamma+2s), \gamma+2s+(-w_1)^++(-w_2)^+\}$.
\end{enumerate}
\end{thm}

\begin{rmk} The estimates (\ref{u1}-\ref{u3}) are sharp in weighted Sobolev spaces in the sense that we have the freedom of choosing  derivatives and  weights  for functions $h$ and $f$. 
They will play a crucial role in solvability of the Cauchy problem for the equation \eqref{bol} in weighted Sobolev spaces. In particular, we will use them frequently to balance the energy estimates to close the argument. We will explain it in detail in our forthcoming paper \cite{HEJIANG16}. 
\end{rmk}
\begin{rmk}We refer readers to the very recent work   \cite{jiangliu}  and \cite{ms} on the similar results by using the Bobylev's formula( see \eqref{bobylev}) or using the Random transform. We mention that their results only have the freedom of choosing  derivatives   for functions $h$ and $f$. 
\end{rmk}

\begin{rmk} The estimates are not sharp with respect to the function $g$. For instance, for 
$\gamma+2s>0$ and $\gamma>-\f32$, we can drop $\|g\|_{L^2}$ in the estimate. For $\gamma>-\f52$, we can replace  $\|g\|_{L^2_l}$ by  $\|g\|_{L^{\f32}_l}$ in the estimates.
For the maxwellian molecular($\gamma=0$),    the estimate can be rewritten as
 \ben\label{u5} |\langle Q(g,h),f\rangle_v|\le C(a,b)
\|g\|_{L^1_{2s+(-w_1)^++(-w_2)^+}} \|h\|_{H^{a}_{w_1}}\|f\|_{H^{b }_{w_2}},
 \een
where  $a+b=2s$ and $w_1+w_2=2s$. Here we remove the restriction $a,b\in[0,2s]$.
 \end{rmk}

Next we will state our new coercivity estimates for the Boltzmann collision operator:

\begin{thm}\label{thmlb} Suppose  $g$ is a non-negative and smooth function verifying that \ben\label{lbc} \|g\|_{L^1}> \delta \quad \mbox{and}\quad \|g\|_{L^1_2}+\|g\|_{L\log L}<\lambda,\een and  let $\mathbf{A}=0,1$. Then for  sufficiently small $\eta>0$, there exist  
 constants $\mathbf{C}_1(\delta, \lambda,\eta^{-1}), \mathbf{C}_2(\lambda, \delta)$, $\mathbf{C}_3(\delta, \lambda, \eta^{-1})$, $\mathbf{C}_4(\lambda, \delta)$ and $\mathbf{C}_5(\lambda, \delta)$  such that
\begin{enumerate}
\item if $\gamma+2s\ge0$,
\beno   \langle -Q(g,f),f\rangle_v&\gtrsim&\mathbf{A} \bigg[\mathbf{C}_1(\delta,\lambda,\eta^{-1}) \big(\| (-\triangle_{\SS^2})^{s/2}f \|_{L^2_{\gamma/2}}^2+
 \|  f \|_{H^s_{\gamma/2}}^2\big)\notag\\ && \quad- \eta \mathbf{C}_2(\delta,\lambda)\|  f\|_{L^2_{\gamma/2+s}}^2-\mathbf{C}_3(\delta,\lambda,\eta^{-1})\|f\|_{L^2_{\gamma/2}}^2\bigg]\\
 &&+\mathbf{C}_4(\lambda, \delta) \|  f \|_{H^s_{\gamma/2}}^2-\mathbf{C}_5(\lambda, \delta)\|f\|_{L^2_{\gamma/2}}^2,
\eeno

\item  if $-1-2s<\gamma<-2s$,  \beno \langle -Q(g,f),f\rangle_v&\gtrsim& \mathbf{A}\bigg[\mathbf{C}_1(\delta,\lambda,\eta^{-1})\big(\| (-\triangle_{\SS^2})^{s/2}f \|_{L^2_{\gamma/2}}^2+
 \|  f \|_{H^s_{\gamma/2}}^2\big)- \eta \mathbf{C}_2(\delta,\lambda) \|  f\|_{L^2_{\gamma/2+s}}^2\notag\\ && - \mathbf{C}_3(\delta,\lambda, \eta^{-1}) (1+ \|g\|_{L^p_{|\gamma|}}^{\f{(\gamma+2s+3)p}{(\gamma+2s+3)p-3}})\|f\|_{L^2_{\gamma/2}}^2\bigg]\\
 &&+\mathbf{C}_4(\lambda, \delta) \|  f \|_{H^s_{\gamma/2}}^2-\mathbf{C}_5(\lambda, \delta)(1+\|g\|_{L^p_{|\gamma|}}^{\f{(\gamma+2s+3)p}{(\gamma+2s+3)p-3}})\|f\|_{L^2_{\gamma/2}}^2,\eeno
 with $p>\f3{\gamma+2s+3}$.
\end{enumerate}
\end{thm}

\begin{rmk} Here $(-\triangle_{\SS^2})^{s/2}$ is the fractional Laplace-Beltrami operator. One may check the  definition of the operator in  \eqref{defilb1} and \eqref{defilb2}.
\end{rmk}
  
\begin{rmk} Compared to the lower bound of the functional $\langle  \mathcal{L}_B f,f\rangle$(see \eqref{be4}),   we cannot  get the control of  $ \|  f\|_{L^2_{\gamma/2+s}}^2$ from the below of the functional   $\langle -Q(g,f),f\rangle_v$. In fact, it is false to get
\ben\label{falsee} \langle -Q(g,f),f\rangle_v\gtrsim   \|  f\|_{L^2_{\gamma/2+s}}^2-LOT.\een
 Suppose it is true, then  combining with upper bound(see Remark \ref{varub}), we   derive  that
\beno  \|  f\|_{L^2_{\gamma/2+s}}^2\lesssim \big(\| (-\triangle_{\SS^2})^{s/2}f \|_{L^2_{\gamma/2}}^2+
 \|  f \|_{H^s_{\gamma/2}}^2\big). \eeno
It is obvious that the radial function does not verify such kind of the estimate.  Then we get the contradiction. It shows on one hand the lower bounds  are sharp in anisotropic spaces. On the other hand, the behavior of   the original operator is different from that of the linearized operator $\mathcal{L}_B$.
\end{rmk}

\begin{rmk} We need an additional assumption  $ f\in L^2_{\gamma/2+s}$ to obtain the fractional Laplace-Beltrami derivative, $\| (-\triangle_{\SS^2})^{s/2}f \|_{L^2_{\gamma/2}}$, in the coercivity estimates.    We comment that it is a bad news to the Cauchy problem for the equation  \eqref{bol}.  It means the framework used in the close-to-equilibrium setting cannot be applied to the original equation since  it will bring the trouble to close the energy estimates, in particular, in the case of $\gamma+2s>0$. We will explain it in detail in our forthcoming paper \cite{HEJIANG16}. \end{rmk}

As a direct consequence, now we can complete the entropy dissipation estimate as follows:
\begin{thm}\label{entropyproduction} Suppose   $f$ is a non-negative function  verifying the condition \eqref{lbc}. Then it holds
\beno  D_B(f)+\|f\|_{L^1_{w}}\gtrsim C(\lambda, \delta)\bigg(\|\sqrt{f}\|^2_{H^s_{\gamma/2}}+\|(-\triangle_{\SS^2})^{s/2}\sqrt{f}\|^2_{L^2_{\gamma/2}}\bigg), \eeno
where $w=\max\{\gamma+2s,2\}$ and $D_B(f)=-\int_{\R^3} Q(f,f)\ln f dv$.
\end{thm}
\begin{rmk} Suppose that $f$ is a solution to the spatially homogeneous Boltzmann equation with initial data $f_0\in L^1_2\cap L\log L$. Then we obtain that for $\gamma+2s\le 2$, there exists a constant $C(f_0)$ such that 
\beno  D_B(f)+\|f_0\|_{L^1_{2}}\gtrsim C(f_0) \bigg(\|\sqrt{f}\|^2_{H^s_{\gamma/2}}+\|(-\triangle_{\SS^2})^{s/2}\sqrt{f}\|^2_{L^2_{\gamma/2}}\bigg). \eeno
 \end{rmk} 
\begin{rmk} Compared with  the entropy production estimates in \cite{grstr2}, our results do not need additional regularity assumption on $f$ for soft potentials. 
\end{rmk}

Finally let us give the sharp bounds  of the Boltzmann collision operator in  anisotropic spaces:
\begin{thm}\label{thumbanti}
Let  $ a, b\in [0,2s], a_1, b_1, w_1, w_2\in \R$ with $a+b=2s$, $a_1+b_1=s$ and $w_1+w_2=\gamma+s$.   Then for  smooth functions $g, h$ and
$f$, it hold that
\begin{enumerate}
\item
if $\gamma>0$
\beno  |\langle Q(g,h),f\rangle_v|&\lesssim& (\|g\|_{L^1_{\gamma+2s}}+\|g\|_{L^1_{\gamma+s+(-w_1)^++(-w_2)^+}}+\| g\|_{L^2})\bigg((\|(-\triangle_{\SS^2})^{a/2}  h\|_{L^2_{\gamma/2}}+\|  h\|_{H^a_{\gamma/2}})\\&&\quad\times
(\|(-\triangle_{\SS^2})^{b/2} f)\|_{L^2_{\gamma/2}}+\| f\|_{H^b_{\gamma/2}})+ \|h\|_{H^{a_1}_{w_1}}\|f\|_{H^{b_1}_{w_2}}
\bigg), \eeno
\item if $\gamma=0$, for any $\delta>0$,
\beno  |\langle Q(g,h),f\rangle_v|&\lesssim& (\|g\|_{L^1_{2s+\delta}}+\|g\|_{L^1_{s+(-w_1)^++(-w_2)^+}}+\| g\|_{L^2})\bigg((\|(-\triangle_{\SS^2})^{a/2}  h\|_{L^2}+\|  h\|_{H^a})\\&&\quad\times
(\|(-\triangle_{\SS^2})^{b/2} f)\|_{L^2}+\| f\|_{H^b})+ \|h\|_{H^{a_1}_{w_1}}\|f\|_{H^{b_1}_{w_2}}
\bigg),\eeno
\item if $\gamma<0$,
\beno  |\langle Q(g,h),f\rangle_v|&\lesssim& (\|g\|_{L^1_{-\gamma+2s}}+\|g\|_{L^1_{\gamma+s+(-w_1)^++(-w_2)^+}}+\| g\|_{L^2_{-\gamma}})\bigg((\|(-\triangle_{\SS^2})^{a/2}  h\|_{L^2_{\gamma/2}}+\|  h\|_{H^a_{\gamma/2}})\\&&\quad\times
(\|(-\triangle_{\SS^2})^{b/2} f)\|_{L^2_{\gamma/2}}+\| f\|_{H^b_{\gamma/2}})+ \|h\|_{H^{a_1}_{w_1}}\|f\|_{H^{b_1}_{w_2}}
\bigg). \eeno
\end{enumerate}
\end{thm}

\begin{rmk} Thanks to the interpolation inequality \beno \|h\|_{H^{a_1}_{w_1}}^2\le \|h\|_{H^{2a_1}_{\gamma/2}} \|h\|_{L^2_{2w_1-\gamma/2}},\eeno  take $a_1=a/2, b_1=b/2$ and $w_1=\gamma/2+a/2, w_2=\gamma/2+b/2$,  then we have
\beno   |\langle Q(g,h),f\rangle_v|&\lesssim& C(g)
(\|(-\triangle_{\SS^2})^{a/2}  h\|_{L^2_{\gamma/2}}+\|  h\|_{H^a_{\gamma/2}}+\|h\|_{L^2_{\gamma/2+a}})\\&&\quad\times
(\|(-\triangle_{\SS^2})^{b/2} f)\|_{L^2_{\gamma/2}}+\| f\|_{H^b_{\gamma/2}}+\|f\|_{L^2_{\gamma/2+b}})\\
&\sim& C(g)\|(\mathcal{L}_L^{a/2}+1)h\|_{L^2}\|(\mathcal{L}_L^{b/2}+1)f\|_{L^2},\eeno
where in the last equivalence we use  the facts \eqref{blconnection} and  \eqref{be4}.  In other words, \eqref{sharpform} is proved. \end{rmk}

\begin{rmk}\label{varub} By taking $a=b=s$, $a_1=s, a_2=0$ $w_1=\gamma/2, w_2=\gamma/2+s$, we   deduce that for any $\eta>0$,
\beno  |\langle Q(g,h),f\rangle_v|&\lesssim& C(g)
(\|(-\triangle_{\SS^2})^{s/2}  h\|_{L^2_{\gamma/2}}+\eta^{-1}\|  h\|_{H^s_{\gamma/2}}) \\&&\quad\times
(\|(-\triangle_{\SS^2})^{s/2} f)\|_{L^2_{\gamma/2}}+\| f\|_{H^s_{\gamma/2}}+\eta\|f\|_{L^2_{\gamma/2+s}}).\eeno
Thanks to the symmetric property for functions $h$ and $f$ in the estimates, we also have 
\beno  |\langle Q(g,h),f\rangle_v|&\lesssim& C(g)
(\|(-\triangle_{\SS^2})^{s/2}  h\|_{L^2_{\gamma/2}}+\|  h\|_{H^s_{\gamma/2}}+\eta\|h\|_{L^2_{\gamma/2+s}}) \\&&\quad\times
(\|(-\triangle_{\SS^2})^{s/2} f)\|_{L^2_{\gamma/2}}+\eta^{-1}\| f\|_{H^s_{\gamma/2}}).\eeno
We remark that both estimates improve the previous upper bounds in the two senses: the first one is that we only need to assume that  one of $h$ and $f$ is in the space $L^2_{\gamma/2+s}$; the second one is  the free choice of the constant $\eta$ in the estimates which enables us to prove that  \eqref{falsee} is false. \end{rmk}

Thanks to the grazing collisions limit, we can extend the above estimates to the Landau collision operator.  

\begin{thm}\label{thmblandau}  Let  $w_1,w_2,w_3,w_4,a_1, b_1\in \R, a,b\in [0,2]$  with $w_1+w_2=\gamma+2$, $w_3+w_4=\gamma+1$,  $a+b=2$  and  $a_1+b_1=1$.  Then for  smooth functions $g, h$ and
$f$, it hold
\begin{enumerate}
\item if $\gamma+2>0$,     \ben\label{u1L} |\langle Q_L(g,h),f\rangle_v|\lesssim
(\|g\|_{L^1_{\gamma+2+(-w_1)^++(-w_2)^+}}+\|g\|_{L^2})\|h\|_{H^{a}_{w_1}}\|f\|_{H^{b }_{w_2}},
 \een
 \item if $\gamma+2= 0$,     \ben\label{u2L} |\langle Q_L(g,h),f\rangle_v|\lesssim
(\|g\|_{L^1_{w_5}}+\|g\|_{L^2})\|h\|_{H^{a}_{w_1}}\|f\|_{H^{b }_{w_2}},
 \een where $w_5=\max\{\delta,  (-w_1)^++(-w_2)^+\}$ with $\delta>0$ which is sufficiently small,
 \item if $\gamma+2<0$,     \ben\label{u3L} |\langle Q_L(g,h),f\rangle_v|\lesssim
(\|g\|_{L^1_{w_6}}+\|g\|_{L^2_{-(\gamma+2s)}})\|h\|_{H^{a}_{w_1}}\|f\|_{H^{b }_{w_2}},
 \een where  $w_6=\max\{-(\gamma+2), \gamma+2+(-w_1)^++(-w_2)^+\}$,
 
\item
if $\gamma>0$
\beno  |\langle Q_L(g,h),f\rangle_v|&\lesssim& (\|g\|_{L^1_{\gamma+2}}+\|g\|_{L^1_{\gamma+1+(-w_3)^++(-w_4)^+}}+\| g\|_{L^2})\bigg((\|(-\triangle_{\SS^2})^{a/2}  h\|_{L^2_{\gamma/2}}+\|  h\|_{H^a_{\gamma/2}})\\&&\quad\times
(\|(-\triangle_{\SS^2})^{b/2} f)\|_{L^2_{\gamma/2}}+\| f\|_{H^b_{\gamma/2}})+ \|h\|_{H^{a_1}_{w_3}}\|f\|_{H^{b_1}_{w_4}}
\bigg), \eeno
\item if $\gamma=0$,
\beno  |\langle Q_L(g,h),f\rangle_v|&\lesssim& (\|g\|_{L^1_{2+\delta}}+\|g\|_{L^1_{\gamma+1+(-w_3)^++(-w_4)^+}}+\| g\|_{L^2})\bigg((\|(-\triangle_{\SS^2})^{a/2}  h\|_{L^2}+\|  h\|_{H^a})\\&&\quad\times
(\|(-\triangle_{\SS^2})^{b/2} f)\|_{L^2}+\| f\|_{H^b})+  \|h\|_{H^{a_1}_{w_3}}\|f\|_{H^{b_1}_{w_4}}\bigg), \eeno
\item if $\gamma<0$,
\beno  |\langle Q_L(g,h),f\rangle_v|&\lesssim& (\|g\|_{L^1_{-\gamma+2}}+\|g\|_{L^1_{\gamma+1+(-w_3)^++(-w_4)^+}}+\| g\|_{L^2_{-\gamma}})\bigg((\|(-\triangle_{\SS^2})^{a/2}  h\|_{L^2_{\gamma/2}}+\|  h\|_{H^a_{\gamma/2}})\\&&\quad\times
(\|(-\triangle_{\SS^2})^{b/2} f)\|_{L^2_{\gamma/2}}+ \|h\|_{H^{a_1}_{w_3}}\|f\|_{H^{b_1}_{w_4}} \bigg). \eeno
\end{enumerate}
\end{thm}

 \begin{thm}\label{thmlblandau} 
Suppose  $g$ is a non-negative and smooth function verifying \eqref{lbc} and let $\mathbf{A}=0,1$.   Then for  sufficiently small $\eta>0$, there exist  
 constants $\mathbf{C}_1(\delta, \lambda,\eta^{-1}), \mathbf{C}_2(\lambda, \delta), \mathbf{C}_3(\delta, \lambda, \eta^{-1}), \mathbf{C}_4(\lambda, \delta)$ and $ \mathbf{C}_5(\lambda, \delta)$  such that
 \begin{enumerate}
\item if $\gamma+2>0$,
\beno   \langle -Q_L(g,f),f\rangle_v&\gtrsim&\mathbf{A}\bigg[ \mathbf{C}_1(\delta,\lambda,\eta^{-1})\big(\| (-\triangle_{\SS^2})^{1/2}f \|_{L^2_{\gamma/2}}^2+
 \|  f \|_{H^1_{\gamma/2}}^2\big)\notag\\ && \quad- \eta \mathbf{C}_2(\delta,\lambda)\|  f\|_{L^2_{\gamma/2+1}}^2-\mathbf{C}_3(\delta,\lambda,\eta^{-1})\|f\|_{L^2_{\gamma/2}}^2\big]\\
&& + \mathbf{C}_4(\lambda, \delta)\|  f \|_{H^1_{\gamma/2}}^2-\mathbf{C}_5(\delta,\lambda)\|f\|_{L^2_{\gamma/2}}^2,\eeno
\item if $\gamma=-2$,
\beno   |\langle Q_L(g,f),f\rangle_v|&\gtrsim& \mathbf{A}\bigg[\mathbf{C}_1(\delta,\lambda,\eta^{-1})\big(\| (-\triangle_{\SS^2})^{1/2}f \|_{L^2_{\gamma/2}}^2+
 \|  f \|_{H^1_{\gamma/2}}^2\big)- \eta \mathbf{C}_2(\delta,\lambda) \|  f\|_{L^2_{\gamma/2+1}}^2\notag\\ && \quad- \exp\big\{\big(\mathbf{C}_3(\delta,\lambda,\eta^{-1})(1+w^{-1}\|g\|_{L^1_{w+2}})\big)^{\f{2+w}{w}}\big\}\|f\|_{L^2_{\gamma/2}}^2\bigg]\\
 \\&& + \mathbf{C}_4(\lambda, \delta)\|  f \|_{H^1_{\gamma/2}}^2-  \exp\big\{\big(\mathbf{C}_5(\delta,\lambda)(1+w^{-1}\|g\|_{L^1_{w+2}})\big)^{\f{2+w}{w}}\big\}\|f\|_{L^2_{\gamma/2}}^2,
\eeno  where $w>0$,
\item  if $-3<\gamma<-2$,  \beno |\langle Q_L(g,f),f\rangle_v|&\gtrsim&\mathbf{A} \bigg[\mathbf{C}_1(\delta,\lambda,\eta^{-1})\big(\| (-\triangle_{\SS^2})^{1/2}f \|_{L^2_{\gamma/2}}^2+
 \|  f \|_{H^1_{\gamma/2}}^2\big)- \eta \mathbf{C}_2(\delta,\lambda)\|  f\|_{L^2_{\gamma/2+1}}^2\notag\\ && - \mathbf{C}_3(\delta,\lambda,\eta^{-1}) (1+\|g \|_{L^p_{|\gamma|}}^{\f{(\gamma+5)p}{(\gamma+5)p-3}})\|f\|_{L^2_{\gamma/2}}^2\bigg]\\&&+ \mathbf{C}_4(\lambda, \delta)\|  f \|_{H^1_{\gamma/2}}^2-\mathbf{C}_5(\delta,\lambda)(1+\|g \|_{L^p_{|\gamma|}}^{\f{(\gamma+5)p}{(\gamma+5)p-3}})\|f\|_{L^2_{\gamma/2}}^2,\eeno
 with $p>\f3{\gamma+5}$.
\end{enumerate}
\end{thm}
\medskip

\begin{thm}\label{thmeplandau}
Suppose    $f$ is a non-negative  and smooth function verifying \eqref{lbc}. Then we have
  \beno  D_L(f)+\|f\|_{L^1_{w}}\gtrsim C(\lambda, \delta) \bigg(\|\sqrt{f}\|^2_{H^1_{\gamma/2}}+\|(-\triangle_{\SS^2})^{1/2}\sqrt{f}\|^2_{L^2_{\gamma/2}}\bigg), \eeno
where $w=\max\{\gamma+2,2\}$ and $D_L(f)=-\int_{\R^3} Q_L(f,f)\ln f dv$.
\end{thm}

\subsection{The new strategy: dyadic and geometric decompositions}  In this subsection, we will illustrate how to catch the structure of the Boltzmann collision operator to get the sharp bounds. Roughly speaking, the new strategy relies on the two types of the dyadic decomposition performed in both phase and frequency spaces and also the geometric structure of the elastic collision.

\subsubsection{Dyadic decompositions in phase and frequency spaces} 
 We first introduce two types of the dyadic decomposition. 
Let $B_{\frac{4}{3}}\eqdefa  \{\xi\in \mathrm{R}^3 ~|~ |\xi|\leq \frac{4}{3} \} $
and $ \emph{C}\eqdefa     \{\xi\in \mathrm{R}^3 ~|~ \frac{3}{4} \leq |\xi|\leq \frac{8}{3} \} $. Then one may introduce two radial functions $ \psi \in
C_0^{\infty}(B_{\frac{4}{3}})$ and $ \varphi \in C_0^{\infty}(\emph{C} )$  which satisfy
\begin{eqnarray}\label{defpsivarphi} 0\le\psi, \varphi \le0,\quad\mbox{and}\quad \psi(\xi) + \sum_{j\geq 0} \varphi(2^{-j} \xi) =1,~~\xi \in \mathrm{R}^3.  \end{eqnarray}

The first decomposition is performed in the phase space. We introduce the dyadic operator $ \mathcal{P}_j$  defined as
  \begin{eqnarray*} \mathcal{P}_{-1}f(x) =  \psi(x)f(x),~~~~
\mathcal{P}_{j}f(x) = \varphi(2^{-j}x)f(x)  ,~(j\geq 0). \end{eqnarray*}
 We also introduce operators $ \tilde{ \mathcal{P}}_{j}$ and $\mathcal{U}_{j}$  which are related to  $ \mathcal{P}_j$:
\beno \tilde{ \mathcal{P}}_{j}f(x)=\sum_{|k-j|\le N_0}\mathcal{P}_{k}f(x)  \quad\mbox{and}\quad   \mathcal{U}_{j}f(x) = \sum_{k\le j}\mathcal{P}_{k}f(x). \eeno
Here $N_0$ is a integer which will be chosen in the later.
Then     for any $f \in L^2(\mathbb{R}^3)$, it holds
\begin{eqnarray*} f = \mathcal{P}_{-1} f +\sum_{j\geq 0} \mathcal{P}_j f. \end{eqnarray*}

 We
set \ben\label{DefPhi} \Phi_k^\gamma(v)\eqdefa
\left\{\begin{aligned} & |v|^\gamma \varphi(2^{-k}|v|), \quad\mbox{if}\quad k\ge0;\\
& |v|^\gamma \psi( |v|),\quad\mbox{if}\quad k=-1.\end{aligned}\right.\een
 Then we derive that \beno
\langle Q(g, h), f \rangle_v=\sum_{k=-1}^\infty \langle Q_k(g, h), f \rangle_v=\sum_{k=-1}^\infty\sum_{j=-1}^\infty \langle Q_k(\mathcal{P}_jg, h), f \rangle_v,  \eeno
where
\beno
  Q_{k}(g, h)=\iint_{\sigma\in \SS^2,v_*\in \R^3} \Phi_k^\gamma(|v-v_*|)b(\cos\theta) (g'_*h'-g_*h)d\sigma dv_*. \eeno

\medskip

The second decomposition is performed in the frequency space. In fact, it is the standard Littlewood-Paley theory.
We denote $ \tilde{m}\overset {def} =\mathcal{F} ^{-1} \psi $ and $\phi \overset {def} = \mathcal{F}^{-1} \varphi $, where they are the inverse Fourier Transform of $\varphi$ and $\psi$.
If we set $\phi_j(x)\eqdefa2^{3j}\phi(2^{j}x)$, then the dyadic operators $ \mathfrak{F}_j$ can be defined as
follows \begin{eqnarray*} \mathfrak{F}_{-1}f(x) = \int_{\mathrm{R}^3}\tilde{m}(x-y) f(y)dy,~~~~
\mathfrak{F}_{j}f(x) =  \int_{\mathrm{R}^3} \phi_j(x-y) f(y)dy,~(j\geq 0).  \end{eqnarray*} We also introduce  operators $  \tilde{\mathfrak{F}}_{j}$ and $ \mathcal{S}_{j}$ which are related to  $ \mathfrak{F}_j$:
\beno    \tilde{\mathfrak{F}}_{j}f(x)=\sum_{|k-j|\le 3N_0}\mathfrak{F}_{k}f(x)   
&&\mbox{and}\quad \mathcal{S}_{j}f(x) = \sum_{k\le j}\mathfrak{F}_{k}f . \eeno
Then for any $f \in \mathbf{S}'(\mathbb{R}^3)$, it holds 
\begin{eqnarray*} f = \mathfrak{F}_{-1} f +\sum_{j\geq 0} \mathfrak{F}_j f. \end{eqnarray*}
It yields that 
\beno \langle Q_k(g, h), f \rangle_v & =&\sum_{p=-1}^\infty\sum_{l=-1}^\infty \langle Q_k(\mathfrak{F}_p g, \mathfrak{F}_l h), f \rangle_v.\eeno

Let us give some remarks on  these two decompositions:
\begin{enumerate}
\item The main purpose of the introduction of  the dyadic decomposition in the frequency space is to make full use of the interaction and the cancellation between the different parts of  frequency of functions $h$ and $f$. It will enable us to have the freedom of choosing  derivatives for  functions $h$ and $f$. However it is not enough to get the sharp bounds considering the fact  that  additional weight  is paid in the upper bound compared to that in the lower bound. Obviously it is due to the anisotropic structure of the operator.

\item To  clarify where the additional weight comes from, we introduce the dyadic decomposition in the phase space.  By careful analysis, we can distinguish which part  of the operator is the worst term that brings the additional weight to $h$ and $f$. In fact,  the worst situation happens in the case that  functions $h$ and $f$ are localized in the same region both in phase and frequency spaces and at the same time the relative velocity $|v-v_*|$ is far away from the zero.  In other words, in such a situation the collision operator is dominated by the anisotropic structure. It is the key point to obtain the estimate \eqref{sharpform} in anisotropic spaces.

\item These two  dyadic decompositions are consistent with the new profiles of the weighted Sobolev spaces(see Theorem \ref{baslem3} in Section 5).

\item The decompositions are very stable in the grazing collision limit. In fact, we can apply them to the rescaled Boltzmann operator to get the upper bounds. By taking the grazing collision limit, these upper bounds turn to be the sharp upper bounds of the Landau operator in weighted Sobolev spaces. The reader may check details in Section 4.
\end{enumerate}

\subsubsection{Key observation: geometric decomposition}  The key observation which enables us to get the new sharp bounds for the original collision operator is due to the geometric structure of the elastic collision. To explain it clearly, in what follows, we only focus on the maxwellian molecular case($\gamma=0$). 

We revisit the quantity $\mathcal{E}_g(f)$. In particular, we look for  a new decomposition for the term $f'-f$ contained in $\mathcal{E}_g(f)$.    Our main observation can be depicted schematically as follows:

\centerline{\includegraphics[width = 3in]{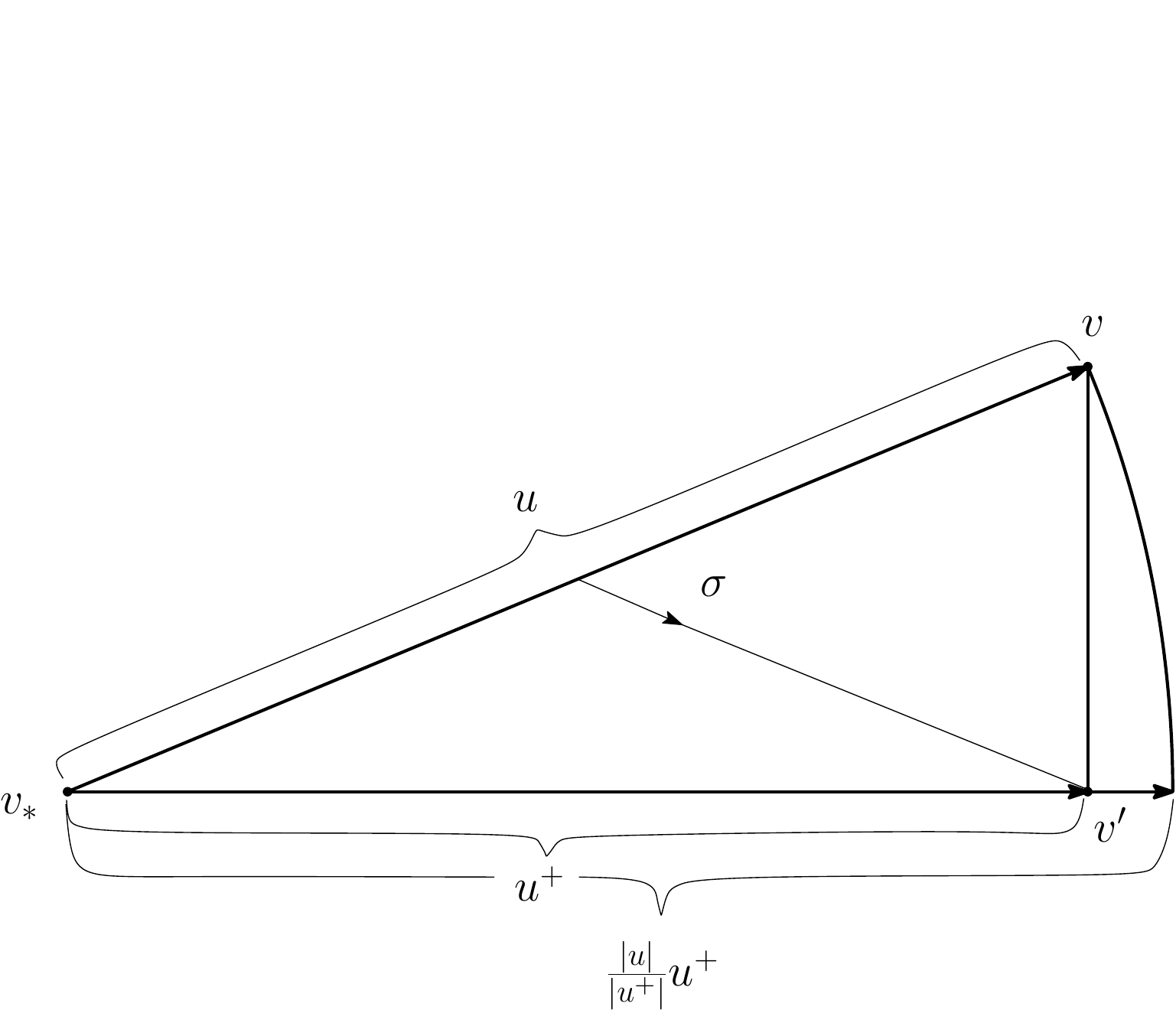}}

\medskip

We first note that $v'=v_*+u^+$ and $v=v_*+u$,
where $u=v-v_*$ and $u^+=\f{u+|u|\sigma}2$. Now assuming $u=r\tau$ with $r=|u|$ and $\tau\in \SS^2$, we obtain that
\beno v=v_*+r\tau, v'=v_*+r\f{\tau+\sigma}2. \eeno
Let $\varsigma=\f{\tau+\sigma}{|\tau+\sigma|}\in \SS^2$. Then we have the geometric decomposition:
\ben\label{geodecom} f(v')-f(v)&=&\big(f(v_*+\f{|\tau+\sigma|}2r\varsigma)-f(v_*+r\varsigma)\big)+\big(f(v_*+r\varsigma)-f(v_*+r\tau)\big)\nonumber\\
&=&\big(f(v_*+u^+)-f(v_*+|u|\f{u^+}{|u^+|})\big)+\big((T_{v_*}f)(r\varsigma)-(T_{v_*}f)(r\tau)\big), \een
where $T_v f\eqdefa f(v+\cdot)$.
Applying it to the functional $\langle Q(g,h), f\rangle$,  we  have 
\beno \langle Q(g,h), f\rangle&=&\underbrace{\iint_{\sigma\in \SS^2, u,v_*\in \R^3}   b(\cos\theta) g_*(T_{v_*}h)(u)\big(f(v_*+u^+)-f(v_*+|u|\f{u^+}{|u^+|})\big)d\sigma dv_*du}_{=\langle \mathcal{Q}_gh, f\rangle}\\&&+\underbrace{\iint_{\sigma\in \SS^2, u,v_*\in \R^3}   b(\cos\theta) g_*(T_{v_*}h)(u)\big((T_{v_*}f)(r\varsigma)-(T_{v_*}f)(r\tau)\big)d\sigma dv_*du}_{=\langle\mathcal{G}_g h, f\rangle_v}.\eeno

Notice that  $  |u^+-|u|\f{u^+}{|u^+|}|\sim \theta^2 |u|$. Then by technical argument(see the proof of Theorem \ref{thmlb}), the operator $\mathcal{Q}_g$ behaves like
\beno \mathcal{Q}_g \sim C_g\langle v\rangle^{s}(-\triangle_v)^{s/2}. \eeno Recalling the rough behavior of $Q(g, \cdot)$(see \eqref{behavqb}), we may regard $\mathcal{Q}_g$ as the lower order term. 

Now we concentrate on the functional $\langle\mathcal{G}_g h, f\rangle_v$. By \eqref{a4}, it is easy to check  
\beno  \langle\mathcal{G}_g h, f\rangle_v=\iint_{r>0,\sigma,\tau\in \SS^2,  v_*\in \R^3}  b(\sigma\cdot \tau)\mathrm{1}_{\sigma\cdot \tau\ge0} g_*(T_{v_*}h)(r\tau) \big((T_{v_*}f)(r\varsigma)-(T_{v_*}f)(r\tau))r^{2}d\sigma dv_*d\tau dr.\eeno
For fixed $v_*$, $\tau$ and $r$, if $\tau$ is chosen to be the polar direction, one has
\beno d\sigma= \sin\theta d\theta d\SS^1, d\varsigma=\sin \phi d\phi d\SS^1, \eeno
where $\theta=2\phi$. We deduce that
\beno d\sigma=4\cos\phi d\varsigma. \eeno
Then by change of  the variable from $\sigma$ to $\varsigma$, we derive that 
 \ben\label{geocha1} \langle\mathcal{G}_g h, f\rangle_v&= &\iint_{ r>0,\varsigma,\tau\in \SS^2,  v_*\in \R^3}  b\big(2(\varsigma\cdot \tau)^2-1\big)\mathrm{1}_{\varsigma\cdot \tau\ge\sqrt{2}/2} 4(\varsigma\cdot \tau) g_*(T_{v_*}h)(r\tau)\notag \\&&\times\big((T_{v_*}f)(r\varsigma)-(T_{v_*}f)(r\tau))r^{2}d\varsigma d\tau drdv_*.\een
 By the symmetric property of $\tau$ and $\varsigma$, we get
 \beno   \langle\mathcal{G}_g h, f\rangle_v&= &\f12\iint_{r>0,\varsigma,\tau\in \SS^2,  v_*\in \R^3}  b\big(2(\varsigma\cdot \tau)^2-1\big)\mathrm{1}_{\varsigma\cdot \tau\ge\sqrt{2}/2} 4(\varsigma\cdot \tau) g_*\big((T_{v_*}h)(r\tau)-(T_{v_*}h)(r\varsigma))\notag \\&&\times\big((T_{v_*}f)(r\varsigma)-(T_{v_*}f)(r\tau)\big)r^{2}d\varsigma d\tau drdv_*.\eeno
Thus to give the estimate  of $\mathcal{G}_{g}$, it suffices to consider the functional
\beno  \mathfrak{G}(h,f)\eqdefa\int_{\varsigma,\tau\in\SS^2}\big( h(\tau)-h(\varsigma)\big)\big( f(\varsigma)-f(\tau)\big)H(\varsigma\cdot \tau)d\varsigma d\tau,\eeno
 where $H(\varsigma\cdot \tau)=b\big(2(\varsigma\cdot \tau)^2-1\big)4(\varsigma\cdot \tau)\mathrm{1}_{\varsigma\cdot \tau\ge\sqrt{2}/2} $.
Recall that the assumption \eqref{a2} is equivalent to 
$b(\sigma\cdot\tau)\sim |\sigma-\tau|^{-2-2s}$, then by the fact $|\sigma-\tau|\sim |\varsigma-\tau|$, we   get  \beno  H(\varsigma\cdot \tau)\sim |\varsigma-\tau|^{-2-2s}\mathrm{1}_{|\varsigma-\tau|^2\le 2-\sqrt{2}}. \eeno

Let us first consider the lower bound of the operator. Thanks to Lemma \ref{antin1}, we get
\ben\label{geocha3}- \mathfrak{G}(f,f)+\|f\|_{L^2(\SS^2)}^2&\sim& \int_{\sigma,\tau\in \SS^2} \f{|f(\sigma)-f(\tau)|^2}{|\sigma-\tau|^{2+2s}}d\sigma d\tau+\|f\|_{L^2(\SS^2)}^2\\&&\sim \| (-\triangle_{\SS^2})^{s/2}f\|_{L^2(\SS^2)}+\|f\|_{L^2(\SS^2)}^2.\notag\een
 Then we have  
\beno  \langle-\mathcal{G}_g f, f\rangle_v\gtrsim \int_{\R^3} g_*\|(-\triangle_{\SS^2})^{s/2}T_{v_*}f\|_{L^2}^2dv_*-LOT,\eeno
which almost yields  
\beno \langle -Q(g, f), f\rangle_v\gtrsim C_g \|(-\triangle_{\SS^2})^{s/2}f\|_{L^2}^2-LOT.\eeno

Now what remains is to   prove 
\ben\label{labelint}  \|(-\triangle_{\SS^2})^{s/2}T_{v_*}f\|_{L^2}\lesssim  \langle v_*\rangle^s(\|(-\triangle_{\SS^2})^{s/2}f\|_{L^2}+\|f\|_{H^s}).\een  
It is easy to check that it holds for $s=0$ and $s=1$. Unfortunately because $-\triangle_{\SS^2}$ does not commutate with $\na$, the standard real interpolation method cannot be applied directly.  To solve the problem, we develop a new interpolation theory(which is  of independent interest) to overcome the difficulty. We refer readers to check the theory in Section 5. 
Then by combining the coercivity estimate in Sobolev spaces (see \eqref{boundWSP}), roughly speaking, we derive that
\beno \langle -Q(g, f), f\rangle_v\gtrsim C_g (\|(-\triangle_{\SS^2})^{s/2}f\|_{L^2}^2+\|f\|_{H^s}^2)-LOT.\eeno
We remark that the estimate is sharp since it is consistent with the behavior of the linearized operator \eqref{be5}.

Next we turn to the upper bounds.  We will show that our new observation can be used to prove \eqref{sharpform}. Indeed, by Cauchy-Schwartz inequality, we have
\ben\label{geocha2} |\mathfrak{G}(h,f)|\lesssim  \bigg(\int_{\sigma,\tau\in \SS^2} \f{|h(\sigma)-h(\tau)|^2}{|\sigma-\tau|^{2+2s}}d\sigma d\tau\bigg)^{\f12} \bigg(\int_{\sigma,\tau\in \SS^2} \f{|f(\sigma)-f(\tau)|^2}{|\sigma-\tau|^{2+2s}}d\sigma d\tau\bigg)^{\f12}.\een
Thanks to the Addition Theorem(see the statement in Section 5), we deduce that for constants $a, b\in\R$ with $a+b=2s$,
\beno |\mathfrak{G}(h,f)|\lesssim \|(1-\triangle_{\SS^2})^{a/2}h\|_{L^2} \|(1-\triangle_{\SS^2})^{b/2}f\|_{L^2},  \eeno
which implies
\beno  \langle-\mathcal{G}_g h, f\rangle_v\lesssim\bigg(\int_{\R^3} g_*\|(-\triangle_{\SS^2})^{a/2}T_{v_*}h\|_{L^2}^2dv_*\bigg)^{\f12}\bigg(\int_{\R^3} g_*\|(-\triangle_{\SS^2})^{b/2}T_{v_*}f\|_{L^2}^2dv_*\bigg)^{\f12}+LOT.\eeno
From which together with the sharp upper bounds  in weighted Sobolev spaces and \eqref{labelint}, we finally arrive at \eqref{sharpform}.

\medskip

Some remarks are in order:
\begin{enumerate}
\item The geometric decomposition plays essential role to catch the anisotropic structure of the operator in the lower and upper bounds. Since it does not use the symmetry and regularity   of the function $g$, it is more robust than the previous work. 

\item The geometric decomposition is stable in the process of the grazing collision limit. Actually we can give an explicit description of the asymptotic behavior of the anisotropic structure in the limit (see Lemma \ref{lbplim}). Roughly speaking, in the process of the limit, the behavior of collision operator depends on the parameter $\epsilon$.  If the eigenvalue of $-\triangle_{\SS^2}$ is less than $\epsilon^{-2}$, then the operator behaves like $-\triangle_{\SS^2}$. While if  the eigenvalue of $-\triangle_{\SS^2}$ is bigger than $\epsilon^{-2}$,
 then the operator behaves like $\epsilon^{2s-2}(-\triangle_{\SS^2})^{s}$. It reflects the strong connection between Boltzmann and Landau collision operators.
 
 \item  We remark that the geometric decomposition can also be applied to the operator in frequency space (recalling \eqref{bobylev}) to catch the anisotropic structure. It will give an alternative  proof to the lower and upper  bounds  of the operator in   anisotropic spaces. But it  only works well for the maxwellian case ($\gamma=0$) because  of the simplicity of  \eqref{bobylev} in this case. 
  \end{enumerate}

\subsection{Organization of the paper}   In Section 2, based on two types of the dyadic decomposition performed both in phase and frequency spaces, we give a complete proof to the sharp bounds of the collision operator in weighted Sobolev spaces.  

Based on the geometric decomposition, we  give a proof to the sharp  bounds of the collision operator in  anisotropic spaces in Section 3. 

In Section 4, we show that the strategy is so stable that we can extend all the estimates to the Landau collision operator by taking the grazing collision limit.  We also show the asymptotic behavior of the anisotropic structure of collision operator in the process of the grazing collision limit. 

In Section 5, we list some important lemmas which are of independent interest to the proof of the main theorems.  We first give some auxiliary lemmas on the new profiles  of the weighted Sobolev Spaces, new version of the interpolation theory and the basic properties of the real spherical harmonics. Then in the next we give the $L^2$  profile of  the fractional Laplace-Beltrami operator.  Finally we give the proof to \eqref{labelint}.  

At the end of the paper, we  give the conclusions and perspectives.
\setcounter{equation}{0}

\section{Upper bound for the Boltzmann collision operator in weighted Sobolev Spaces}

In this section, we will make full use of dyadic decompositions which are performed in both phase and  frequency spaces to give the precise estimates of the collision operator in  weighted Sobolev spaces.   Of course these estimates are not optimal. But they are still interesting. In fact, they have two advantages.   The first one is that we have the freedom of choosing  derivatives  and  weights for the functions compared to the previous work. The
second one is that these two decompositions used in the proof enable us to find out where the additional weight comes from in the upper bound.  It will be crucial to improve the upper bound of the operator in  anisotropic spaces.
\medskip

We first show how to reduce the estimate of the functional $\langle Q(g,h), f\rangle_v$ with the aid of the two types of  the decomposition.  For $ Q_{k}$, relative velocity $v-v_*$ is localized in the ring $\{\f34 2^k\le |v-v_*|\le \f83 2^k\}$.  Suppose that  $g$ is localized in the ring $\{ \f34 2^j\le |v_*|\le \f83 2^j\}$. Then thanks to the fact $\f{\sqrt{2}}2|v-v_*|\le |v'-v_*|\le |v-v_*|$, we have
\begin{itemize}
\item If $j\le k-N_0$, then  $|v|, |v'|\in [(\f34-\f832^{-N_0})2^k, \f83(1+2^{-N_0})2^k]$;
\item If $j\ge k+N_0$, then $|v|\in [(\f34-\f832^{-N_0})2^j, \f83(1+2^{-N_0})2^j]$ and $|v'|\in  [(\f{\sqrt{2}}2\f34-\f832^{-N_0})2^j, \f83(1+2^{-N_0})2^j]$;
\item If $|j-k|< N_0$, then $|v|\le 22^{k+N_0}, |v'|\le 22^{k+N_0}.$
\end{itemize}
From which together with the fact 
\beno  \langle Q_k(g,h), f\rangle_v=\iint_{\sigma\in\SS^2,v_*,v\in \R^3}\Phi_k^\gamma(|v-v_*|)b(\cos\theta) g_*h(f'-f)d\sigma dv_*dv ,\eeno we deduce that there exists a integer $N_0\in \N$ such that
 \ben\label{ubdecom} \langle Q(g,h), f\rangle_v&=&\sum_{k\ge-1}\sum_{j\ge-1}\langle Q_k(\mathcal{P}_jg, h), f \rangle_v\notag\\&=&\sum_{j\le k-N_0}\langle Q_k(\mathcal{P}_j g, \tilde{\mathcal{P}}_kh), \tilde{\mathcal{P}}_kf \rangle_v +
\sum_{j\ge k+N_0}\langle Q_k(\mathcal{P}_{j} g, \tilde{\mathcal{P}}_jh), \tilde{\mathcal{P}}_jf \rangle_v\notag\\&&\quad+\sum_{|j-k|\le N_0}\langle Q_k( \mathcal{P}_{j} g, \mathcal{U}_{k+N_0}h), \mathcal{U}_{k+N_0}f \rangle_v\\&=&\sum_{k\ge N_0-1}\langle Q_k(\mathcal{U}_{k-N_0} g, \tilde{\mathcal{P}}_kh), \tilde{\mathcal{P}}_kf \rangle_v +
\sum_{j\ge k+N_0}\langle Q_k(\mathcal{P}_{j} g, \tilde{\mathcal{P}}_jh), \tilde{\mathcal{P}}_jf \rangle_v\notag\\&&\quad+\sum_{|j-k|\le N_0}\langle Q_k( \mathcal{P}_{j} g, \mathcal{U}_{k+N_0}h), \mathcal{U}_{k+N_0}f \rangle_v.\notag \een

 We recall that   the Bobylev's formula of the operator can be stated as
\ben\label{bobylev}&& \qquad\langle\mathcal{F}\big( Q_k(g, h)\big), \mathcal{F}f \rangle\\&&=\iint_{\sigma\in \SS^2, \eta_*,\xi\in \R^3} b(\f{\xi}{|\xi|}\cdot \sigma)\big[ \mathcal{F}(\Phi_k^\gamma ) (\eta_*-\xi^{-})-\mathcal{F}(\Phi_k^\gamma)(\eta_*)\big](\mathcal{F}g)(\eta_*)(\mathcal{F}h)(\xi-\eta_*)\overline{(\mathcal{F}f)}(\xi)d\sigma d\eta_* d\xi,\nonumber \een
where $\mathcal{F}f$ denotes the Fourier transform of $f$ and $\xi^-=\f{\xi-|\xi|\sigma}2$.  Suppose that functions $g$ and $h$ are localized in rings $\{\f34 2^p\le |\xi|\le \f83 2^p\}$ and $\{\f34 2^l\le |\xi|\le \f83 2^l\}$ respectively in the frequency space.  Due to  \eqref{bobylev}, we have $\f34 2^p\le |\eta_*|\le \f83 2^p$ and $\f34 2^l\le |\xi-\eta_*|\le \f83 2^l$. Then \begin{itemize}
\item If $l\le p-N_0$, then $|\xi|\in [(\f34-\f832^{-N_0})2^p, \f83(1+2^{-N_0})2^p]$. Set $\xi^+\eqdefa \xi-\xi^-$, then it verifies $|\xi^+| \in [\f{\sqrt{2}}2(\f34-\f832^{-N_0})2^p, \f83(1+2^{-N_0})2^p]$. Notice that $|\eta_*-\xi^-|=|(\eta_*-\xi)+\xi^+|$, then one has $|\eta_*-\xi^-|\in [\f{\sqrt{2}}2(\f34-(1+\sqrt{2})\f832^{-N_0})2^p, \f83(1+22^{-N_0})2^p]  $.
\item If $l\ge p+N_0$, then $|\xi|\in [(\f34-\f832^{-N_0})2^l, \f83(1+2^{-N_0})2^l]$ and $|\eta_*-\xi^-|,|\eta_*|\le (1+2\f832^{-N_0})2^l$.
\item If $|l-p|< N_0$, then $|\xi|\le 2\f832^{p+N_0}$.  Now let $|\xi|\in[\f342^m,\f832^m]$. In the case of $|m-p|\le 2N_0$, one has $|\xi|\in[2^{-2N_0}\f342^p,2^{2N_0}\f832^p]$. While in the case of $m<p-2N_0$, one has $|\eta_*|, |\eta_*-\xi^-|\in[(2^{-N_0}\f34-\f832^{-2N_0})2^p, (\f832^{2N_0}+\f832^{-2N_0})2^p]$.
\end{itemize}
Then we have
\beno \langle Q_k(g, h), f \rangle_v & =&\sum_{p,l\ge-1} \langle Q_k(\mathfrak{F}_p g, \mathfrak{F}_l h), f \rangle_v \\
 &=& \sum_{l\le p-N_0} \mathfrak{W}_{k,p,l}^1  +\sum_{l\ge p+N_0} \mathfrak{W}_{k,p,l}^2+\sum_{|l-p|< N_0}\big(\sum_{|m-p|\le 2N_0} \mathfrak{W}_{k,p,l,m}^3+\sum_{m<p-2N_0} \mathfrak{W}_{k,p,l,m}^4\big), \eeno
where
\beno   \mathfrak{W}_{k,p,l}^1&\eqdefa&\iint_{\sigma\in \SS^2,v_*,v\in \R^3} \big(\tilde{ \mathfrak{F}}_{p}\Phi_k^\gamma\big)(|v-v_*|)b(\cos\theta) (\mathfrak{F}_{p}g)_*(\mathfrak{F}_{l}h)\big[
(\tilde{ \mathfrak{F}}_{p}f)'-
\tilde{ \mathfrak{F}}_{p}f\big]d\sigma dv_* dv,\\
\mathfrak{W}_{k,p,l}^2&\eqdefa&\iint_{\sigma\in \SS^2,v_*,v\in \R^3} \Phi_k^\gamma(|v-v_*|)b(\cos\theta)  (\mathfrak{F}_{p}g)_*(\mathfrak{F}_{l}h)\big[
(\tilde{ \mathfrak{F}}_{l}f)'-
\tilde{ \mathfrak{F}}_{l}f\big]d\sigma dv_* dv,\\
\mathfrak{W}_{k,p,l,m}^3&\eqdefa&\iint_{\sigma\in \SS^2,v_*,v\in \R^3} \Phi_k^\gamma(|v-v_*|)b(\cos\theta)  (\mathfrak{F}_{p}g)_*(\mathfrak{F}_{l}h)\big[
(  \mathfrak{F}_{m}f)'-
 \mathfrak{F}_{m}f\big]d\sigma dv_* dv,\\
\mathfrak{W}_{k,p,l,m}^4&\eqdefa&\iint_{\sigma\in \SS^2,v_*,v\in \R^3}\big(\tilde{ \mathfrak{F}}_{p}\Phi_k^\gamma\big)(|v-v_*|)b(\cos\theta) (\mathfrak{F}_{p}g)_*(\mathfrak{F}_{l}h)\big[
( \mathfrak{F}_{m}f)'-
 \mathfrak{F}_{m}f\big]d\sigma dv_* dv.
\eeno
We remark that  we use the fact that the Fourier transform maps a radial function into a radial function.

By simple calculation, we arrive at
 \ben\label{fsdecom} &&\langle Q_k(g, h), f \rangle_v  \notag\\
&&= \sum_{l\le p-N_0} \mathfrak{W}_{k,p,l}^1  +\sum_{l\ge-1 } \mathfrak{W}_{k,l}^2+\sum_{p\ge-1}  \mathfrak{W}_{k,p}^3+\sum_{m<p-N_0} \mathfrak{W}_{k,p,m}^4 , \een
where \beno
\mathfrak{W}_{k,l}^2&\eqdefa&\iint_{\sigma\in \SS^2,v_*,v\in \R^3} \Phi_k^\gamma(|v-v_*|)b(\cos\theta)  (\mathcal{S}_{l-N_0}g)_*(\mathfrak{F}_{l}h)\big[
(\tilde{ \mathfrak{F}}_{l}f)'-
\tilde{ \mathfrak{F}}_{l}f\big]d\sigma dv_* dv,\\
\mathfrak{W}_{k,p}^3&\eqdefa&\iint_{\sigma\in \SS^2,v_*,v\in \R^3} \Phi_k^\gamma(|v-v_*|)b(\cos\theta)  (\mathfrak{F}_{p}g)_*(\tilde{\mathfrak{F}}_{p}h)\big[
(  \tilde{\mathfrak{F}}_{p}f)'-
 \tilde{\mathfrak{F}}_{p}f\big]d\sigma dv_* dv,\\
\mathfrak{W}_{k,p,m}^4&\eqdefa&\iint_{\sigma\in \SS^2,v_*,v\in \R^3} \big(\tilde{ \mathfrak{F}}_{p}\Phi_k^\gamma\big)(|v-v_*|)b(\cos\theta)  (\mathfrak{F}_{p}g)_*(\tilde{\mathfrak{F}}_{p}h)\big[
(   \mathfrak{F}_{m}f)'-
 \mathfrak{F}_{m}f\big]d\sigma dv_* dv.
\eeno
Now the estimate of the functional $\langle Q(g,h), f\rangle_v$ is reduced to the estimates of the terms in the righthand sides of \eqref{ubdecom} and \eqref{fsdecom}.
\medskip

 \medskip

We begin with two useful propositions.
\begin{prop}\label{baslem5} Suppose $\varpi\in (0,1]$. Then  for $|y|\neq 0$,  one has \beno  \big| |x|^{2\varpi}-|y|^{2\varpi}\big|\lesssim \left\{\begin{aligned} & |x-y| |y|^{2\varpi-1}, \quad\mbox{if}\quad 0<\varpi\le\f12;\\
&  |x-y| |y|^{2\varpi-1}+|x-y|^{2\varpi},\quad\mbox{if}\quad \f12< \varpi\le1.\end{aligned}\right.\eeno
\end{prop}
\begin{proof} We first treat the case $2\varpi\le1$.  Suppose that $|x|\le b|y|$ with $0<b<1$. Then one has \beno (1-b)|y|\le|x-y|\le (1+b)|y|,\eeno which implies
\beno  \big| |x|^{2\varpi}-|y|^{2\varpi}\big|\lesssim |y|^{2\varpi}\lesssim |x-y||y|^{2\varpi-1}.\eeno
Next we handle the case $|x|> b|y|$.  If $\theta\in[0,1]$, we have \beno (1-\theta)|x|+\theta |y|\ge [(1-\theta)b+\theta]|y|.  \eeno
From which together with the fact
\ben\label{meva} \big| |x|^{2\varpi}-|y|^{2\varpi}\big|\lesssim |x-y|\int_0^1 [(1-\theta)|x|+\theta |y|]^{2\varpi-1}d\theta,   \een  we get the desired result.

When $2\varpi>1$, the proposition is easily followed from \eqref{meva} and the fact \beno (1-\theta)|x|+\theta |y|\le |y|+|x-y|.  \eeno \end{proof}

\begin{prop}\label{baslem6}  Suppose $\varpi\in (0,1]$ and $N\in \N$. Recall that   $\Phi_k^\gamma$ is defined in \eqref{DefPhi}.
\begin{enumerate}
\item  Set $A_k^{\varpi}(v)\eqdefa (\tilde{ \mathfrak{F}}_{p}\Phi_k^\gamma\big)(v)|v|^{2\varpi}.$  Then if $k\ge0$, we have
 \beno
\|A_{k}^{\varpi}\|_{L^\infty}&\lesssim& 2^{k(\gamma+\f32-N)}2^{-pN}\|\Phi_0^\gamma\|_{H^{N+2}} \|\varphi\|_{W^{2,\infty}_N}.
\eeno If $\gamma=0$ and $k= -1$,
\beno  \|A_{-1}^{\varpi}\|_{L^\infty}\lesssim 2^{-pN}\|\psi\|_{H^{N+2}} \|\varphi\|_{W^{2,\infty}_N}.\eeno

\item  Set $B_{-1}^{\varpi}(v)\eqdefa (\tilde{ \mathfrak{F}}_{p}\Phi_{-1}^\gamma\big)(v)|v|^{2\varpi}-(\tilde{ \mathfrak{F}}_{p}\Phi_{-1}^{\gamma+2\varpi}\big)(v) .$  Then if $\gamma+2\varpi>0$,  there exists a   constant  $\eta_1=\min\{\gamma+2\varpi, 1\}$   such that
\beno  |B_{-1}^{\varpi}|\le |B_1|+|B_2|,\eeno
where \beno  \|B_1\|_{L^2}\le 2^{-(\eta_1+\f32)p}\quad\mbox{and}\quad
\|B_2\|_{L^\infty}\lesssim 2^{-\eta_1 p}.\eeno
If $\gamma+2\varpi >-1$, then there exists a constant $\eta_2$ such that
\beno  \|B_{-1}^{\varpi}\|_{L^2}\lesssim 2^{-(\eta_2+\f12) p}, \eeno
where
\beno  \eta_2= \left\{\begin{aligned} & \f12 , \quad\mbox{if}\quad \gamma+2\varpi>-\f12;\\
&  \f12-(\log_2 p)/p,\quad\mbox{if}\quad \gamma+2\varpi=-\f12;\\
&\gamma+2\varpi+1,\quad\mbox{if}\quad -1<\gamma+2\varpi<-\f12. \end{aligned}\right.\eeno
 \end{enumerate}
\end{prop}

\begin{proof} (i). For $k\ge0$,  we recall that $\Phi_k^\gamma(v)=|v|^\gamma \varphi(2^{-k}v)$. Then by
 direct calculation, we have  \ben\label{Phi3} \mathcal{F}(\Phi_k^\gamma)(\xi)=2^{(\gamma+3)k}\mathcal{F}(\Phi_0^{\gamma})(2^k\xi),\een
which yields \beno  \|\tilde{ \mathfrak{F}}_{p}\Phi_k^\gamma\|_{L^\infty_{2}}&\lesssim&  \|(-\triangle)\mathcal{F}\big( \tilde{ \mathfrak{F}}_{p}\Phi_k^\gamma\big)\|_{L^1}+\|\mathcal{F}(\tilde{ \mathfrak{F}}_{p}\Phi_k^\gamma)\|_{L^1}\lesssim \| \mathcal{F}(\tilde{ \mathfrak{F}}_{p}\Phi_k^\gamma)\|_{H^2_2}\\
&\lesssim& 2^{k(\gamma+\f32-N)}2^{-pN}\|\Phi_0^\gamma\|_{H^{N+2}} \|\varphi\|_{W^{2,\infty}_N}.
  \eeno
From which, we obtain that
\beno  \|A_k^{\varpi}\|_{L^\infty}\lesssim2^{k(\gamma+\f32-N)}2^{-pN}\|\Phi_0^\gamma\|_{H^{N+2}} \|\varphi\|_{W^{2,\infty}_N}
.\eeno

If $\gamma=0$ and $k=-1$, by the definition, we have
\beno  \|A_{-1}^{\varpi}\|_{L^\infty}\lesssim \|(-\triangle)\mathcal{F}\big(\tilde{ \mathfrak{F}}_{p}\psi\big)\|_{L^1}+\|\mathcal{F}( \mathfrak{F}_{p}\psi)\|_{L^1}\lesssim 2^{-pN}\|\psi\|_{H^{N+2}} \|\varphi\|_{W^{2,\infty}_N}. \eeno

(ii).  Let $\tilde{\phi}_p=\sum_{|m-p|\le N_0}  \phi_m$. Then by the definition of $\tilde{ \mathfrak{F}}_{p}$, we have  
\beno |B_{-1}^{\varpi}(v)|&=&\bigg|\int_{\R^3} \tilde{\phi}_p(v-y)\Phi_{-1}^\gamma(y)(|v|^{2\varpi}-|y|^{2\varpi})dy\bigg|. \eeno
Thanks to Proposition \ref{baslem5},  it can be reduced to bound the terms $B_{-1}^1$ and $B_{-1}^2$ which are defined by
\beno  B_{-1}^1&\eqdefa&\int_{\R^3} |\tilde{\phi}_p(v-y)|\Phi_{-1}^{\gamma+2\varpi-1}(y)|v-y|dy, \\
\mbox{and} \quad B_{-1}^2&\eqdefa&\int_{\R^3} |\tilde{\phi}_p(v-y)|\Phi_{-1}^\gamma(y)|v-y|^{2\varpi}dy.\eeno
We remind the reader that the term $B_{-1}^2$ is only needed to be considered in the case of  $2\varpi> 1$.

 We begin with the estimate of $B^{1}_{-1}$. Observe that
\beno  B^{1}_{-1}&=&\int_{\R^3} |\tilde{\phi}_p(v-y)|\Phi_{-1}^{\gamma+2\varpi-1}(y)\mathrm{1}_{|y|\le 2^{-p}}|v-y|dy+\int_{\R^3} |\tilde{\phi}_p(v-y)|\Phi_{-1}^{\gamma+2\varpi-1}(y)\mathrm{1}_{|y|\ge 2^{-p}}|v-y|dy\\&\eqdefa& B_{-1}^{1,1}+B_{-1}^{1,2}.
\eeno
It is easy to check \beno  \|B_{-1}^{1,1}\|_{L^2}\lesssim \bigg( \int_{\R^3} |\tilde{\phi}_p|^2|y|^2 dy\bigg)^{1/2}2^{-(\gamma+2\varpi-1+3)p}\lesssim 2^{-(\gamma+2\varpi+\f32)p}.\eeno We turn to the estimate of $B_{-1}^{1,2}$. To make full use of the structure, we separate the estimate into several cases.
\begin{itemize}
\item Case 1: $\gamma+2\varpi>0$.
Then it holds
\beno  \|B_{-1}^{1,2}\|_{L^\infty}\lesssim   2^{-\eta p}, \eeno
where $\eta=\min \{1, \gamma+2\varpi\}$. On the other hand, by Young inequality, we also have
\beno  \|B_{-1}^{1,2}\|_{L^2}\lesssim   2^{-p}.  \eeno
\item Case 2: $-\f12\le\gamma+2\varpi\le0$.  Then by Young inequality, one has that if $\gamma+2s>-\f12$,
\beno \|B_{-1}^{1,2}\|_{L^2}&\lesssim& 2^{-p}. \eeno  While if $\gamma+2s=-\f12$, we have
\beno \|B_{-1}^{1,2}\|_{L^2}&\lesssim& 2^{-p}p\lesssim 2^{-p(1-(\log_2p)/p)}.\eeno
\item Case 3: $-1<\gamma+2\varpi<-\f12$. We have
\beno \|B_{-1}^{1,2}\|_{L^2}&\lesssim& 2^{-p}\bigg(\int_{2^{-p}\le |y|\le 2} |y|^{2(\gamma+2\varpi-1)}dy\bigg)^{\f12}\\&\lesssim& 2^{-(\gamma+2\varpi+\f32)p}.\eeno
 \end{itemize}

The similar argument can be applied to $B^2_{-1}$ with $2\varpi> 1$. Notice that
\beno   B^{2}_{-1}&=&\int_{\R^3} |\tilde{\phi}_p(v-y)|\Phi_{-1}^{\gamma}(y)\mathrm{1}_{|y|\le 2^{-p}}|v-y|^{2\varpi}dy\\&&+\int_{\R^3} |\tilde{\phi}_p(v-y)|\Phi_{-1}^{\gamma}(y)\mathrm{1}_{|y|\ge 2^{-p}}|v-y|^{2\varpi}dy\\&\eqdefa& B_{-1}^{2,1}+B_{-1}^{2,2}.
\eeno
It is easy to check \beno  \|B_{-1}^{2,1}\|_{L^2}\lesssim  \bigg( \int |\tilde{\phi}_p|^2|y|^{4\varpi} dy\bigg)^{1/2}2^{-(\gamma+3)p}\lesssim 2^{-(\gamma+2\varpi+\f32)p}.\eeno  Similarly,   the estimate of $B_{-1}^{2,2}$ falls in several cases.
\begin{itemize}
\item Case 1: $\gamma\ge0$.
Then it holds
\beno  \|B_{-1}^{2,2}\|_{L^\infty}\lesssim  \int_{\R^3} |\tilde{\phi}_p||y|^{2\varpi} dy \lesssim 2^{-2\varpi p}. \eeno
\item Case 2: $\gamma+2\varpi>0$ and $\gamma<0$.  Use the fact $|\Phi_{-1}^{\gamma}(y)\mathrm{1}_{|y|\ge 2^{-p}}|\le 2^{-\gamma p}$ , then one has
\beno \|B_{-1}^{2,2}\|_{L^\infty}&\lesssim&2^{-\gamma p}  \int_{\R^3} |\tilde{\phi}_p||y|^{2\varpi} dy\lesssim  2^{-(\gamma+2\varpi)p}. \eeno  
\item Case 3:  $ \gamma\ge -\f32$. By Young inequality, we have
\beno \|B_{-1}^{2,2}\|_{L^2}&\lesssim&\bigg( \int_{\R^3} |\Phi_{-1}^\gamma|^2\mathrm{1}_{|y|\ge 2^{-p}} dy\bigg)^{1/2}\big(\int_{\R^3} |\tilde{\phi}_p||y|^{2\varpi} dy\big) \lesssim  2^{-(\eta+\f12)p},\eeno
where $\eta=2\varpi-\f12$ if $\gamma>-\f32$ and $\eta=2\varpi-\f12-(\log_2p)/p$ if $\gamma=-\f32$.
\item Case 4:  $\gamma<-\f32$.  Again by Young inequality, we have
\beno \|B_{-1}^{2,2}\|_{L^2}&\lesssim& \bigg( \int_{\R^3} |\Phi_{-1}^\gamma|^2\mathrm{1}_{|y|\ge 2^{-p}} dy\bigg)^{1/2}\big(\int_{\R^3} |\tilde{\phi}_p||y|^{2\varpi} dy\big) \lesssim 2^{-(\gamma+2\varpi+\f32)p}.\eeno
\end{itemize}
Notice that there are only two types of the estimates,  $L^2$ and $L^\infty$ estimates, in the proof, then the proposition is easily obtained by patching together all the estimates.
\end{proof}
\medskip

\subsection{Estimates of $\mathfrak{W}_{k,p,l}^1$ and $\mathfrak{W}_{k,p,m}^4$  defined in \eqref{fsdecom}}
We first give the estimate to $\mathfrak{W}_{k,p,l}^1$.
\begin{lem}\label{lemub1}   Suppose $N\in \N$. For $k\ge0$, it holds
\beno |\mathfrak{W}_{k,p,l}^1|&\lesssim&2^{k(\gamma+\f52-N)}(2^{-p(N-2s)}2^{2s(l-p)}+2^{-(N-\f52)p}2^{\f32(l-p)})\\&&\times\|\Phi_0^\gamma\|_{H^{N+2}}\|\varphi\|_{W^{2,\infty}_N} \|\mathfrak{F}_{p}g\|_{L^1}\|\mathfrak{F}_{l}h\|_{L^2}\|\tilde{ \mathfrak{F}}_{p}f\|_{L^2}.\eeno
If $k=-1$, we have
\begin{enumerate}
\item if $\gamma=0$, \beno |\mathfrak{W}_{-1,p,l}^1|&\lesssim& \big(2^{-p(N-2s)}2^{2s(l-p)} +2^{-(N-1)p}2^{\f32 l} \big)\\&&\times\|\psi\|_{H^{N+2}}\|\varphi\|_{W^{2,\infty}_N}\|\mathfrak{F}_{p}g\|_{L^1}\|\mathfrak{F}_{l}h\|_{L^2}\|\tilde{ \mathfrak{F}}_{p}f\|_{L^2},\eeno
\item if $\gamma+2s>0$ and $\gamma>-\f32$, \beno |\mathfrak{W}_{-1,p,l}^1|\lesssim
  2^{\tilde{\eta}(l-p)}2^{(2s-\eta_1)l}\|\mathfrak{F}_{p}g\|_{L^1}\|\mathfrak{F}_{l}h\|_{L^2}\|\tilde{ \mathfrak{F}}_{p}f\|_{L^2}, \eeno
   \item  if  $ \gamma+2s>-1$ and $\gamma>-\f52$, \beno |\mathfrak{W}_{-1,p,l}^1|\lesssim \ 2^{\tilde{\eta}(l-p)}2^{(2s-\eta_2)l} \|\mathfrak{F}_{p}g\|_{L^{\f32}} \|\mathfrak{F}_{l}h\|_{L^2}\|\tilde{ \mathfrak{F}}_{p}f\|_{L^2},  \eeno
 \item  if  $ \gamma+2s>-1$, \beno |\mathfrak{W}_{-1,p,l}^1|&\lesssim&
( 2^{2sl}2^{-(\f12+\eta_2)p}+2^{\f32l}2^{-(\gamma+3)p})\|\mathfrak{F}_{p}g\|_{L^{2}} \|\mathfrak{F}_{l}h\|_{L^2}\|\tilde{ \mathfrak{F}}_{p}f\|_{L^2}\\
&\lesssim&2^{\tilde{\eta}(l-p)}2^{(2s-\f12-\eta_2)l} \|\mathfrak{F}_{p}g\|_{L^{2}} \|\mathfrak{F}_{l}h\|_{L^2}\|\tilde{ \mathfrak{F}}_{p}f\|_{L^2},  \eeno
\end{enumerate}
where $\tilde{\eta}$ is a positive constant which depends only on $\gamma$ and $s$ and varies for different cases. We remark that  constants $\eta_1$ and $\eta_2$ are stated in Proposition \ref{baslem6} and functions $\psi$ and $\varphi$ are defined in \eqref{defpsivarphi}.
\end{lem}
\begin{proof}  We recall that $l$ and $p$  verify the condition $l<p-N_0$. To make full use of the cancellation   and to handle  the singularity caused by the angular function,    we make the following decomposition:
\beno    \mathfrak{W}_{k,p,l}^1= \mathfrak{D}_{k}^1+\mathfrak{D}_k^2,
\eeno where \beno
\mathfrak{D}_{k}^1&\eqdefa&\iint_{\sigma\in \SS^2,v_*,v\in \R^3} \big(\tilde{ \mathfrak{F}}_{p}\Phi_k^\gamma\big)(|v-v_*|)b(\cos\theta) (\mathfrak{F}_{p}g)_*\big[(\mathfrak{F}_{l}h)-
(\mathfrak{F}_{l}h)'\big](\tilde{ \mathfrak{F}}_{p}f)'d\sigma dv_* dv,\\
\mathfrak{D}_{k}^2&\eqdefa&\iint_{\sigma\in \SS^2,v_*,v\in \R^3} \big(\tilde{ \mathfrak{F}}_{p}\Phi_k^\gamma\big)(|v-v_*|)b(\cos\theta) (\mathfrak{F}_{p}g)_*\big[(\mathfrak{F}_{l}h)'
(\tilde{ \mathfrak{F}}_{p}f)'-(\mathfrak{F}_{l}h)
\tilde{ \mathfrak{F}}_{p}f\big]d\sigma dv_* dv. \eeno
The proof falls in several steps.

\medskip

\noindent{\it Step 1: Estimate of  $\mathfrak{D}_k^1$.}  Observe the facts
\ben\label{taylor1}  &&(\mathfrak{F}_{l}h)(v)-
(\mathfrak{F}_{l}h)(v')\nonumber\\&&=(v-v')\cdot(\na \mathfrak{F}_lh)(v')+\f12\int_0^1(1-\kappa)(v-v')\otimes(v-v'): (\na^2\mathfrak{F}_lh)(\kappa(v))d\kappa,\een
where $\kappa(v)=v'+\kappa(v-v'),$
 and  \ben\label{vequ}  &&\iint_{\sigma\in\SS^2,v\in\R^3} \Gamma(|v-v_*|) b(\f{v-v_*}{|v-v_*|}\cdot\sigma)w(|v'-v|)(v-v')\rho(v')d\sigma dv\notag\\&&=4\iint_{\sigma\in\SS^2,v\in\R^3} \Gamma(|T_\sigma(v')-v_*|) b(\f{T_\sigma(v')-v_*}{|T_\sigma(v')-v_*|}\cdot\sigma)w(|v'-T_\sigma(v')|)\f{T_\sigma(v')-v'}{ \big(\f{v'-v_*}{|v'-v_*|}\cdot\sigma\big)^2}\rho(v')d\sigma dv'=0,\een
where $w$ and $\rho$ are smooth functions and $T_\sigma$ represents the  transform  such that $ T_\sigma(v')=v$. We refer   readers to  \cite{advw} or \cite{av1} to check the change of the variable from $v$ to $v'$.
Then in order to get the optimal estimate,  we follow  the idea in \cite{chenhe1} to introduce the function $\psi$ (defined in \eqref{defpsivarphi}) to decompose $\mathfrak{D}_k^1$ into  angular cutoff part  and  angular non cutoff part, that is,
\beno  \mathfrak{D}_k^1=\mathfrak{D}_k^{1,1}+\mathfrak{D}_k^{1,2},\eeno
where
\beno \mathfrak{D}_k^{1,1}
&\eqdefa& \f12\int_0^1 \iint_{\sigma\in \SS^2,v_*,v\in \R^3} A_k^s(|v-v_*|)b(\cos\theta) \psi(2^l(v'-v)) (\mathfrak{F}_{p}g)_*\\&&\qquad\qquad\times |v-v_*|^{-2s}\big((v-v')\otimes(v-v'): (\na^2 \fF_l h)(\kappa(v))\big)(\tilde{ \mathfrak{F}}_{p}f)'d\sigma dv_* dvd\kappa,\\
 \mathfrak{D}_k^{1,2}&\eqdefa&\iint_{\sigma\in \SS^2,v_*,v\in \R^3} A_k^s(|v-v_*|)b(\cos\theta)(1-\psi(2^l(v'-v))) (\mathfrak{F}_{p}g)_*\\&&\qquad\qquad\times |v-v_*|^{-2s}\big[(\mathfrak{F}_{l}h)-
(\mathfrak{F}_{l}h)'\big](\tilde{ \mathfrak{F}}_{p}f)'d\sigma dv_* dv,\eeno
where  \eqref{taylor1} and \eqref{vequ} are used and  $A_k^s$ is defined in   Proposition \ref{baslem6}.

\medskip

\noindent{\it Step 1.1: Estimate of $\mathfrak{D}_k^{1,1}$.} 
We divide the estimate into three cases. 

{\it Case 1: $k\ge0$.}
By Cauchy-Schwartz inequality, one has
\beno |\mathfrak{D}_k^{1,1}|
&\lesssim& \|A^s_k\|_{L^\infty} \bigg(\int_0^1 \iint_{\sigma\in \SS^2,v_*,v\in \R^3}   |(\mathfrak{F}_{p}g)_*|b(\cos\theta)\sin^2(\theta/2)|\psi(2^l(v'-v))|\\&&\times|v-v_*|^{2-2s}|(\tilde{ \mathfrak{F}}_{p}f)'|^2d\sigma dv_* dvd\kappa\bigg)^\f12\bigg(\int_0^1 \iint_{\sigma\in \SS^2,v_*,v\in \R^3} |(\mathfrak{F}_{p}g)_*|b(\cos\theta)\\&&\times\sin^2(\theta/2)|\psi(2^l(v'-v))||v-v_*|^{2-2s}\big|(\na^2 \fF_l h)(\kappa(v))\big|^2d\sigma dv_* dvd\kappa\bigg)^\f12,
\eeno
where we use the fact  $|v-v'|=|v-v_*|\sin(\theta/2)$.

Then we follow the change of variables: $(v_*,v)\rightarrow (v_*,u_1=v')$ and $(v_*,v)\rightarrow \big(v_*,u_2=\kappa(v)\big)$.
 Thanks to the fact
\begin{eqnarray}\label{Jacobi} |\frac{\partial u_2}{\partial v}| =(1-\frac{\kappa}{2})^2 \Big\{ (1-\frac{\kappa}{2})+\frac{\kappa}{2} \f{v-v_*}{|v-v_*|}\cdot\sigma\Big\},
\end{eqnarray}
  we derive that
\beno
|\mathfrak{D}_k^{1,1}|
&\lesssim&  \|A^s_k\|_{L^\infty}\bigg(\int_0^1 \iint_{\sigma\in \SS^2,v_*,u_1\in \R^3}  |(\mathfrak{F}_{p}g)_*|b(\cos\theta)\sin^2(\theta/2)|\psi(2^l|v-v_*|\sin(\theta/2))|\\&&\times|u_1-v_*|^{2-2s}|(\tilde{ \mathfrak{F}}_{p}f)(u_1)|^2d\sigma dv_*du_1d\kappa\bigg)^\f12\bigg(\int_0^1 \iint_{\sigma\in \SS^2,v_*,u_2\in \R^3} |(\mathfrak{F}_{p}g)_*|b(\cos\theta)\\&&\times\sin^2(\theta/2)|\psi(2^l|v-v_*|\sin(\theta/2))||u_2-v_*|^{2-2s}\big|(\na^2 \fF_l h)(u_2))\big|^2d\sigma dv_* du_2d\kappa\bigg)^\f12,
\eeno
where we use the fact $|v-v_*|\sim |u_1-v_*|\sim |u_2-v_*|$.
It is not difficult to check
\ben\label{angular2}  &&|v-v_*|^{2-2s}\int_{\sigma\in\SS^2} b(\cos\theta)\sin^2(\theta/2) \psi(2^{l}|v-v_*|\sin(\theta/2)) d\sigma\notag\\
&&\lesssim  |u_2-v_*|^{2-2s}\int_{0}^{\f832^{-l} |u_2-v_*|^{-1}} b(\cos\theta)\theta^2 \sin\tilde{\theta}  d\tilde{\theta}  \notag\\
&&\lesssim 2^{-(2-2s)l},\een
where   $\tilde{\theta}$ verifies $\cos\tilde{\theta}=\f{u_2-v_*}{|u_2-v_*|}\cdot \sigma$ and $\theta/2\le\tilde{\theta}\le \theta$.
By Bernstein inequalities (\ref{bern1}-\ref{bern2}), it is easy to derive that  \beno  |\mathfrak{D}_k^{1,1}| &\lesssim& 2^{2sl}\|A_k\|_{L^\infty}
 \|\mathfrak{F}_{p}g\|_{L^1}\|\fF_l h\|_{L^2}\|\tilde{ \mathfrak{F}}_{p}f\|_{L^2}\\
 &\lesssim& 2^{k(\gamma+\f32-N)}2^{-p(N-2s)}2^{2s(l-p)}\|\Phi_0^\gamma\|_{H^{N+2}}\|\varphi\|_{W^{2,\infty}_N}\|\mathfrak{F}_{p}g\|_{L^1}\|\fF_l h\|_{L^2}\|\tilde{ \mathfrak{F}}_{p}f\|_{L^2}.\eeno

{\it Case 2: $k=-1$ and $\gamma=0$.}  Thanks to Proposition \ref{baslem6}, we easily get
\beno |\mathfrak{D}_{-1}^{1,1}|&\lesssim& 2^{-p(N-2s)}2^{2s(l-p)}\|\psi\|_{H^{N+2}}\|\varphi\|_{W^{2,\infty}_N}\|\mathfrak{F}_{p}g\|_{L^1}\|\fF_l h\|_{L^2}\|\tilde{ \mathfrak{F}}_{p}f\|_{L^2}.\eeno

{\it Case 3: $k=-1$ and $\gamma\neq0$.} For the general case, following the decomposition in Proposition \ref{baslem6}:
\beno  A^s_{-1}=B^s_{-1}+\tilde{ \mathfrak{F}}_{p}\Phi_{-1}^{\gamma+2s}, \eeno
 one has the corresponding decomposition: \beno \mathfrak{D}_{-1}^{1,1}=\mathfrak{D}_{-1,1}^{1,1}+\mathfrak{D}_{-1,2}^{1,1} . \eeno
 We first have
\beno |\mathfrak{D}_{-1,2}^{1,1}|
&\lesssim& \bigg(\int_0^1 \iint_{\sigma\in \SS^2,v_*,v\in \R^3} \big|\big(\tilde{ \mathfrak{F}}_{p}\Phi_{-1}^{\gamma+2s}\big)(|v-v_*|)\big|^2b(\cos\theta)\sin^2(\theta/2)|\psi(2^l(v'-v))|\\&&\times|v-v_*|^{2-2s}|(\tilde{ \mathfrak{F}}_{p}f)'|^2d\sigma dv_* dvd\kappa\bigg)^\f12\bigg(\int_0^1 \iint_{\sigma\in \SS^2,v_*,v\in \R^3} |(\mathfrak{F}_{p}g)_*|^2b(\cos\theta)\\&&\times\sin^2(\theta/2)|\psi(2^l(v'-v))||v-v_*|^{2-2s}\big|(\na^2 \fF_l h)(\kappa(v))\big|^2d\sigma dv_* dvd\kappa\bigg)^\f12.
\eeno
By change of variables from $(v_*, v)$ to $(u_1=v-v_*, u_2=v')$ and from $(v_*,v)$ to $(v_*, u_3=\kappa(v))$, we  get
\beno
|\mathfrak{D}_{-1,2}^{1,1}|
&\lesssim& \bigg(\int_0^1 \iint_{\sigma\in \SS^2,u_1,u_2\in \R^3} \big|\big(\tilde{ \mathfrak{F}}_{p}\Phi_{-1}^{\gamma+2s}\big)(|u_1|)\big|^2b(\cos\theta)\sin^2(\theta/2)|\psi(2^l|u_1|\sin(\theta/2))|\\&&\times|u_1|^{2-2s}|(\tilde{ \mathfrak{F}}_{p}f)(u_2)|^2d\sigma du_1 du_2d\kappa\bigg)^\f12\bigg(\int_0^1 \iint_{\sigma\in \SS^2,v_*,u_3\in \R^3} |(\mathfrak{F}_{p}g)_*|^2b(\cos\theta)\\&&\times\sin^2(\theta/2)|\psi(2^l|v-v_*|\sin(\theta/2))||u_3-v_*|^{2-2s}\big|(\na^2 \fF_l h)(u_3))\big|^2d\sigma dv_* du_3d\kappa\bigg)^\f12.
\eeno
Then we have 
\beno  |\mathfrak{D}_{-1,2}^{1,1}|\lesssim 2^{2sl} \|\tilde{ \mathfrak{F}}_{p}\Phi_{-1}^{\gamma+2s}\|_{L^2}\|\mathfrak{F}_{p}g\|_{L^2}\|\fF_l h\|_{L^2}\|\tilde{ \mathfrak{F}}_{p}f\|_{L^2}.\eeno

Due to the fact (see \cite{amuxy1}) that for  $\gamma>-3$ and the multi-index $\alpha$ with $|\alpha|=k\in\N$, 
\ben\label{Phi2} |\big(\pa^\alpha_\xi \mathcal{F}(\Phi_{-1}^\gamma)\big)(\xi)|\lesssim \langle \xi\rangle^{-3-\gamma-k}, \een   we have 
\beno \|\tilde{\mathfrak{F}}_{p}\Phi_{-1}^{\gamma+2s}\|_{L^2}^2&\lesssim& \int_{\R^3} |\varphi_p(\xi)|^2|\xi|^{-2(\gamma+2s)-6}d\xi
\\&\lesssim& 2^{-2(\gamma+2s+\f32)p}.\eeno
We deduce that if $\gamma+2s>-1$,
\beno  |\mathfrak{D}_{-1,2}^{1,1}|\lesssim 2^{2sl}2^{-(\gamma+2s+\f32)p} \|\mathfrak{F}_{p}g\|_{L^2}\|\fF_l h\|_{L^2}\|\tilde{ \mathfrak{F}}_{p}f\|_{L^2}.\eeno
Due to the Bernstein inequalities (\ref{bern1}-\ref{bern2}),  if $\gamma+2s>0$, then we have
\beno  |\mathfrak{D}_{-1,2}^{1,1}|\lesssim 2^{2sl}2^{-(\gamma+2s)p} \|\mathfrak{F}_{p}g\|_{L^1}\|\fF_l h\|_{L^2}\|\tilde{ \mathfrak{F}}_{p}f\|_{L^2}.\eeno
 
Next we turn to the estimate of $\mathfrak{D}_{-1,1}^{1,1}$. Notice that $B^{s}_{-1}$  can be separated into two parts $B_1$ and $B_2$ which   can be controlled in $L^2$ and $L^\infty$   spaces thanks to Proposition \ref{baslem6}. Thus  we may copy the argument for   $\mathfrak{D}_{-1,2}^{1,1}$  to $\mathfrak{D}_{-1,1}^{1,1}$  when $B_1$ is bounded in  $L^2$ space and apply the argument for $\mathfrak{D}_{k}^{1,1}$   to  $\mathfrak{D}_{-1,1}^{1,1}$  when $B_2$ is controlled in $L^\infty$ space. Finally we obtain that  in the case of
 $\gamma+2s>0$,
\beno   |\mathfrak{D}_{-1,1}^{1,1}|&\lesssim& 2^{2sl}
(\|B_{1}\|_{L^2}\|\mathfrak{F}_{p}g\|_{L^2}+\|B_2\|_{L^\infty}\|\mathfrak{F}_{p}g\|_{L^1})\|\fF_l h\|_{L^2}\|\tilde{ \mathfrak{F}}_{p}f\|_{L^2} \\
&\lesssim&2^{2sl}2^{-\eta_1 p}\|\mathfrak{F}_{p}g\|_{L^{1}}\|\fF_l h\|_{L^2}\|\tilde{ \mathfrak{F}}_{p}f\|_{L^2}.\eeno
While in  the case of  $\gamma+2s>-1$,  
\beno   |\mathfrak{D}_{-1,1}^{1,1}|
&\lesssim&  2^{2sl}2^{-(\f12+\eta_2)p}\|\mathfrak{F}_{p}g\|_{L^{2}}\|\fF_l h\|_{L^2}\|\tilde{ \mathfrak{F}}_{p}f\|_{L^2}. \eeno

\noindent{\it Step 1.2: Estimate of $\mathfrak{D}_k^{1,2}$.} We separate the estimate into several cases.  

{\it Case 1: $k\ge0$}.
By Cauchy-Schwartz inequality,  one has \beno |\mathfrak{D}_k^{1,2}|
&\lesssim& \|A^s_k\|_{L^\infty}\bigg(  \iint_{\sigma\in \SS^2,v_*,v\in \R^3}  |(\mathfrak{F}_{p}g)_*|b(\cos\theta)|1-\psi(2^l(v'-v))|\\&&\times|v-v_*|^{-2s}|(\tilde{ \mathfrak{F}}_{p}f)'|^2d\sigma dv_* dv \bigg)^\f12\bigg( \iint_{\sigma\in \SS^2,v_*,v\in \R^3} |(\mathfrak{F}_{p}g)_*| b(\cos\theta)\\&&\times|1-\psi(2^l(v'-v))||v-v_*|^{-2s}\big|( \fF_l h)(v)\big|^2d\sigma dv_* dv \bigg)^\f12\\
&&+\|A^s_k\|_{L^\infty}\bigg( \iint_{\sigma\in \SS^2,v_*,v\in \R^3} |(\mathfrak{F}_{p}g)_*|b(\cos\theta)|1-\psi(2^l(v'-v))|\\&&\times|v-v_*|^{-2s}|(\tilde{ \mathfrak{F}}_{p}f)'|^2d\sigma dv_* dv \bigg)^\f12\bigg(\int_0^1 \iint_{\sigma\in \SS^2,v_*,v\in \R^3}  |(\mathfrak{F}_{p}g)_*|b(\cos\theta)\\&&\times|1-\psi(2^l(v'-v))||v-v_*|^{-2s}\big|( \fF_l h)(v')\big|^2d\sigma dv_* dv \bigg)^\f12.\eeno
Follow the change of variables:   $(v_*,v)\rightarrow  (v_*,u_2=v')$  and the fact \eqref{Jacobi}, then we get
 \beno |\mathfrak{D}_k^{1,2}|
&\lesssim&   \|A^s_k\|_{L^\infty}\bigg(  \iint_{\sigma\in \SS^2,v_*,u_2\in \R^3} \big| (\mathfrak{F}_{p}g)_*|b(\cos\theta)|1-\psi(2^l|v-v_*|\sin(\theta/2))|\\&&\times|v-v_*|^{-2s}|(\tilde{ \mathfrak{F}}_{p}f)(u_2)|^2d\sigma dv_*du_2 \bigg)^\f12\bigg(  \iint_{\sigma\in \SS^2,v_*,v\in \R^3} |(\mathfrak{F}_{p}g)_*|b(\cos\theta)\\&&\times|1-\psi(2^l|v-v_*|\sin(\theta/2))||v-v_*|^{-2s}\big|( \fF_l h)(v)\big|^2d\sigma dv_* dv \bigg)^\f12\\
&&+ \|A^s_k\|_{L^\infty}\bigg(  \iint_{\sigma\in \SS^2,v_*,u_2\in \R^3}  |(\mathfrak{F}_{p}g)_*|b(\cos\theta)|1-\psi(2^l|v-v_*|\sin(\theta/2))|\\&&\times|v-v_*|^{-2s}|(\tilde{ \mathfrak{F}}_{p}f)(u_2)|^2d\sigma dv_* du_2\bigg)^\f12\bigg(  \iint_{\sigma\in \SS^2,v_*,u_2\in \R^3} |(\mathfrak{F}_{p}g)_*|b(\cos\theta)\\&&\times|1-\psi(2^l|v-v_*|\sin(\theta/2))||u_2-v_*|^{-2s}\big|( \fF_l h)(u_2)\big|^2d\sigma dv_* du_2\bigg)^\f12.\eeno
Notice
\ben \label{ub2}  &&|v-v_*|^{-2s}\int_{\SS^2} b(\cos\theta) (1-\psi(2^{l}|v-v_*|\sin(\theta/2)) )d\sigma\notag\\
&& \lesssim\left\{\begin{aligned} &|v-v_*|^{-2s}\int_{\f432^{-l}|v-v_*|^{-1}}^{\pi/2} b(\cos\theta)\sin\theta d\theta\\
& |u_2-v_*|^{-2s}\int_{\f432^{-l}|u_2-v_*|^{-1}}^{\pi/2} b(\cos\theta) \sin\tilde{\theta}  d\tilde{\theta} \end{aligned}\right.\notag \\&&
\lesssim 2^{2sl},\een
where $\tilde{\theta}$ verifies $\cos\tilde{\theta}=\f{u_2-v_*}{|u_2-v_*|}\cdot \sigma$  and $\theta/2\le\tilde{\theta}\le \theta$.
Finally we get the estimate to $\mathfrak{D}_k^{1,2}$, that is, for any $N\in \N$,
\beno |\mathfrak{D}_k^{1,2}|\lesssim  2^{(\gamma+\f32-N)k}2^{2s(l-p)} 2^{(2s-N)p}\|\Phi_0^{\gamma}\|_{H^{N+2}}\|\varphi\|_{W^{2,\infty}_N} \|\mathfrak{F}_{p}g\|_{L^1}\|\fF_l h\|_{L^2}\|\tilde{ \mathfrak{F}}_{p}f\|_{L^2}.\eeno

 {\it Case 2: $k=-1$.}  Following the similar argument used in the previous step, we conclude that in the case of $\gamma=0$,
  \beno |\mathfrak{D}_{-1}^{1,2}|\lesssim  2^{2s(l-p)} 2^{(2s-N)p}\|\psi\|_{H^{N+2}}\|\varphi\|_{W^{2,\infty}_N}\|\mathfrak{F}_{p}g\|_{L^1}\|\fF_l h\|_{L^2}\|\tilde{ \mathfrak{F}}_{p}f\|_{L^2},\eeno
in the case of $\gamma+2s>0$,
\beno  |\mathfrak{D}_{-1}^{1,2}| &\lesssim& 2^{2sl}
(\|B_{1}\|_{L^2}\|\mathfrak{F}_{p}g\|_{L^2}+\|B_2\|_{L^\infty}\|\mathfrak{F}_{p}g\|_{L^1})\|\fF_l h\|_{L^2}\|\tilde{ \mathfrak{F}}_{p}f\|_{L^2}
 \\&\lesssim&2^{2sl}2^{-\eta_1 p}\|\mathfrak{F}_{p}g\|_{L^{1}}\|\fF_l h\|_{L^2}\|\tilde{ \mathfrak{F}}_{p}f\|_{L^2},\eeno
and in  the case of  $\gamma+2s>-1$,
\beno  |\mathfrak{D}_{-1}^{1,2} |
&\lesssim& 2^{2sl}2^{-(\f12+\eta_2)p}\|\mathfrak{F}_{p}g\|_{L^{2}}\|\fF_l h\|_{L^2}\|\tilde{ \mathfrak{F}}_{p}f\|_{L^2}. \eeno
\medskip

{\it Step 2: Estimate  of $\mathfrak{D}_k^2$.} Thanks to the Cancellation Lemma in \cite{advw}, we obtain that
\beno \mathfrak{D}_k^2&=&|\SS^1|\iint_{\theta\in [0,\pi/2],v_*,v\in \R^3} \bigg[\big(\tilde{ \mathfrak{F}}_{p}\Phi_k^\gamma\big)\bigg(\f{|v-v_*|}{\cos\f{\theta}2}\bigg)\f1{\cos^3\f{\theta}2}-\big(\tilde{ \mathfrak{F}}_{p}\Phi_k^\gamma\big)(|v-v_*|)\bigg]b(\cos\theta)\sin\theta\\ &&\quad\times(\mathfrak{F}_{p}g)_*(\mathfrak{F}_{l}h)
(\tilde{ \mathfrak{F}}_{p}f)d\theta dv_* dv.\eeno

 Notice that
\beno  C_k(\theta, \xi)&\eqdefa&\f1{\cos^3\f{\theta}2}\mathcal{F}\bigg(\big(\tilde{ \mathfrak{F}}_{p}\Phi_k^\gamma\big)\big(\f{\cdot}{\cos\f{\theta}2}\big)\bigg)(\xi)-\mathcal{F}\bigg( \big(\tilde{ \mathfrak{F}}_{p}\Phi_k^\gamma\big)(\cdot)\bigg)(\xi)\\
&=&\varphi_p(\xi)2^{(\gamma+3)k} \bigg(\mathcal{F}\big(\Phi_0^{\gamma}\big)(\cos\f{\theta}2 2^k\xi)- (\mathcal{F}\big(\Phi_0^{\gamma}\big)(2^k\xi)\bigg),\eeno
where we use \eqref{Phi3}. We split the estimate into several cases.

{\it Case 1: $k\ge0$ or $k=-1$ with $\gamma=0$.} Thanks to   the mean value theorem,
we have \beno  \|C_k(\theta,\cdot)\|_{L^2}\lesssim   \theta^2 2^{(\gamma+\f52-N)k}2^{(1-N)p}\|  \Phi_{0}^\gamma \|_{H^N}\|\varphi\|_{W^{2,\infty}_N}\eeno
and \beno  \|C_{-1}(\theta,\cdot)\|_{L^2}\lesssim   \theta^2 2^{(\gamma+\f52-N)k}2^{(1-N)p}\|  \psi\|_{H^N}\|\varphi\|_{W^{2,\infty}_N}.\eeno
Then by Bernstein inequalities (\ref{bern1}-\ref{bern2}), for any $N\in \N$, we get   
\beno  |\mathfrak{D}_k^2|&\lesssim &\int_{\SS^2} b(\cos\theta)\|C_k(\theta, \cdot)\|_{L^2}\|\mathcal{F}\big(\mathfrak{F}_{p}g\big)\|_{L^\infty}
\|\mathcal{F}\big(\mathfrak{F}_{l}h\tilde{ \mathfrak{F}}_{p}f\big)\|_{L^2} d\sigma\\
&\lesssim& 2^{(\gamma+\f52-N)k}2^{(1-N)p}\|\varphi\|_{W^{2,\infty}_N}\|  \Phi_{0}^\gamma\|_{H^N}\|\mathfrak{F}_{p}g\|_{L^1}
\|\mathfrak{F}_{l}h\|_{L^\infty} \|\tilde{ \mathfrak{F}}_{p}f\|_{L^2}\\
&\lesssim& 2^{(\gamma+\f52-N)k}2^{(1-N)p}2^{\f32 l}\|\varphi\|_{W^{2,\infty}_N}\|  \Phi_{0}^\gamma\|_{H^N}\|\mathfrak{F}_{p}g\|_{L^1}
\|\mathfrak{F}_{l}h\|_{L^2} \|\tilde{ \mathfrak{F}}_{p}f\|_{L^2}
\eeno
and \beno |\mathfrak{D}_{-1}^2|&\lesssim & 2^{(1-N)p}2^{\f32 l}\|\varphi\|_{W^{2,\infty}_N}\| \psi\|_{H^N}\|\mathfrak{F}_{p}g\|_{L^1}
\|\mathfrak{F}_{l}h\|_{L^2} \|\tilde{ \mathfrak{F}}_{p}f\|_{L^2}.\eeno

{\it Case 2: $k=-1$ with general potentials.} Due to the fact \eqref{Phi2}, it is easy to check that
\beno  |C_{-1}(\theta, \xi)|&\lesssim& \theta^2|\xi||\varphi_p(\xi)|\|\na \mathcal{F}(\Phi_{-1}^\gamma)\|_{L^\infty} \\
&\lesssim& 2^{-(\gamma+3)p}\theta^2.\eeno
Thanks to   Plancherel theorem and Bernstein inequalities (\ref{bern1}-\ref{bern2}),    one has
\beno  |\mathfrak{D}_{-1}^{2}|&\lesssim&2^{-(\gamma+3)p}\|\mathfrak{F}_{p}g\|_{L^2}\|(\mathfrak{F}_{l}h)
(\tilde{ \mathfrak{F}}_{p}f)\|_{L^2}\\
&\lesssim& 2^{-(\gamma+3)p}2^{\f32 l}\|\mathfrak{F}_{p}g\|_{L^2}\|\mathfrak{F}_{l}h\|_{L^2}\|
\tilde{ \mathfrak{F}}_{p}f\|_{L^2}.\eeno
Thus   we have
\begin{enumerate}
\item  if $\gamma+2s>0 $ and $\gamma>-\f32$,
\beno |\mathfrak{D}_{-1}^{2}|&\lesssim&  2^{(l-p)(\gamma+\f32)}2^{(2s-(\gamma+2s))l} 
\|\mathfrak{F}_{p}g\|_{L^1}\|\mathfrak{F}_{l}h\|_{L^2}\|
\tilde{ \mathfrak{F}}_{p}f\|_{L^2},\eeno
\item  if $-1<\gamma+2s\le0 $ and $\gamma>-\f52$,
\beno |\mathfrak{D}_{-1}^{2}|&\lesssim&
  2^{(l-p)(\gamma+\f52)}2^{-(\gamma+2s+1)l}2^{2sl} 
\|\mathfrak{F}_{p}g\|_{L^\f32}\|\mathfrak{F}_{l}h\|_{L^2}\|
\tilde{ \mathfrak{F}}_{p}f\|_{L^2},\eeno
\item if $-1<\gamma+2s$,
\beno |\mathfrak{D}_{-1}^{2}|&\lesssim&
 2^{(l-p)(\gamma+3)}2^{-(\gamma+2s+1)l}2^{(2s-\f12)l}\|\mathfrak{F}_{p}g\|_{L^2}\|\mathfrak{F}_{l}h\|_{L^2}\|
\tilde{ \mathfrak{F}}_{p}f\|_{L^2}.\eeno
\end{enumerate}
\bigskip

Now combine all the estimates in    {\it Step 1} and {\it Step 2} and use  Bernstein inequalities (\ref{bern1}-\ref{bern2}), then we finally get the desired results in the lemma.
 \end{proof}
\medskip

Next we have
\begin{lem}\label{lemub2} Suppose $N\in \N$.   For $k\ge0$, it holds
\beno  |\mathfrak{W}_{k,p,m}^4| & \lesssim& 2^{2s(m-p)} 2^{(\gamma+\f32-N)k}2^{-p(N-\f52)}\|\Phi_0^\gamma\|_{H^{N+2}}\|\varphi\|_{W^{2,\infty}_N} \|\mathfrak{F}_{p}g\|_{L^1}\|
\tilde{ \mathfrak{F}}_{p}h\|_{L^2}\|\mathfrak{F}_{m}f\|_{L^2}.\eeno
For $k=-1$, we have
\begin{enumerate}
\item if $\gamma=0$,  \beno  |\mathfrak{W}_{-1,p,m}^4| & \lesssim&   2^{2s(m-p)}  2^{-p(N-\f52)}\|\psi\|_{H^{N+2}}\|\varphi\|_{W^{2,\infty}_N} \|\mathfrak{F}_{p}g\|_{L^1}\|
\tilde{ \mathfrak{F}}_{p}h\|_{L^2}\|\mathfrak{F}_{m}f\|_{L^2},\eeno
\item if $\gamma+2s>0 $, \beno |\mathfrak{W}_{-1,p,m}^4|\lesssim
  2^{\tilde{\eta}(m-p)}2^{(2s-\eta_1)m}\|\mathfrak{F}_{p}g\|_{L^1}\|\tilde{ \mathfrak{F}}_{p}h\|_{L^2}\|\mathfrak{F}_{m}f\|_{L^2}, \eeno
  \item if $\gamma+2s>-1 $, \beno   |\mathfrak{W}_{-1,p,m}^4|&\lesssim& 2^{2sm}2^{-(\f12+\eta_2)p}\|\mathfrak{F}_{p}g\|_{L^{2}}\|\tilde{ \mathfrak{F}}_{p}h\|_{L^2}\|\mathfrak{F}_{m}f\|_{L^2}
\\  &\lesssim& 2^{\eta(m-p)}2^{-\f12p}2^{(2s-\eta_2)m}\|\mathfrak{F}_{p}g\|_{L^{2}}\|\tilde{ \mathfrak{F}}_{p}h\|_{L^2}\|\mathfrak{F}_{m}f\|_{L^2}, \eeno
\end{enumerate}
where $\tilde{\eta}$ is a positive constant which depends only on $\gamma$ and $s$. We remark that  constants 
   $\eta_1$ and $\eta_2$ are stated in Proposition \ref{baslem6} and functions $\psi$ and $\varphi$ are defined in \eqref{defpsivarphi}.
\end{lem}
\begin{proof} Noticing the fact $m<p-N_0$,  we follow the similar decomposition used in Lemma \ref{lemub1} to get
\beno   \mathfrak{W}_{k,p,m}^4&=&\iint_{\sigma\in \SS^2,v_*,v\in \R^3} \big(\tilde{ \mathfrak{F}}_{p}\Phi_k^\gamma\big)(|v-v_*|)b(\cos\theta) \psi(2^m(v'-v)) (\mathfrak{F}_{p}g)_*(\tilde{\mathfrak{F}}_{p}h)\\&&\times\big[
(   \mathfrak{F}_{m}f)'-
 \mathfrak{F}_{m}f\big]d\sigma dv_* dv+\iint_{\sigma\in \SS^2,v_*,v\in \R^3} \big(\tilde{ \mathfrak{F}}_{p}\Phi_k^\gamma\big)(|v-v_*|)b(\cos\theta) \\
&&\times\big(1-\psi(2^m(v'-v))\big)(\mathfrak{F}_{p}g)_*(\tilde{\mathfrak{F}}_{p}h)\big[
(   \mathfrak{F}_{m}f)'-
 \mathfrak{F}_{m}f\big]d\sigma dv_* dv\\
&\eqdefa& \mathfrak{E}_{k}^1+\mathfrak{E}_k^2. \eeno

Observe the fact
\ben\label{taylor2}  &&(\mathfrak{F}_{m}f)(v')-
(\mathfrak{F}_{m}f)(v)\nonumber\\&&=(v'-v)\cdot (\na\mathfrak{F}_mf)(v)+\f12\int_0^1(1-\kappa)(v'-v)\otimes(v'-v): (\na^2\mathfrak{F}_mf)(\kappa(v))d\kappa,\een
where $\kappa(v)=v+\kappa(v'-v)$, then we have the further decomposition:
\beno \mathfrak{E}_{k}^1&=&\mathfrak{E}_{k}^{1,1}+\mathfrak{E}_{k}^{1,2},\eeno
where \beno \mathfrak{E}_{k}^{1,1}&=&\iint_{\sigma\in \SS^2,v_*,v\in \R^3}  \big(\tilde{ \mathfrak{F}}_{p}\Phi_k^\gamma\big)(|v-v_*|)b(\cos\theta) \psi(2^m(v'-v)) (\mathfrak{F}_{p}g)_*(\tilde{\mathfrak{F}}_{p}h)\\&&\times(v'-v)\cdot (\na\mathfrak{F}_mf)(v)d\sigma dv_* dv,\\
 \mathfrak{E}_{k}^{1,2}&=&\int_0^1\iint_{\sigma\in \SS^2,v_*,v\in \R^3} A^s_k(|v-v_*|)b(\cos\theta) \psi(2^m(v'-v)) (\mathfrak{F}_{p}g)_*(\tilde{\mathfrak{F}}_{p}h)\\&&\times |v-v_*|^{-2s}\big[
(v'-v)\otimes(v'-v): (\na^2\mathfrak{F}_mf)(\kappa(v))\big]d\sigma dv_* dvd\kappa,\eeno
where $A_k^s$ is defined in   Proposition \ref{baslem6}.

It is not difficult to check the main structures of $\mathfrak{E}_{k}^{1,2}$  and  $\mathfrak{E}_{k}^{2}$ are almost as the same as those of
$\mathfrak{D}_{k}^{1,1}$ and $\mathfrak{D}_{k}^{1,2}$. We conclude
that for any $N\in \N$,
 \beno  |\mathfrak{E}_{k}^{1,2}|+|\mathfrak{E}_{k}^{2}|\lesssim\left\{\begin{aligned} & 2^{(\gamma+\f32-N)k} 2^{(2s-N)p} 2^{2s(m-p)}\|\Phi_0^\gamma\|_{H^{N+2}}\|\varphi\|_{W^{2,\infty}_N}\|\mathfrak{F}_{p}g\|_{L^1}\|\tilde{ \mathfrak{F}}_{p}h\|_{L^2}\|\fF_m f\|_{L^2}, \mbox{if}\, k\ge0,\\
  &2^{2s(m-p)} 2^{(2s-N)p}\|\psi\|_{H^{N+2}} \|\varphi\|_{W^{2,\infty}_N} \|\mathfrak{F}_{p}g\|_{L^1}\|\tilde{ \mathfrak{F}}_{p}h\|_{L^2}\|\fF_m f\|_{L^2}, \mbox{if}\, k=-1\,\mbox{and}\, \gamma=0,
 \\&2^{2sm}2^{-\eta_1 p}\|\mathfrak{F}_{p}g\|_{L^{1}}\|\tilde{ \mathfrak{F}}_{p}h\|_{L^2}\|\fF_m f\|_{L^2}, \,\mbox{if}\,\,k=-1\,\,\mbox{and}\,\,\gamma+2s>0,\\&2^{2sm}2^{-(\f12+\eta_2) p}
  \|\mathfrak{F}_{p}g\|_{L^{ 2}}\|\tilde{ \mathfrak{F}}_{p}h\|_{L^2}\|\fF_m f\|_{L^2}, \,\mbox{if}\,\,k=-1\,\,\mbox{and}\,\,\gamma+2s>-1.
 \end{aligned}\right.\eeno

Now we only need to give the bound to  $\mathfrak{E}_{k}^{1,1}$. 
Thanks to the fact
\beno &&\int_{\SS^2}  b(
\frac{v-v_*}{|v-v_*|}\cdot \sigma )(v-v')\psi(2^m|v-v'|)d\sigma\\
&&\quad=\int_{\SS^2}  b(
\frac{v-v_*}{|v-v_*|}\cdot\sigma )\frac{v-v'}{|v-v'|}\cdot\frac{v-v_*}{|v-v_*|}|v-v'|\psi (|2^m|v-v'|)\frac{v-v_*}{|v-v_*|}d\sigma\\
&&\quad=\int_{\SS^2}  b(  \frac{v-v_*}{|v-v_*|}\cdot\sigma )
 \f{1-\langle
\frac{v-v_*}{|v-v_*|},\sigma\rangle}{2} \psi (2^m|v-v'|)d\sigma
(v-v_*),
 \eeno
one has  \ben\label{angular1} &&2^{(2-2s)m}|v-v_*|^{-2s}\bigg|\int_{\SS^2}  b(
\frac{v-v_*}{|v-v_*|}\cdot\sigma )(v-v')\psi (2^m|v-v'|)d\sigma\bigg|\nonumber\\
&&\lesssim 2^{(2-2s)m}|v-v_*|^{ -2s}\int_{\sqrt{\f{1-
\frac{v-v_*}{|v-v_*|}\cdot\sigma}{2}}\lesssim 2^{-m}|v-v_*|^{-1}}
b(  \frac{v-v_*}{|v-v_*|}\cdot\sigma ) \f{1-
\frac{v-v_*}{|v-v_*|}\cdot\sigma}{2}d\sigma
|v-v_*|\nonumber\\
&&\lesssim 2^{(2-2s)m}|v-v_*|^{-2s}\int_0^{2^{-m}|v-v_*|^{-1}} \theta^{1-2s}d\theta |v-v_*| \nonumber\\
&&\lesssim |v-v_*|^{-1}.
 \een

 In other words, if we set
 \beno U(v)\eqdefa 2^{(2-2s)m}|v|^{-2s}v\int_{\SS^2}  b(  \frac{v }{|v |}\cdot\sigma )
 \f{1-\langle
\frac{v }{|v |},\sigma\rangle}{2} \psi \bigg(2^{m-1}|v|\sqrt{ 1-\langle
\frac{v }{|v |},\sigma\rangle}\sqrt{2}\bigg) d\sigma, \eeno then  \eqref{angular1} yields $|U(v)|\lesssim |v|^{-1}$.
  
  Due to this observation, we have
\beno |\mathfrak{E}_{k}^{1,1}|&=& \bigg|2^{(2s-2)m}\iint_{v,v_*\in\R^3}  A^s_k(|v-v_*|)  (\mathfrak{F}_{p}g)_*(\tilde{\mathfrak{F}}_{p}h)(v)U(v-v_*)\cdot\na\mathfrak{F}_mf(v)dv_*dv\bigg|.\eeno
We divide the estimate into three cases.

{\it Case 1: $k\ge0$.}  By  the definition, we have $|A^s_kU|\lesssim A^{s-\f12}_k$.  If $s\ge 1/2$, then 
\beno |\mathfrak{E}_{k}^{1,1}|&\lesssim&  2^{(2s-2)m}\|A^{s-\f12}_k\|_{L^\infty}\|\mathfrak{F}_{p}g\|_{L^1}\|
\tilde{ \mathfrak{F}}_{p}h\|_{L^2}\|\na\mathfrak{F}_{m}f\|_{L^2}\\&\lesssim& 2^{2s(m-p)}2^{(\gamma+\f32-N)k}2^{-p(N-2s)}\|\Phi_0^\gamma\|_{H^{N+2}}\|\varphi\|_{W^{2,\infty}_N} \|\mathfrak{F}_{p}g\|_{L^1}\|
\tilde{ \mathfrak{F}}_{p}h\|_{L^2}\|\mathfrak{F}_{m}f\|_{L^2}.\eeno
 In the case of $s<1/2$, one has
 \beno |\mathfrak{E}_{k}^{1,1}|&\lesssim&  2^{(2s-2)m}\|A^{s-\f12}_k\|_{L^2}\|\mathfrak{F}_{p}g\|_{L^2}\|
\tilde{ \mathfrak{F}}_{p}h\|_{L^2}\|\na\mathfrak{F}_{m}f\|_{L^2}\\&\lesssim&2^{2s(m-p)} 2^{(\gamma+\f32-N)k}2^{-p(N-2s-\f32)}\|\Phi_0^\gamma\|_{H^{N+2}}\|\varphi\|_{W^{2,\infty}_N} \|\mathfrak{F}_{p}g\|_{L^1}\|
\tilde{ \mathfrak{F}}_{p}h\|_{L^2}\|\mathfrak{F}_{m}f\|_{L^2},\eeno
where we use the Hardy inequality to get
\beno  \|A^{s-\f12}_k\|_{L^2}\le \|\tilde{ \mathfrak{F}}_{p}\Phi_k^{\gamma}\|_{H^{1-2s}}\lesssim 2^{(\gamma+\f32-N)k}2^{-pN}\|\Phi_0^\gamma\|_{H^{N+2}}\|\varphi\|_{W^{2,\infty}_N}.\eeno

{\it  Case 2: $k=-1$ and $\gamma=0$.} In this case, we only need to copy the argument in $Case\, 1$ to get
\beno   |\mathfrak{E}_{-1}^{1,1}| &\lesssim&2^{2s(m-p)}  2^{-p(N-\f52)}\|\psi\|_{H^{N+2}}\|\varphi\|_{W^{2,\infty}_N} \|\mathfrak{F}_{p}g\|_{L^1}\|
\tilde{ \mathfrak{F}}_{p}h\|_{L^2}\|\mathfrak{F}_{m}f\|_{L^2}.\eeno

{\it  Case 3: $k=-1$ with general potentials.} 
Following the decomposition \beno  A^s_{-1}= \tilde{ \mathfrak{F}}_{p}\Phi_{-1}^{\gamma+2s}+B_{-1}^s,\eeno we split $\mathfrak{E}_{-1}^{1,1}$ into two parts: $ \mathfrak{E}_{-1}^{1,1,1}$ and  $ \mathfrak{E}_{-1}^{1,1,2}$ which are defined by 
\beno  \mathfrak{E}_{-1}^{1,1,1}&=&2^{(2s-2)m}\iint_{v,v_*\in\R^3}  \big(\tilde{ \mathfrak{F}}_{p}\Phi_{-1}^{\gamma+2s}\big)(|v-v_*|) (\mathfrak{F}_{p}g)_*(\tilde{\mathfrak{F}}_{p}h)(v)U(v-v_*)\cdot\na\mathfrak{F}_mf(v)dv_*dv,\\
  \mathfrak{E}_{-1}^{1,1,2}&=&2^{(2s-2)m}\iint_{v,v_*\in\R^3}  B_{-1}^s(|v-v_*|) (\mathfrak{F}_{p}g)_*(\tilde{\mathfrak{F}}_{p}h)(v)U(v-v_*)\cdot\na\mathfrak{F}_mf(v)dv_*dv.  \eeno

Thanks to Cauchy-Schwartz inequality, we have 
\beno  |\mathfrak{E}_{-1}^{1,1,1}|&\le& 2^{(2s-2)m} \bigg(\iint_{v,v_*\in\R^3}   |(\tilde{ \mathfrak{F}}_{p}\Phi_k^{\gamma+2s})(|v-v_*|)|^2|(\mathfrak{F}_{p}g)_*||\tilde{\mathfrak{F}}_{p} h | dvdv_*\bigg)^{\f12}\\&&\times\bigg(\iint_{v,v_*\in\R^3}    |v-v_*|^{-2}|(\mathfrak{F}_{p}g)_*||\tilde{\mathfrak{F}}_{p} h ||\na\mathfrak{F}_mf|^2 dvdv_*\bigg)^{\f12}\\
&\lesssim& 2^{(2s-2)m} \|\tilde{ \mathfrak{F}}_{p}\Phi_k^{\gamma+2s}\|_{L^2}\|\mathfrak{F}_{p}g\|_{L^2}\|\tilde{\mathfrak{F}}_{p} h\|_{L^2} \|\na\mathfrak{F}_mf\|_{L^6}\\
&\lesssim& 2^{2sm}2^{-(\gamma+2s+\f32)p} \|\mathfrak{F}_{p}g\|_{L^2}\|\tilde{\mathfrak{F}}_{p} h\|_{L^2} \|\mathfrak{F}_mf\|_{L^2},\eeno
where we use Young and Hardy-Littlewood-Sobolev inequalities and Lemma \ref{Bernstein-Ineq}.

Finally we turn to   the estimate of $\mathfrak{E}_{-1}^{1,1,2}$.  If $\gamma+2s>0$, thanks to Proposition \ref{baslem6},  Young and Hardy-Littlewood-Sobolev inequalities, we have for  $\delta<\eta_1/3$ where $\eta_1$ is defined in Proposition \ref{baslem6},
 \beno  |\mathfrak{E}_{-1}^{1,1,2}|&\lesssim&    2^{(2s-2)m} \| B_1\|_{L^2}\|\mathfrak{F}_{p}g\|_{L^2}\|\tilde{\mathfrak{F}}_{p} h\|_{L^2} \|\na\mathfrak{F}_mf\|_{L^6}\\&& + 2^{(2s-2)m}\|B_{2}\|_{L^\infty} \|\mathfrak{F}_{p}g\|_{L^{1+\delta}}\|\tilde{ \mathfrak{F}}_{p}h\|_{L^2} \|\na\mathfrak{F}_{m}f\|_{L^{\f{6(1+\delta)}{1+7\delta}}}
\\
 &\lesssim& 2^{2sm}2^{-\eta_1p} \|\mathfrak{F}_{p}g\|_{L^1}\|\tilde{\mathfrak{F}}_{p} h\|_{L^2} \|\mathfrak{F}_mf\|_{L^2}\\
 &&+ 2^{-(\eta_1-\f{3\delta}{1+\delta})p}2^{(2s-\f{3\delta}{1+\delta})m} \|\mathfrak{F}_{p}g\|_{L^1}\|\tilde{\mathfrak{F}}_{p} h\|_{L^2} \|\mathfrak{F}_mf\|_{L^2}\\
 &\lesssim&(2^{2sm}2^{-\eta_1p}+2^{(2s-\eta_1)m}2^{(\eta_1-\f{3\delta}{1+\delta})(m-p)}) \|\mathfrak{F}_{p}g\|_{L^1}\|\tilde{\mathfrak{F}}_{p} h\|_{L^2} \|\mathfrak{F}_mf\|_{L^2}.\eeno

While in the case of $\gamma+2s>-1$, with the help of Proposition \ref{baslem6},  follow the similar argument  applied to $\mathfrak{E}_{-1}^{1,1,1}$, then we get  \beno |\mathfrak{E}_{-1}^{1,1,2}| 
&\lesssim&  2^{(2s-2)m} \| B_{-1}^s\|_{L^2}\|\mathfrak{F}_{p}g\|_{L^2}\|\tilde{\mathfrak{F}}_{p} h\|_{L^2} \|\na\mathfrak{F}_mf\|_{L^6}\\
&\lesssim& 2^{2sm}2^{-(\f12+\eta_2)p} \|\mathfrak{F}_{p}g\|_{L^2}\|\tilde{\mathfrak{F}}_{p} h\|_{L^2} \|\mathfrak{F}_mf\|_{L^2}.  \eeno
 
 Now patch together all the estimates, then we are led to the desired results. \end{proof}

\subsection{Estimates  of $\mathfrak{W}_{k,l}^2$ and $\mathfrak{W}_{k,p}^3$  defined in \eqref{fsdecom}} Since $\mathfrak{W}_{k,l}^2$ enjoys almost the same structure as that of $\mathfrak{W}_{k,p}^3$, it suffices to give the estimate to $\mathfrak{W}_{k,p}^3$.

\begin{lem}\label{lemub3} If $k\ge0$, we have
\beno |\mathfrak{W}_{k,l}^2|&\lesssim& 2^{(\gamma+2s)k}2^{2sl}\|\mathcal{S}_{l-N_0}g\|_{L^1}\| \mathfrak{F}_{l} h\|_{L^2}\|\tilde{ \mathfrak{F}}_{l}f\|_{L^2}, \\
 |\mathfrak{W}_{k,p}^3|&\lesssim& 2^{(\gamma+2s)k}2^{2sp}\|\mathfrak{F}_{p}g\|_{L^1}\| \tilde{\mathfrak{F}}_{p}h \|_{L^2}\|\tilde{ \mathfrak{F}}_{p}f\|_{L^2}.\eeno
If $ k=-1$, we have
\beno |\mathfrak{W}_{-1,l}^2|&\lesssim& \left\{\begin{aligned} & 2^{2sl}\|\mathcal{S}_{l-N_0}g\|_{L^1}\|  \mathfrak{F}_{l} h\|_{L^2}\|\tilde{ \mathfrak{F}}_{l}f\|_{L^2}, \quad\mbox{if}\,\, \gamma+2s>0,\\
&  2^{2sl}\|\mathcal{S}_{l-N_0}g\|_{L^2}\| \mathfrak{F}_{l} h\|_{L^2}\|\tilde{ \mathfrak{F}}_{l}f\|_{L^2},\quad\mbox{if}\,\, \gamma+2s\le0, \end{aligned}\right. \\
 |\mathfrak{W}_{-1,p}^3|&\lesssim& \left\{\begin{aligned} & 2^{2sp}\|\mathfrak{F}_{p}g\|_{L^1}\| \tilde{\mathfrak{F}}_{p}h \|_{L^2}\|\tilde{ \mathfrak{F}}_{p}f\|_{L^2}, \quad\mbox{if}\,\, \gamma+2s>0,\\
& 2^{2sp}\|\mathfrak{F}_{p}g\|_{L^2}\| \tilde{\mathfrak{F}}_{p}h \|_{L^2}\|\tilde{ \mathfrak{F}}_{p}f\|_{L^2},\quad\mbox{if}\,\, \gamma+2s\le0.\end{aligned}\right. \eeno
\end{lem}

\begin{proof}  We introduce the function $\psi$ defined in \eqref{defpsivarphi} to decompose $\mathfrak{W}_{k,p}^3$ into two parts:
\beno   \mathfrak{W}_{k,p}^3&=&\iint_{\sigma\in \SS^2,v_*,v\in \R^3}  \Phi_k^\gamma (|v-v_*|)b(\cos\theta) \psi(2^p(v'-v)) (\mathfrak{F}_{p}g)_*(\tilde{\mathfrak{F}}_{p}h)\\&&\times\big[
(   \tilde{\mathfrak{F}}_{p}f)'-
 \tilde{\mathfrak{F}}_{p}f\big]d\sigma dv_* dv+\iint_{\sigma\in \SS^2,v_*,v\in \R^3}   \Phi_k^\gamma (|v-v_*|)b(\cos\theta) \\
&&\times\big(1-\psi(2^p(v'-v))\big)(\mathfrak{F}_{p}g)_* (\tilde{\mathfrak{F}}_{p}h)\big[
(  \tilde{\mathfrak{F}}_{p}f)'-
 \tilde{\mathfrak{F}}_{p}f\big]d\sigma dv_* dv. \eeno
By using Taylor expansion \eqref{taylor2}, the facts \eqref{angular2}, \eqref{angular1} and H\"{o}lder inequality, we deduce that for $k\ge0$,
\beno |\mathfrak{W}_{k,p}^3|\lesssim 2^{(\gamma+2s)k}2^{2sp}\|\mathfrak{F}_{p}g\|_{L^1}\| \tilde{\mathfrak{F}}_{p}h \|_{L^2}\|\tilde{ \mathfrak{F}}_{p}f\|_{L^2},\eeno
and
\beno |\mathfrak{W}_{-1,p}^3|\lesssim \left\{\begin{aligned} & 2^{2sp}\|\mathfrak{F}_{p}g\|_{L^1}\| \tilde{\mathfrak{F}}_{p}h \|_{L^2}\|\tilde{ \mathfrak{F}}_{p}f\|_{L^2}, \quad\mbox{if}\,\, \gamma+2s>0,\\
& 2^{2sp}\|\mathfrak{F}_{p}g\|_{L^2}\| \tilde{\mathfrak{F}}_{p}h \|_{L^2}\|\tilde{ \mathfrak{F}}_{p}f\|_{L^2},\quad\mbox{if}\,\, \gamma+2s\le0.\end{aligned}\right. \eeno
The similar results hold for $\mathfrak{W}_{k,l}^2$. We complete the proof of the lemma.
\end{proof}

\medskip
\subsection{ Proof of Theorem \ref{thmub} }      Now we are ready to give the proof to Theorem \ref{thmub}. \begin{proof} 
Let $a_1,b_1\in \R$ with $a_1+b_1=2s$. Then by Lemma \ref{lemub1}, Lemma \ref{lemub2} and Lemma \ref{lemub3}, we conclude that  for $k\ge0$ or $\gamma=0$ with $k\ge-1$,
 \beno \sum_{l\le p-N_0} |\mathfrak{W}_{k,p,l}^1|+\sum_{m<p-N_0} |\mathfrak{W}_{k,p,m}^4|&\lesssim&  C(a_1,b_1)\|g\|_{L^1}\|h\|_{H^{a_1 }}\|f\|_{H^{b_1 }}, \\
   \sum_{l\ge-1 } |\mathfrak{W}_{k,l}^2|+\sum_{p\ge-1}  |\mathfrak{W}_{k,p}^3|&\lesssim&2^{(\gamma+2s)k} \|g\|_{L^1} \|h\|_{H^{a_1 }}\|f\|_{H^{b_1 }},\eeno
Then we are led to that if $k\ge0$ or $\gamma=0$ with $k\ge-1$,
\ben\label{estqk1}  |\langle Q_k(g, h), f \rangle_v|  &\lesssim&C(a_1,b_1) 2^{(\gamma+2s) k} \|g\|_{L^1} \|h\|_{H^{a_1 }}\|f\|_{H^{b_2 }}. \een

  Let $a,b\in [0,2s]$ with $a+b=2s$.   If  $k=-1$, then by Lemma \ref{lemub1}, Lemma \ref{lemub2} and Lemma \ref{lemub3}, we have
 \beno &&\sum_{l\le p-N_0} |\mathfrak{W}_{k,p,l}^1|+\sum_{l\ge-1 } |\mathfrak{W}_{k,l}^2|+\sum_{p\ge-1}  |\mathfrak{W}_{k,p}^3|
+\sum_{m<p-N_0} |\mathfrak{W}_{k,p,m}^4|\\&&\lesssim  (\|g\|_{L^1}+\|g\|_{L^2})\|h\|_{H^{a }}\|f\|_{H^{b }},\eeno which yields
\ben\label{estqk2}  |\langle Q_{-1}(g, h), f \rangle_v|  &\lesssim&  (\|g\|_{L^1}+\|g\|_{L^2})\|h\|_{H^{a }}\|f\|_{H^{b }}. \een

Now we are in a position to give the upper bound for  the collision operator in weighted Sobolev space. Let $w_1,w_2\in\R$ with $w_1+w_2=\gamma+2s$. Recalling
\eqref{ubdecom},
we infer from \eqref{estqk1} and \eqref{estqk2},
\beno  | \langle Q ( g, h), f \rangle_v|&\lesssim& \sum_{k\ge N_0 -1} 2^{(\gamma+2s)k} \|\mathcal{U}_{k-N_0} g\|_{L^1} \|\tilde{\mathcal{P}}_kh\|_{H^{a }}\|\tilde{\mathcal{P}}_kf  \|_{H^{b }}\\&&+\bigg( \sum_{j\ge k+N_0,k\ge0}2^{(\gamma+2s)k}  \|\mathcal{P}_{j} g\|_{L^1} \|\tilde{\mathcal{P}}_jh\|_{H^{a }}\|\tilde{\mathcal{P}}_jf   \|_{H^{b }}\\&&+ \sum_{j\ge -1+N_0}  (\|\mathcal{P}_{j} g\|_{L^1}+\|\mathcal{P}_{j}g\|_{L^2})\|\tilde{\mathcal{P}}_jh\|_{H^{a }}\|\tilde{\mathcal{P}}_jf   \|_{H^{b }}\bigg)\\
&&+\bigg(\sum_{k\ge0, |j-k|\le N_0} 2^{(\gamma+2s)k}\|\mathcal{P}_{j} g  \|_{L^1} \| \mathcal{U}_{k+N_0}h\|_{H^{a }}\| \mathcal{U}_{k+N_0}f  \|_{H^{b }}
\\ &&\quad+(\|g\|_{L^1}+\|g\|_{L^2})\|\mathcal{U}_{N_0}h\|_{H^{a}}\|\mathcal{U}_{N_0}f  \|_{H^{b}}\bigg) \\ &\eqdefa& \mathfrak{U}_1+\mathfrak{U}_2+\mathfrak{U}_3.\eeno

For the term $\mathfrak{U}_1$, thanks to Theorem \ref{baslem3}, one has \beno   \mathfrak{U}_1\lesssim  \|g\|_{L^1} \|h\|_{H^{a}_{w_1}}\|f\|_{H^{b }_{w_2}}.\eeno

For the term $\mathfrak{U}_2$, we separate the estimate into three cases. If $\gamma+2s>0$, we have
\beno \mathfrak{U}_2&\lesssim &  \sum_{j\ge -1}2^{(\gamma+2s)j} (\|\mathcal{P}_{j} g\|_{L^1}+\|\mathcal{P}_{j}g\|_{L^2})\|\tilde{\mathcal{P}}_jh\|_{H^{a }}\|\tilde{\mathcal{P}}_jf   \|_{H^{b }}\\&\lesssim&  (\|g\|_{L^1}+\|g\|_{L^2})\|h\|_{H^{a}_{w_1}}\|f\|_{H^{b }_{w_2}}.\eeno
In the case of $\gamma+2s=0$, it holds for any $\delta>0$,
\beno \mathfrak{U}_2&\lesssim &  \sum_{j\ge -1} (\|g\|_{L^1_\delta}+\|g\|_{L^2})\|\tilde{\mathcal{P}}_jh\|_{H^{a }}\|\tilde{\mathcal{P}}_jf   \|_{H^{b }}\\&\lesssim&  (\|g\|_{L^1_\delta}+\|g\|_{L^2 })\|h\|_{H^{a}_{w_1}}\|f\|_{H^{b }_{w_2}}.\eeno
While in the case of $\gamma+2s<0$, we have
\beno \mathfrak{U}_2&\lesssim &  \sum_{j\ge -1} (\|\mathcal{P}_{j} g\|_{L^1}+\|\mathcal{P}_{j}g\|_{L^2})\|\tilde{\mathcal{P}}_jh\|_{H^{a }}\|\tilde{\mathcal{P}}_jf   \|_{H^{b }}\\&\lesssim&  (\|g\|_{L^1_{-(\gamma+2s) }}+\|g\|_{L^2_{-(\gamma+2s) }})\|h\|_{H^{a}_{w_1}}\|f\|_{H^{b }_{w_2}}.\eeno
We remark that in each case Theorem \ref{baslem3} is used in the last inequality.

Now we turn to the term $\mathfrak{U}_3$.
We  first claim that it holds
\ben\label{u4} \|\mathcal{U}_{k+N_0}h\|_{H^{a}}\lesssim 2^{k(-w_1)^+} \|h\|_{H^a_{w_1}}. \een
By the definition of $\mathcal{U}$,  
we have
\beno (\mathcal{U}_{k+N_0}h)(v)=\big[ \sum_{j\le k+N_0} \varphi(2^{-j}v) +\psi(v)\big]h(v)\eqdefa \tilde{\psi}_{k+N_0}(v)h(v).\eeno
Thanks to Lemma \ref{baslem2} and the facts  if $w_1\ge 0$,
\beno \pa^{\alpha}_v\big(\tilde{\psi}_{k+N_0}\langle v\rangle^{-w_1}\big)\lesssim \langle v\rangle^{-w_1-|\alpha|}\lesssim \langle v\rangle^{-|\alpha|},  \eeno
and if $w_1<0$,
\beno \pa^{\alpha}_v\big(2^{kw_1}\tilde{\psi}_{k+N_0}\langle v\rangle^{-w_1}\big)\lesssim \langle v\rangle^{-|\alpha|},  \eeno
we deduce that
\beno \|\mathcal{U}_{k+N_0}h\|_{H^{a}}&\lesssim& \|\tilde{\psi}_{k+N_0}W_{-w_1}(W_{w_1}h)\|_{H^{a}}\\
&\lesssim& 2^{k(-w_1)^+} \|h\|_{H^a_{w_1}},\eeno
  which completes the proof to the claim.
Now we apply the claim to the estimate of $\mathfrak{U}_3$. It is easy to get
\beno \mathfrak{U}_3\lesssim (\|g\|_{L^1_{\gamma+2s+(-w_1)^++(-w_2)^+}}+\|g\|_{L^2})\|h\|_{H^{a}_{w_1}}\|f\|_{H^{b }_{w_2}}. \eeno

Now patching together all the estimates for $\mathfrak{U}_1$, $\mathfrak{U}_2$ and $\mathfrak{U}_3$, we obtain the desired results.
For the special case $\gamma=0$, by the estimate \eqref{estqk1} and   the similar argument, we can easily get \eqref{u5}.
\end{proof}

\setcounter{equation}{0}

\section{Lower and upper bounds for the Boltzmann collision operator in anisotropic spaces}
In this section, we will give the proof to the sharp bounds for the collision operator in anisotropic spaces. The main idea is to use the geometric decomposition explained in the introduction and also the $L^2$ profile of the fractional Laplace-Beltrami operator  in Section 5.

\subsection{Proof of Theorem \ref{thmlb}}   
We take the following decomposition: \ben\label{sharpdec}
\langle  -Q(g,f), f\rangle_v&=&-\int_{\R^6}dv_*dv\int_{\SS^2}
B (|v-v_*|,\sigma)g_*f(f'-f)d\sigma\notag \\
&=&-\f12\underbrace{\int_{\R^6}dv_*dv\int_{\SS^2}
B (|v-v_*|,\sigma)g_*(f'^2-f^2)d\sigma}_{\mathcal{L}}\notag \\&&\quad+\f12
\underbrace{\int_{\R^6}dv_*dv
\int_{\SS^2}B (|v-v_*|,\sigma)g_*(f'-f)^2d\sigma}_{\mathcal{E}^\gamma_{g}(f)}.\een

 By change of variables, we have
\beno |\mathcal{L}|&=& |\SS^1|
\bigg|\int_{\R^6}\int_0^\frac{\pi}{2}
\sin\theta\bigg(\f1{\cos^3\frac{\theta}{2}}B(\frac{|v-v_*|}{\cos\frac{\theta}{2}},\cos\theta)-B(|v-v_*|,\cos\theta)\bigg)g_*f^2d\theta
dv_*dv\bigg|\\
&\lesssim& \int_{\R^6} |v-v_*|^\gamma g_*f^2 dv_*dv\eqdefa \mathcal{R}.
  \eeno

We  first show that  $\mathcal{R}$ can be estimated by the following lemma:

\begin{lem}\label{grlimlem3}  Let $\varpi\in (0,1]$. For  smooth functions $g   $ and
$f$, there exists a sufficiently small constant $\eta$ and a universal constant $a\in(0,1)$ such that
\begin{enumerate}
\item if $\gamma\ge 0$,     \beno  |\mathcal{R}|\lesssim \|g\|_{L^1_\gamma}\|f\|_{L^2}^2+\|g\|_{L^1}\|f\|_{L^2_{\gamma/2}}^2,\eeno
 \item if $-2\varpi<\gamma<0$,\beno |\mathcal{R}|\lesssim \eta^{\f{\gamma}{\gamma+2\varpi}}\f{\gamma+2\varpi}{2\varpi} \|f\|_{L^2_{\gamma/2}}^2+\eta \|f\|_{H^\varpi_{\gamma/2}}^2,\eeno
 \item if $\gamma+2\varpi= 0$,   \beno |\mathcal{R}|&\lesssim&\big[\|g\|_{L^1_{|\gamma|}}+ \exp ([\eta^{-1}(1-a)\|g\|_{L\log L}+2\eta^{-1}(1-a)^{-1}\|g\|_{L^1_{2\varpi/a}}]^{\f1{1-a}})\big]\|f\|_{L^2_{\gamma/2}}^2+ \eta\|f\|_{H^\varpi_{\gamma/2}}^2,\eeno
  \item if $-1\le \gamma+2\varpi<0$ and $p> 3/2$, \beno |\mathcal{R}|\lesssim  \eta^{-\f3{(\gamma+2\varpi+3)p-3}} \|g\langle\cdot\rangle^{|\gamma|}\|_{L^p}^{\f{(\gamma+2\varpi+3)p}{(\gamma+2\varpi+3)p-3}}\|f\|_{L^2_{\gamma/2}}^2+\eta \|f\|_{H^\varpi_{\gamma/2}}^2.  \eeno
\end{enumerate}
\end{lem}
\begin{proof}

It is easy to check that if $\gamma>0$, we have \beno  |\mathcal{R}|\lesssim \|g\|_{L^1_\gamma}\|f\|_{L^2}^2+\|g\|_{L^1}\|f\|_{L^2_{\gamma/2}}^2.\eeno

For $\gamma<0$, we observe that
\beno  |\mathcal{R}| &=&\bigg| \iint_{v_*, v\in \R^3} |v-v_*|^\gamma \big(\langle v\rangle \langle v_*\rangle^{-1}\big)^{|\gamma|}  g_*\langle v_*\rangle^{|\gamma|}(f\langle v\rangle^{\gamma/2})^2 dv_*dv\bigg|\\
&=&\bigg| \iint_{v_*, v\in \R^3} |v-v_*|^\gamma \big(\langle v\rangle\langle v_*\rangle^{-1}\big)^{|\gamma|} G_* F^2 dv_*dv\bigg|\\
&\lesssim& \iint_{v_*, v\in \R^3}\f{\langle v-v_*\rangle^{|\gamma|} } {|v-v_*|^{|\gamma|}} |G_*| F^2 dv_*dv , \eeno
where $G=g\langle v\rangle^{|\gamma|}$, $F=f\langle v\rangle^{\gamma/2}$.

In the case of $-2\varpi<\gamma<0$,  by using  Hardy inequality, we get
 \beno |\mathcal{R}|
&\lesssim& \|g\|_{L^1_{|\gamma|}} \|f\|_{H^{|\gamma|/2}_{\gamma/2}}^2.\eeno
Thanks to the interpolation inequality and the condition $\gamma+2\varpi>0$, we  derive that
\beno |\mathcal{R}|\lesssim \eta^{\f{\gamma}{\gamma+2\varpi}}\f{\gamma+2\varpi}{2\varpi} \|g\|_{L^1_{|\gamma|}}^{2\varpi/|\gamma|}\|f\|_{L^2_{\gamma/2}}^2+\eta \|f\|_{H^\varpi_{\gamma/2}}^2.\eeno

For the case of $\gamma+2\varpi=0$, we have
\beno   |\mathcal{R}|&\lesssim& \|G\|_{L^1}\|F\|_{L^2}^2+ M\|F\|_{L^2}^2+\iint_{v,v_*\in\R^3} |v-v_*|^{-2\varpi} \mathrm{1}_{|v-v_*|\le1}(G\mathrm{1}_{|G|\ge M})_* F^2dv_*dv\\&\lesssim&\|g\|_{L^1_{|\gamma|}}\|f\|_{L^2_{\gamma/2}}^2+ M\|f\|_{L^2_{\gamma/2}}^2+\|G\mathrm{1}_{|G|\ge M}\|_{L^1} \|F\|_{H^\varpi}^2.
\eeno
where  the Hardy inequality is used in the last step.
 Choose $a\in (0,1)$,  then we get \beno  \|G\mathrm{1}_{|G|\ge M}\|_{L^1}&\lesssim& (\log M)^{-(1-a)}\|G(\log G)^{1-a}\|_{L^1}\\
&\lesssim& (\log M)^{-(1-a)}[\|g\|_{L\log L}^{1-a}\|g\|_{L^1_{2\varpi/a}}^a+(1-a)^{-1}|\gamma|\|g\|_{L^1_{2\varpi/a}}].
\eeno
It yields \beno |\mathcal{R}|&\lesssim& (M+\|g\|_{L^1_{|\gamma|}})\|f\|_{L^2_{\gamma/2}}^2+ (\log M)^{-(1-a)}[\|g\|_{L\log L}^{1-a}\|g\|_{L^1_{2\varpi/a}}^{a}+(1-a)^{-1}|\gamma|\|g\|_{L^1_{2\varpi/a}}]\|f\|_{H^\varpi_{\gamma/2}}^2.\eeno
Thus we have
 \beno |\mathcal{R}|&\lesssim&\big[\|g\|_{L^1_{|\gamma|}}+ \exp ([\eta^{-1}\|g\|_{L\log L}^{1-a}\|g\|_{L^1_{2\varpi/a}}^a+\eta^{-1}(1-a)^{-1}|\gamma|\|g\|_{L^1_{2\varpi/a}}]^{\f1{1-a}})\big]\|f\|_{L^2_{\gamma/2}}^2+ \eta\|f\|_{H^\varpi_{\gamma/2}}^2.\eeno
In particular, if $\varpi\in (0,1)$, then  for $a\in (\varpi,1)$,
\beno |\mathcal{R}|&\lesssim& \big[\|g\|_{L^1_2}+\exp ([\eta^{-1}(1-a)\|g\|_{L\log L}+\eta^{-1}(1-a)^{-1}|\gamma|\|g\|_{L^1_{2}})]^{\f1{1-a}})\big]\|f\|_{L^2_{-\varpi}}^2+ \eta\|f\|_{H^\varpi_{-\varpi}}^2
. \eeno
If $\varpi=1$,  then for any $\delta>0$, \beno 
|\mathcal{R}|&\lesssim& \big[\|g\|_{L^1_2}+\exp ([\eta^{-1}\|g\|_{L\log L}+\eta^{-1}\delta^{-1}\|g\|_{L^1_{2+\delta}})]^{\f{2+\delta}{\delta}})\big]\|f\|_{L^2_{-1}}^2+ \eta\|f\|_{H^\varpi_{-1}}^2.\eeno

Finally we handle the case $-1<\gamma+2\varpi<0$. By Hardy-Littlewood-Sobolev inequality, we have
\beno |\mathcal{R}|&\lesssim& M\|f\|_{L^2_{\gamma/2}}^2+\|G\mathrm{1}_{|G|\ge M}\|_{L^{\f3{\gamma+2\varpi+3}}} \|F\|_{H^\varpi}^2.\eeno
Since $p> 3/2$, we have
\beno \|G\mathrm{1}_{|G|\ge M}\|_{L^{\f3{\gamma+2\varpi+3} }}\lesssim M^{-\f{(\gamma+2\varpi+3)p-3}3}\|G\|_{L^p}^{\f{(\gamma+2\varpi+3)p}3}, \eeno
Thus we get
\beno |\mathcal{R}|\lesssim  \eta^{-\f3{(\gamma+2\varpi+3)p-3}} \|g \|_{L^p_{|\gamma|}}^{\f{(\gamma+2\varpi+3)p}{(\gamma+2\varpi+3)p-3}}\|f\|_{L^2_{\gamma/2}}^2+\eta \|f\|_{H^\varpi_{\gamma/2}}^2. \eeno
We complete the proof of the lemma.
\end{proof}

From now on, we focus on the estimate of the elliptic part $\mathcal{E}^\gamma_{g}$. We begin with two useful lemmas to deal with the simple case $\gamma=0$ and then extend the results to the general cases.
\medskip

\begin{lem}\label{coermaxw1} Suppose  $g$ is a   non-negative  and smooth function. Then  for any $\eta>0$,
\ben\label{cm1} \mathcal{E}^0_g(f)&\gtrsim& \mathcal{C}_4(g)\big(\|(-\triangle_{\SS^2})^{s/2} f\|_{L^2}^2+
 \|f\|_{H^s}^2\big) -\eta\|g\|_{L^1_{s}}\|f\|_{L^2_s}^2-(\|g\|_{L^1_{s}}\eta^{-1}\mathcal{C}_3(g)^{-1}+1)\|f\|_{L^2}^2, \een
where $$\mathcal{C}_4(g)\eqdefa\frac{\min\{\mathcal{C}_3(g),\|g\|_{L^1_{-s}}\}}{\|g\|_{L^1_{s}}\eta^{-1}\mathcal{C}_3(g)^{-1}+2}.$$
Here $\mathcal{C}_3(g)\eqdefa \min\{\mathcal{C}_1(g),\mathcal{C}_2(g),1 \}$ and  $\mathcal{C}_1(g)$ and $\mathcal{C}_2(g)$ are defined as follows:
\beno \mathcal{C}_1(g)&\eqdefa& 2\sin^2 \varepsilon\bigg[ |g|_{L^1}-\f{|g|_{L^1_1}}{r}-\sup_{|A|<4\varepsilon(2r)^2+\f{2\varepsilon}{\pi}(2r)^3 }\int_{A} g(v)dv\bigg],\eeno
where $\varepsilon$ and $r$ are chosen in such a way that this quantity is positive, and
\beno
\mathcal{C}_2(g)&\eqdefa& 2\vartheta^2\inf_{|\xi|\le1} \bigg|\f{\sin^2(\vartheta|\xi|)}{\vartheta^2|\xi|^2}\bigg|\bigg[ |g|_{L^1}-\f{|g|_{L^1_1}}{r}-\sup_{|A|<4\vartheta(2r)^3(\f1\pi+\f{1}r)}\int_{A} g(v)dv\bigg],\eeno
where $\vartheta$ and $r$ are chosen in such a way that this quantity is positive.
 
If the function $g$ verifies the condition \eqref{lbc}, then due to the definition of $\mathcal{C}_4(g)$, there exists a constant $C(\delta, \lambda,\eta^{-1})$   such that
\ben\label{cm2} \mathcal{C}_4(g)\ge C(\delta, \lambda,\eta^{-1}). \een
\end{lem}

\begin{proof} 
By the geometric decomposition \eqref{geodecom} with  $u=r\tau$ and $\varsigma=\f{\sigma+\tau}{|\sigma+\tau|}\in \SS^2$, one has
\ben\label{defe20} \mathcal{E}^0_g(f)&\ge&  \f12\iint_{u,v_*\in\R^3,\sigma\in\SS^2} g_* b(\cos\theta)\big((T_{v_*}f)(r\varsigma)-(T_{v_*}f)(r\tau)\big)^2  d\sigma dv_*du\notag\\
&&\quad-
 \iint_{u,v_*\in\R^3,\sigma\in\SS^2} g_* b(\cos\theta)\big(f(v_*+u^+)-f(v_*+|u|\f{u^+}{|u^+|})\big)^2 d\sigma dv_*du\notag \\
&\eqdefa& \mathcal{E}^0_1-\mathcal{E}^0_2. \een

{\it Step 1: Estimate of $\mathcal{E}^0_1$. }
By change of variables, we have
\beno  \mathcal{E}^0_1\gtrsim \iint_{r>0,\tau,\sigma\in\SS^2,v_*\in\R^3} g_* b(\sigma\cdot \tau)\mathrm{1}_{\sigma\cdot \tau\ge0}\big((T_{v_*}f)(r\varsigma)-(T_{v_*}f)(r\tau))^2 r^2 d\sigma d\tau drdv_*. \eeno
For fixed $v_*$, $\tau\in \SS^2$ and $r$, if $\tau$ is chosen to be the polar direction, one has
\beno d\sigma= \sin\theta d\theta d\SS^1, d\varsigma=\sin \phi d\phi d\SS^1, \eeno
where $\theta=2\phi$. We deduce that
\beno d\sigma =4\cos\phi d\varsigma. \eeno
Thanks to the facts $b(\tau\cdot\sigma)\sim |\sigma-\tau|^{-(2+2s)}$ and $ |\sigma-\tau| \sim|\varsigma-\tau| $, we get
\beno  \mathcal{E}^0_1&\gtrsim& \iint_{v_*\in\R^3,r>0,\tau,\varsigma\in\SS^2} g_* |\varsigma-\tau|^{-(2+2s)}\big((T_{v_*}f)(r\varsigma)-(T_{v_*}f)(r\tau))^2 (4\varsigma\cdot\tau)\mathrm{1}_{|\varsigma-\tau|^2\le 2-\sqrt{2}} r^2 d\varsigma d\tau drdv_*\\
&\gtrsim& \iint_{v_*\in\R^3,r>0,\tau,\varsigma\in\SS^2} g_* |\varsigma-\tau|^{-(2+2s)}\big((T_{v_*}f)(r\varsigma)-(T_{v_*}f)(r\tau))^2\mathrm{1}_{|\varsigma-\tau|^2\le 2-\sqrt{2}} r^2 d\varsigma d\tau drdv_*\eeno
Thanks to Lemma \ref{antin2} and Lemma \ref{antin3}, we obtain that 
\beno \mathcal{E}^0_1&\gtrsim&  \int_{\R^3} g_*\|(-\triangle_{\SS^2})^{s/2}T_{v_*}f\|_{L^2}^2dv_*-\|g\|_{L^1}\|f\|_{L^2}^2\\
&\gtrsim&  \|g\|_{L^1_{-2s}} \|(-\triangle_{\SS^2})^{s/2} f\|_{L^2}^2-\|g\|_{L^1}\|f\|_{H^s}^2. \eeno

{\it Step 2: Estimate of $\mathcal{E}^0_2$. }  We introduce the dyadic decomposition in the frequency space.
Set \beno
\mathcal{E}^0_{2,k}(g,f)&\eqdefa& 2
 \iint_{u,v_*\in\R^3,\sigma\in\SS^2} g_* \varphi_k(u) b(\cos\theta)\big(f(v_*+u^+)-f(v_*+|u|\f{u^+}{|u^+|})\big)^2 d\sigma dv_*du \\
 &=& 2\sum_{l,p=-1}^\infty
 \iint_{u,v_*\in\R^3,\sigma\in\SS^2} g_* \varphi_k(u) b(\cos\theta) \big((\mathfrak{F}_{l}f)(v_*+u^+)-(\mathfrak{F}_{l}f)(v_*+|u|\f{u^+}{|u^+|})\big)
 \\&&\times\big((\mathfrak{F}_{p}f)(v_*+u^+)-(\mathfrak{F}_{p}f)(v_*+|u|\f{u^+}{|u^+|})\big)  d\sigma dv_*du\\
 &=& 2 \big(\sum_{l\le p}\mathcal{E}_{l,p}+\sum_{l> p}\mathcal{E}_{l,p}\big). \eeno
 By the symmetric property of $\mathcal{E}_{l,p}$, without loss of the generality, we assume $l\le p$.
 It is easy to check that \ben\label{elp}\mathcal{E}_{l,p}=\mathcal{E}_{l,p}^1+\mathcal{E}_{l,p}^2,  \een
where
\beno  \mathcal{E}_{l,p}^1&=&\iint_{u,v_*\in\R^3,\sigma\in\SS^2} g_* \varphi_k(u) b(\cos\theta) \big((\mathfrak{F}_{l}f)(v_*+u^+)-(\mathfrak{F}_{l}f)(v_*+|u|\f{u^+}{|u^+|})\big)
 \\&&\quad\times (\mathfrak{F}_{p}f)(v_*+u^+)   d\sigma dv_*du,\\
 \mathcal{E}_{l,p}^2&=&\iint_{u,v_*\in\R^3,\sigma\in\SS^2} g_* \varphi_k(u) b(\cos\theta) \big((\mathfrak{F}_{l}f)(v_*+u^+)-(\mathfrak{F}_{l}f)(v_*+|u|\f{u^+}{|u^+|})\big)
 \\&&\quad\times (\mathfrak{F}_{p}f)(v_*+|u|\f{u^+}{|u^+|}) d\sigma dv_*du.
  \eeno

{ \it Step 2.1:  Estimate of $\mathcal{E}_{l,p}^1$.} We introduce the function $\psi$ to split $\mathcal{E}_{l,p}^1$ into two parts:
$\mathcal{E}_{l,p}^{1,1}$ and $\mathcal{E}_{l,p}^{1,2}$ which are defined by
\beno \mathcal{E}_{l,p}^{1,1}&=&\iint_{u,v_*\in\R^3,\sigma\in\SS^2} g_* \varphi_k(u) b(\cos\theta) \big((\mathfrak{F}_{l}f)(v_*+u^+)-(\mathfrak{F}_{l}f)(v_*+|u|\f{u^+}{|u^+|})\big)
 \\&&\quad\times (\mathfrak{F}_{p}f)(v_*+u^+)\psi(2^{k/2}2^{l/2}\sqrt{1-\f{u}{|u|}\cdot \sigma})   d\sigma dv_*du,\\
   \mathcal{E}_{l,p}^{1,2}&=&\iint_{u,v_*\in\R^3,\sigma\in\SS^2} g_* \varphi_k(u) b(\cos\theta) \big((\mathfrak{F}_{l}f)(v_*+u^+)-(\mathfrak{F}_{l}f)(v_*+|u|\f{u^+}{|u^+|})\big)
 \\&&\quad\times (\mathfrak{F}_{p}f)(v_*+u^+)\big(1-\psi(2^{k/2}2^{l/2}\sqrt{1-\f{u}{|u|}\cdot \sigma})\big)   d\sigma dv_*du.\eeno
Observe that \ben &&\big|(\mathfrak{F}_{l}f)(v_*+u^+)-(\mathfrak{F}_{l}f)(v_*+|u|\f{u^+}{|u^+|})\big|\nonumber\\
&& \lesssim
\int_0^1d\kappa |\na (\mathfrak{F}_{l}f)|(v_*+u^+(\kappa+(1-\kappa)\cos^{-1}(\theta/2)))|u^+||1-\cos^{-1}(\theta/2)|.\label{taylor3}
\een
We have \beno |\mathcal{E}_{l,p}^{1,1}|&\lesssim& \bigg(\iint_{\kappa\in[0,1],u,v_*\in\R^3,\sigma\in\SS^2} |g_*| \varphi_k(u) b(\cos\theta) |\na (\mathfrak{F}_{l}f)|^2(v_*+u^+(\kappa+(1-\kappa)\cos^{-1}(\theta/2)))\\
&&\qquad\times|u^+| |1-\cos^{-1}\theta| \mathrm{1}_{|\theta|\lesssim 2^{-k/2}2^{-l/2}}d\kappa d\sigma dv_*du
\bigg)^{\f12} \\
&&\quad\times  \bigg(\iint_{\kappa\in[0,1],u,v_*\in\R^3,\sigma\in\SS^2}| g_* |\varphi_k(u) b(\cos\theta) |u^+| |1-\cos^{-1}(\theta/2)||(\mathfrak{F}_{p}f)(v_*+u^+)|^2\\
&&\qquad\times\mathrm{1}_{|\theta|\lesssim 2^{-k/2}2^{-l/2}}d\kappa d\sigma dv_*du\bigg)^{\f12}.
\eeno
Let $\tilde{u}=u^+(\kappa+(1-\kappa)\cos^{-1}(\theta/2))$. Then  by change of the variable from $u$ to $\tilde{u}$,  one gets
\ben\label{Jacobi3} |\f{d \tilde{u}}{du}|=|\f{d \tilde{u}}{du^+}||\f{d u^+}{du}|\sim 1. \een
Moreover, we have $|u|\sim |\tilde{u}|$. Thanks to this observation, we get
\beno |\mathcal{E}_{l,p}^{1,1}|&\lesssim& 2^{ k}\bigg(\iint_{\kappa\in[0,1],\tilde{u},v_*\in\R^3,\sigma\in\SS^2} |g_*|  |\na (\mathfrak{F}_{l}f)|^2(v_*+\tilde{u})b(\cos\theta)\theta^{2}
 \mathrm{1}_{|\theta|\lesssim 2^{-k/2}2^{-l/2}}d\kappa d\sigma dv_*d\tilde{u}
\bigg)^{\f12} \\
&&\quad\times  \bigg(\iint_{\kappa\in[0,1],\tilde{u},v_*\in\R^3,\sigma\in\SS^2} |g_*|   |(\mathfrak{F}_{p}f)(v_*+\tilde{u})|^2 b(\cos\theta)\theta^{2} \mathrm{1}_{|\theta|\lesssim 2^{-k/2}2^{-l/2}}d\kappa d\sigma dv_*d\tilde{u}\bigg)^{\f12}\\
&\lesssim& 2^{(s-1)k}2^{(s-1)l}2^k\|g\|_{L^1}\|\mathfrak{F}_{p}f\|_{L^2}\|\na \mathfrak{F}_{l}f\|_{L^2}\\&\lesssim&
 2^{(l-p)s/2}2^{ks}\|g\|_{L^1}\|\mathfrak{F}_{p}f\|_{H^{s/2}}\|\mathfrak{F}_{l}f\|_{H^{s/2}}.
\eeno

Now we turn to   the estimate of $\mathcal{E}_{l,p}^{1,2}$. 
 Follow the argument applied to $\mathfrak{D}_k^{1,2}$ in Lemma \ref{lemub1}, then we have 
\beno |\mathcal{E}_{l,p}^{1,2}| \lesssim  2^{(l-p)s/2}2^{ks}\|g\|_{L^1}\|\mathfrak{F}_{p}f\|_{H^{s/2}}\|\mathfrak{F}_{l}f\|_{H^{s/2}},\eeno
which implies
\beno |\mathcal{E}_{l,p}^{1}|\lesssim  2^{(l-p)s/2}2^{ks}\|g\|_{L^1}\|\mathfrak{F}_{p}f\|_{H^{s/2}}\|\mathfrak{F}_{l}f\|_{H^{s/2}}. \eeno

{\it Step 2.2: Estimate of $\mathcal{E}_{l,p}$.} The similar argument can be applied to  $\mathcal{E}_{l,p}^2$ to get
\beno |\mathcal{E}_{l,p}^{2}|\lesssim  2^{(l-p)s/2}2^{ks}\|g\|_{L^1}\|\mathfrak{F}_{p}f\|_{H^{s/2}}\|\mathfrak{F}_{l}f\|_{H^{s/2}}, \eeno
which yields
\beno |\mathcal{E}_{l,p}|\lesssim   2^{(l-p)s/2}2^{ks}\|g\|_{L^1}\|\mathfrak{F}_{p}f\|_{H^{s/2}}\|\mathfrak{F}_{l}f\|_{H^{s/2}}.   \eeno
\smallskip

We arrive at
\beno |\mathcal{E}^0_{2,k}(g,f)|\le 2(\sum_{l\le p} |\mathcal{E}_{l,p}|+\sum_{l\ge p} |\mathcal{E}_{l,p}|)\lesssim 2^{ks}\|g\|_{L^1}\| f\|_{H^{s/2}}^2. \eeno

\bigskip
Suppose  $|v_*|\sim 2^j$ and $|u|\sim 2^k$. Then thanks to the fact $|u|\sim |u^+|$, we have
\begin{itemize}
\item Case 1: $j\le k-N_0$. Then $|v_*+u^+|, |v_*+|u|\f{u^+}{|u^+|}|\sim  2^k$;
\item Case 2: $j\ge k+N_0$. Then $|v_*+u^+|, |v_*+|u|\f{u^+}{|u^+|}|\sim  2^j$;
\item Case 3: $|j-k|< N_0$. Then $|v_*+u^+|, |v_*+|u|\f{u^+}{|u^+|}|\le 2^{k+N_0}, |v'|\le 2^{k+N_0}.$
\end{itemize}
We get
\beno \mathcal{E}^0_{2} &= & \sum_{k=-1}^\infty\mathcal{E}^0_{2,k}(g, f)\\
&=&  \sum_{k<j-N_0}\mathcal{E}^0_{2,k}( \mathcal{P}_j g,\tilde{\mathcal{P}}_j f)   +\sum_{j<k-N_0}\mathcal{E}^0_{2,k}( \mathcal{P}_j g,\tilde{\mathcal{P}}_kf) +\sum_{k=-1}^\infty \mathcal{E}^0_{2,k}(\tilde{\mathcal{P}}_k g, \mathcal{U}_{k+N_0}f).\eeno
Then \beno  |\mathcal{E}^0_{2}|&\lesssim& \sum_{k<j-N_0}  2^{ks}\|\mathcal{P}_j g\|_{L^1}\| \tilde{\mathcal{P}}_j f\|_{H^{s/2}}^2  +\sum_{j<k-N_0} 2^{ks}\|\mathcal{P}_j g\|_{L^1}\|\tilde{\mathcal{P}}_k f\|_{H^{s/2}}^2\\&&\qquad +\sum_{k=-1}^\infty2^{ks}\|\tilde{\mathcal{P}}_kg\|_{L^1}\|\mathcal{U}_{k+N_0} f\|_{H^{s/2}}^2 \\
&\lesssim& \|g\|_{L^1_s}\|f\|_{H^{s/2}_{s/2}}^2,
 \eeno
where we use Theorem \ref{baslem3}.
\medskip

Patch together the estimates of $\mathcal{E}^0_{1}$ and $\mathcal{E}^0_{2}$, then we finally get
\beno \mathcal{E}^0_g(f)&\gtrsim&  \|g\|_{L^1_{-2s}} \|(-\triangle_{\SS^2})^{s/2} f\|_{L^2}^2-\|g\|_{L^1_s}(\|f\|_{H^s}^2+ \|f\|_{H^{s/2}_{s/2}}^2)\\
&\gtrsim&\|g\|_{L^1_{-2s}} \|(-\triangle_{\SS^2})^{s/2} f\|_{L^2}^2-\|g\|_{L^1_s}(\eta^{-1}\|f\|_{H^s}^2+\eta \|f\|_{L^2_s}^2).\eeno
Thanks to Corollary 3 and Proposition 2 in \cite{advw}, we deduce that
\beno \mathcal{E}^0_g(f)+\|f\|_{L^2}^2\gtrsim \mathcal{C}_3(g) \|f\|_{H^s}^2. \eeno
From which together with the previous lower bound, we are led to the desired result. 
\end{proof}

\begin{lem} \label{comparelem}  Suppose the angular function $b$ verifies the conditions $ \int_ 0^{\pi/2} b(\cos\theta)\sin\theta \theta^2 d\theta<\infty$ and
\ben\label{asymsym}  
&& 1+\int_{\sigma\in\SS^2} b(\tau\cdot\sigma)\min\{|\xi|^2|\tau-\sigma|^2,1\}d\sigma\sim  1+\int_{\sigma\in\SS^2} b(2(\tau\cdot\sigma)^2-1)\min\{|\xi|^2|\tau-\sigma|^2,1\}d\sigma
\sim W^2(\xi),\een 
where $\tau\in \SS^2$
and $W$ is a radial function satisfying  $W(|\xi||\zeta|)\lesssim W(|\xi|)W(|\zeta|)$ and $W(\xi)\le \langle \xi\rangle$.
Then for any smooth function $g$, it holds
\beno  | \mathcal{E}^0_g(f)| \lesssim \|g\|_{L^1} \mathcal{E}^0_\mu(f)+\|W^2g\|_{L^1}\|W(D)f\|_{L^2}^2,\eeno
 
If $g$ is a non-negative function verifying the condition \eqref{lbc},  then there exist constants $C(\lambda,\delta)$ and $C(\lambda)$ such that
\beno  C(\lambda,\delta)\mathcal{E}^0_\mu(f)-C(\lambda)\|f\|_{L^2}^2\lesssim \mathcal{E}^0_g(f)\lesssim   C(\lambda)(\mathcal{E}^0_\mu(f)+\|f\|_{L^2}^2),\eeno
in other words, $\mathcal{E}^0_\mu(f)+\|f\|_{L^2}^2\sim \mathcal{E}^0_g(f)+\|f\|_{L^2}^2.$
\end{lem}

\begin{rmk} We remark that \eqref{asymsym} holds under the assumption \eqref{a2} or \eqref{abc2} or
\eqref{ab2}.   Moreover we have
\beno W(\xi)=\left\{\begin{aligned}    &  \langle \xi\rangle^s, \quad \mbox{under the assumption \eqref{a2};} \\ &\psi(\epsilon \xi)\langle \xi\rangle^s+\epsilon^{-s}(1-\psi(\epsilon \xi)), 
 \quad \mbox{under the  assumption \eqref{abc2};} \\& \psi(\epsilon \xi)\langle \xi\rangle+\epsilon^{s-1}(1-\psi(\epsilon \xi))\langle \xi\rangle^{s},  \quad \mbox{under the assumption \eqref{ab2}.}  \end{aligned}\right. 
\eeno  We recall that the function $\psi$ is defined in \eqref{defpsivarphi}. It is easy to check that for all the cases, the symbol function $W$ satisfies the properties: $W(|\xi||\zeta|)\lesssim W(|\xi|)W(|\zeta|)$ and $W(\xi)\le \langle \xi\rangle$. \end{rmk}

\begin{proof} The proof is inspired by \cite{amuxy1}. Without loss of generality, we assume that the function $g$ is non-negative. By Bobylev's formula \eqref{bobylev},
 we have
 \ben\label{boby} &&  \mathcal{E}^0_g(f)\notag\\&&=\f1{(2\pi)^3}\iint_{\xi\in\R^3,\sigma\in\SS^2} b(\f{\xi}{|\xi|}\cdot \sigma)\bigg(\hat{g}(0)|\hat{f}(\xi)-\hat{f}(\xi^+)|^2+
 2\mathrm{Re}\big( (\hat{g}(0)-\hat{g}(\xi^-))\hat{f}(\xi^+)\bar{\hat{f}} (\xi)\big)\bigg)d\xi d\sigma.\een
 We recall that $\xi^-=\f{\xi-|\xi|\sigma}2$ and $\xi^+=\f{\xi+|\xi|\sigma}2$.
 
It implies
\beno  \|\mu\|_{L^1}\mathcal{E}^0_g(f)=\|g\|_{L^1}\big(\mathcal{E}^0_\mu(f)-\f{1}{(2\pi)^3}I_1)+\f{2\|\mu\|_{L^1}}{(2\pi)^3}I_2,\eeno
where
\beno  I_1=\iint_{\xi\in\R^3,\sigma\in\SS^2} b(\f{\xi}{|\xi|}\cdot \sigma)\mathrm{Re}\big( (\hat{\mu}(0)-\hat{\mu}(\xi^-))\hat{f}(\xi^+)\bar{\hat{f}} (\xi) \big)d\sigma d\xi,\eeno
and
\beno I_2&=&\iint_{\xi\in\R^3,\sigma\in\SS^2} b(\f{\xi}{|\xi|}\cdot \sigma)\mathrm{Re}\big( (\hat{g}(0)-\hat{g}(\xi^-))(\hat{f}(\xi^+)-\hat{f}(\xi))\bar{\hat{f}} (\xi)\big)d\xi d\sigma\\&&+\iint_{\xi\in\R^3,\sigma\in\SS^2} b(\f{\xi}{|\xi|}\cdot \sigma)\mathrm{Re}(\hat{g}(0)-\hat{g}(\xi^-))|\hat{f}(\xi)|^2d\xi d\sigma\\&\eqdefa&
I_{2,1}+I_{2,2}.\eeno

Thanks to the fact $\hat{\mu}(0)-\hat{\mu}(\xi^-)=\int_{ \R^3} (1-\cos (v\cdot \xi^-))\mu(v)dv$, we have
\beno  |I_1|&=&\bigg|\iint_{v,\xi\in\R^3,\sigma\in\SS^2} b(\f{\xi}{|\xi|}\cdot \sigma)(1-\cos (v\cdot \xi^-))\mu(v)\mathrm{Re}(\hat{f}(\xi^+)\bar{\hat{f}} (\xi))d\sigma d\xi dv \bigg|\\&\lesssim&
\bigg(\iint_{v,\xi\in\R^3,\sigma\in\SS^2} b(\f{\xi}{|\xi|}\cdot \sigma)(1-\cos (v\cdot \xi^-))\mu(v)|\hat{f}(\xi^+)|^2d\sigma d\xi dv\bigg)^{1/2}\\&&\times
\bigg(\iint_{v,\xi\in\R^3,\sigma\in\SS^2} b(\f{\xi}{|\xi|}\cdot \sigma)(1-\cos (v\cdot \xi^-))\mu(v)|\hat{f}(\xi)|^2d\sigma d\xi dv\bigg)^{1/2}.
\eeno
Observe that \beno (1-\cos (v\cdot \xi^-))\lesssim |v|^2|\xi^-|^2 \le |v|^2|\xi|^2|\f{\xi}{|\xi|}-\sigma|^2 \sim |v|^2|\xi^+|^2|\f{\xi^+}{|\xi^+|}-\sigma|^2 \eeno
and \beno \f{\xi}{|\xi|}\cdot \sigma=2 (\f{\xi^+}{|\xi^+|}\cdot\sigma)^2-1.\eeno
Then by change of the variable from $\xi$ to $\xi^+$, the assumption \eqref{asymsym} and the property  $W(|\xi||\zeta|)\lesssim W(|\xi|)W(|\zeta|)$, we have
\beno  |I_1|&\lesssim& \iint_{v,\xi\in\R^3} W^2(|v||\xi|)|\hat{f}(\xi)|^2\mu(v)dvd\xi
\\
&\lesssim&  \|W^2\mu\|_{L^1}\|W(D)f\|_{L^2}^2.\eeno
 
 Notice that $\mathrm{Re} (\hat{g}(0)-\hat{g}(\xi^-))=\int_{ \R^3}  (1-\cos (v\cdot \xi^-))g(v)dv$. The similar argument can be applied to get
\beno |I_{2,2}|\lesssim \|W^2g\|_{L^1}\|W(D)f\|_{L^2}^2 \eeno

Next, by Cauchy-Schwartz inequality, one has
\beno |I_{2,1}|&\lesssim& \big(\iint_{\xi\in\R^3,\sigma\in\SS^2} b(\cos\theta)|\hat{g}(0)-\hat{g}(\xi^-)|^2|\hat{f}(\xi)|^2d\sigma d\xi\big)^{\f12}\\&&\times \bigg(\iint_{\xi\in\R^3,\sigma\in\SS^2} b(\cos\theta)|\hat{f}(\xi)-\hat{f}(\xi^+)|^2d\sigma d\xi\bigg)^\f12\eqdefa (I_{2,1}^1)^\f12 (I_{2,1}^2)^\f12.\eeno
Observe that $\hat{g}(0)-\hat{g}(\xi^-)=\int_{\R^3} (1-e^{-iv\cdot \xi^-})g(v)dv$, then it holds
\beno I_{2,1}^1&\lesssim& \iint_{v,w,\xi\in\R^3,\sigma\in\SS^2} b(\cos\theta)g(v)g(w)(|1-e^{-iv\cdot \xi^-}|^2+|1-e^{-iw\cdot \xi^-}|^2)|\hat{f}(\xi)|^2d\sigma d\xi dv dw\\
&\lesssim&\|g\|_{L^1}\|W^2g\|_{L^1}\|W(D)f\|_{L^2}^2. \eeno
Thanks to \eqref{boby}, we have
\beno \f{1}{(2\pi)^3}\|\mu\|_{L^1}I_{2,1}^2=\mathcal{E}^0_\mu(f)-\f{1}{(2\pi)^3}I_1, \eeno
which implies
\beno  I_{2,1}^2\lesssim \mathcal{E}^0_\mu(f)+\|W^2\mu\|_{L^1}\|W(D)f\|_{L^2}^2.\eeno
We get
\beno |I_{2,1}|\lesssim \eta \|g\|_{L^1} \mathcal{E}^0_\mu(f)+\eta^{-1} (\|W^2g\|_{L^1}+\|g\|_{L^1})\|W(D)f\|_{L^2}^2.  \eeno

Combining the above estimates, we arrive at \ben\label{comequi}  &&\|\mu\|_{L^1}\mathcal{E}^0_g(f)-C(\eta) \|W^2g\|_{L^1}\|W(D)f\|_{L^2}^2\notag\\&&
\lesssim (1-\eta)\|g\|_{L^1}\mathcal{E}^0_\mu(f) \lesssim \|\mu\|_{L^1}\mathcal{E}^o_g(f)+C(\eta) \|W^2g\|_{L^1} \|W(D)f\|_{L^2}^2, \een
which is enough to derive the first inequality in the lemma. Moreover, if the function $g$ verifies the condition \eqref{lbc},  then by the computation in \cite{advw} and the assumption \eqref{asymsym}, we have
\ben\label{comequi1}  \mathcal{E}^0_g(f)+\|f\|_{L^2}^2\gtrsim \mathcal{C}_3(g)\|W(D)f\|_{L^2}^2.\een From which together with \eqref{comequi}, we get the equivalence in the lemma.
\end{proof}
 
In the next lemma, we will show that the lower bound of $\mathcal{E}^\gamma_g(f)$ can be reduced to the lower bound of $ \mathcal{E}^0_\mu(W_{\gamma/2}f)$.
\begin{lem}\label{reducmaxwellian} Suppose that the angular function $b$ verifies the same conditions in Lemma \ref{comparelem} and $g$ is a non-negative function verifying the condition \eqref{lbc}. Then there exists a constant $C(\lambda, \delta)$ such that 
\beno  \mathcal{E}^0_\mu(W_{\gamma/2}f)\le C(\lambda, \delta)\big(\|f\|^2_{L^2_{\gamma/2}}+\mathcal{E}^\gamma_g(f)\big).\eeno
\end{lem}

\begin{proof}
 Let $\chi$ be a radial and smooth function such that $0\le \chi\le 1$, $\chi=1$ on $B_1$ and $\mbox{Supp} \chi \subset B_2$. We set $\chi_R(v)=\chi(v/R)$. We recall the notation: $W_l(v)=\langle v\rangle^l$. 
 
{\it Case 1: $|v|$ is sufficiently large.} 
It is easy to check
\beno \mathcal{E}^\gamma_{g}(f)&\gtrsim&  \iint_{v,v_*\in\R^3,\sigma\in\SS^2} |v-v_*|^\gamma (g\chi_{\f{R}8})_* b(\cos\theta)(f'-f)^2(1-\chi_R)^2 d\sigma dv_*dv\\&\gtrsim&  \iint_{v,v_*\in\R^3,\sigma\in\SS^2} W_{\gamma/2}^2 (g\chi_{\f{R}8})_* b(\cos\theta)(f'-f)^2(1-\chi_R)^2 d\sigma dv_*dv.\eeno
Thanks to the inequality $(a-b)^2\ge \f12a^2-b^2$, we obtain that
\beno \mathcal{E}^\gamma_{g}(f)&\ge&\f12 \iint_{v,v_*\in\R^3,\sigma\in\SS^2} (g\chi_{\f{R}8})_* b(\cos\theta)\big( (W_{\gamma/2}(1-\chi_R)f)'-W_{\gamma/2}( 1-\chi_R)f\big)^2 d\sigma dv_*dv\\
&&\quad-2 \iint_{v,v_*\in\R^3,\sigma\in\SS^2} (g\chi_{\f{R}8})_* b(\cos\theta)f'^2\big( (W_{\gamma/2}(1-\chi_R))'  -W_{\gamma/2}(1-\chi_R)\big)^2 d\sigma dv_*dv. \eeno

Suppose that $
\kappa(v)=v+\kappa(v'-v)$ with $\kappa\in[0,1]$. It is easy to check that $\f{\sqrt{2}}2|v-v_*|\le |v'-v_*|\le|\kappa(v)-v_*|\le |v-v_*|$. Since now   $|v_*|\le R/4$, then if $|v|\ge R$, we have $|v|\sim |\kappa(v)|\sim |v-v_*|$. Similarly if $|v'|\ge R$, we have  $|v'|\sim |\kappa(v)|\sim |v-v_*|$. Then in both cases, we have $|v'|\sim |v-v_*|\sim |\kappa(v)|$.
By the Mean Value Theorem, we get
\beno &&\bigg|\iint_{v,v_*\in\R^3,\sigma\in\SS^2} (g\chi_{\f{R}8})_* b(\cos\theta)f'^2\big((W_{\gamma/2}(1-\chi_R))'  -W_{\gamma/2}(1-\chi_R)\big)^2 d\sigma dv_*dv\bigg|\\
&&\quad\lesssim \iint_{v,v_*\in\R^3,\sigma\in\SS^2} (g\chi_{\f{R}8})_* b(\cos\theta)f'^2 \langle \kappa(v) \rangle^{\gamma-2}|v-v_*|^2\theta^2 \mathrm{1}_{|v'|\sim |v-v_*|\sim |\kappa(v)|} d\sigma dv_*dv
\\
&&\quad\lesssim \iint_{v,v_*\in\R^3,\sigma\in\SS^2} (g\chi_{\f{R}8})_* b(\cos\theta)\theta^2f'^2W_{\gamma}' d\sigma dv_*dv\\
&&\quad\lesssim  \|g\|_{L^1}\|f\|_{L^2_{\gamma/2}}^2.\eeno
Thus we arrive at
\beno \iint_{v,v_*\in\R^3,\sigma\in\SS^2} (g\chi_{\f{R}8})_* b(\cos\theta)\big((W_{\gamma/2}(1-\chi_R)f)'-W_{\gamma/2}( 1-\chi_R)f\big)^2 d\sigma dv_*dv\lesssim \|g\|_{L^1}\|f\|_{L^2_{\gamma/2}}^2+ \mathcal{E}^\gamma_{g}(f), \eeno  
that is, \ben\label{l2}  \mathcal{E}^0_{g\chi_{\f{R}8}}( ( 1-\chi_R)W_{\gamma/2}f)\lesssim \|g\|_{L^1}\|f\|_{L^2_{\gamma/2}}^2+ \mathcal{E}^\gamma_{g}(f). \een

{\it Case 2: $|v|$ is bounded.} Let $A,B$ be the subsets in $B_{3R}$. We denote $\chi_A$ and $\chi_B$  by the mollified characteristic functions corresponding to the sets $A$ and $B$. Then it yields
\ben\label{l1}&&\iint_{v,v_*\in\R^3,\sigma\in\SS^2} b(\cos\theta)(g\chi_B)_*f'^2(\chi_A'-\chi_A)^2d\sigma dv_*dv\notag\\&&\lesssim  \iint_{v,v_*\in\R^3,\sigma\in\SS^2} b(\cos\theta)(g\chi_B)_*f'^2(\chi_A'-\chi_A)^2\mathrm{1}_{|v-v_*|\le 8R}\mathrm{1}_{|v'|\le 8R}d\sigma dv_*dv\notag\\
&&\lesssim   \|\na (\chi_A)\|_{L^\infty}^2\iint_{v,v_*\in\R^3,\sigma\in\SS^2a} b(\cos\theta)(g\chi_B)_*f'^2|v-v_*|^2\theta^2 \mathrm{1}_{|v-v_*|\le 8R}\mathrm{1}_{|v'|\le 8R}d\sigma dv_*dv\notag\\
&&\lesssim    R^2\|\na (\chi_A)\|_{L^\infty}^2\iint_{v,v_*\in\R^3,\sigma\in\SS^2} b(\cos\theta)(g\chi_B)_*f'^2\mathrm{1}_{|v'|\le 8R}\theta^2  d\sigma dv_*dv
 \notag\\&&\lesssim \|\na (\chi_A)\|_{L^\infty}^2 R^{2}\max \{R^{-\gamma},1\} \|g\|_{L^1}\|f\|_{L^2_{\gamma/2}}^2.
 \een

With the help of Lemma 2.1 in \cite{advw} and replacing (35) in \cite{advw} by \eqref{l1}, we conclude that if $\gamma<0$,
\beno   R^\gamma \iint_{v,v_*\in\R^3,\sigma\in\SS^2} (g\chi_{\f{R}8})_* b(\cos\theta)\big( (\chi_Rf)'-\chi_Rf\big)^2 d\sigma dv_*dv\lesssim  \|g\|_{L^1}\|f\|_{L^2_{\gamma/2}}^2+ \mathcal{E}^\gamma_{g}(f),
\eeno
and if $\gamma>0$,
\beno r_0^\gamma\iint_{v,v_*\in\R^3,\sigma\in\SS^2} (g\chi_{B_j})_* b(\cos\theta)\big( (\chi_{A_j}f)'-\chi_{A_j}f\big)^2 d\sigma dv_*dv\lesssim r_0^{-2 }R^2\|g\|_{L^1}\|f\|_{L^2_{\gamma/2}}^2+ \mathcal{E}^\gamma_{g}(f),\eeno
where $\chi_{A_j}= \chi(\frac{v-v_j}{r_0}), \chi_{B_j}=\chi_{3R}- \chi(\frac{v-v_j}{3r_0})$  with $v_j\in B_{2R}$ and $r_0$ will be chosen later.
Notice that \beno  (\chi_Af)'-\chi_Af=\big((\chi_AW_{\gamma/2}f)'-(\chi_AW_{\gamma/2}f)\big)W_{-\gamma/2}+(\chi_AW_{\gamma/2}f)'\big((W_{-\gamma/2})'-W_{-\gamma/2}\big).\eeno
By a slight modification, we may derive that  if $\gamma<0$,
\ben\label{l3}  && \iint_{v,v_*\in\R^3,\sigma\in\SS^2} (g\chi_{\f{R}8})_* b(\cos\theta)\big( (\chi_RW_{\gamma/2}f )'-(\chi_RW_{\gamma/2}f)\big)^2 d\sigma dv_*dv\notag\\&&\lesssim R^{2-2\gamma }\|g\|_{L^1}\|f\|_{L^2_{\gamma/2}}^2+ R^{-\gamma}\mathcal{E}^\gamma_{g}(f),
\een
and if $\gamma>0$,
\ben\label{l4} &&\iint_{v,v_*\in\R^3,\sigma\in\SS^2} (g\chi_{B_j})_* b(\cos\theta)\big( (\chi_{A_j}W_{\gamma/2}f)'-(\chi_{A_j}W_{\gamma/2}f )\big)^2 d\sigma dv_*dv\notag\\&&\lesssim r_0^{-2}R^{4+\gamma}\|g\|_{L^1}\|f\|_{L^2_{\gamma/2}}^2+ R^\gamma r_0^{-\gamma}\mathcal{E}^\gamma_{g}(f).\een 

By finite covering theorem, there exists an integer $N$ such that
\ben\label{l5}  B_{2R}\subset \bigcup_{j=1}^{N} \{|v-v_j|\le r_0\}\quad \mbox{and}\quad N\sim \bigg(\f{R}{r_0}\bigg)^3,\een
where $v_j\in B_{2R}$.
 Observe that \beno
\| g\chi_{R/8}\|_{L^1}\ge \|g\|_{L^1}-R^{-1}\|g\|_{L^1_1} \eeno
and \beno \| g\chi_{B_j}\|_{L^1}\ge \|g\|_{L^1}-(3R)^{-1}\|g\|_{L^1_1}-M(6r_0)^3-(\log M)^{-1}\|g\|_{L\log L}.  \eeno
Then by choosing $R=\f{4\lambda}{3\delta}+1$, $M=e^{4\lambda/\delta}$ and $r_0=\f16e^{-4\lambda/(3\delta)}$, we get
\beno N\sim 6^3(\f{4\lambda}{3\delta}+1)^3e^{4\lambda/\delta}\eqdefa C_1(\delta, \lambda)\eeno  and \beno \| g\chi_{R/8}\|_{L^1}\ge \delta/4,\, \| g\chi_{B_j}\|_{L^1}\ge \delta/4. \eeno
Then there exists a constant $C(\lambda, \delta)$ such that   \beno \mathcal{C}_3(g\chi_{R/8})\ge C(\lambda, \delta), \quad \mathcal{C}_3(g\chi_{B_j})\ge C(\lambda, \delta). \eeno
 Thanks to \eqref{comequi} and \eqref{comequi1} in the proof of Lemma \ref{comparelem}, we may rewrite (\ref{l2}-\ref{l4}) as:
 \ben\label{l6} \mathcal{E}^0_{\mu}( ( 1-\chi_R)W_{\gamma/2}f)&\lesssim& C_2(\lambda, \delta) \big( \|f\|_{L^2_{\gamma/2}}^2+ \mathcal{E}^\gamma_{g}(f)\big),\\
 \label{l7} \mathcal{E}_\mu^0(\chi_{A_j}W_{\gamma/2}f)&\lesssim& C_3(\lambda,\delta)\big( \|f\|_{L^2_{\gamma/2}}^2+ \mathcal{E}^\gamma_{g}(f)\big),\,\quad\mbox{if}\,\, \gamma>0,
 \\ \label{l8}\mathcal{E}_\mu^0(\chi_{R}W_{\gamma/2}f)&\lesssim& C_4(\lambda,\delta)\big( \|f\|_{L^2_{\gamma/2}}^2+ \mathcal{E}^\gamma_{g}(f)\big),\,\quad\mbox{if}\,\, \gamma<0.
\een

We conclude that \eqref{l6} and \eqref{l8} yield the desired result for  soft potentials.  For $\gamma>0$,
thanks to the facts \beno 
\mathcal{E}_\mu^0(\chi_{A_j}W_{\gamma/2}f)&\ge&\f12\iint_{\sigma\in\SS^2, v,v_*\in\R^3} \mu_*b(\cos\theta)\chi_{A_j}^2\big((W_{\gamma/2}f)'-(W_{\gamma/2}f)\big)^2d\sigma dv_*dv\\&&-\iint_{\sigma\in\SS^2, v,v_*\in\R^3} \mu_*b(\cos\theta)(\chi_{A_j}'-\chi_{A_j})^2\big( (W_{\gamma/2}f)'\big)^2 d\sigma dv_*dv, \eeno
and 
\beno  \mu_*(\chi_{A_j}'-\chi_{A_j})^2(\mathrm{1}_{|v_*|\le 8R}+\mathrm{1}_{|v_*|\ge 8R})\lesssim  \mu_*(\chi_{A_j}'-\chi_{A_j})^2(\mathrm{1}_{|v-v_*|\le 15R}+\mathrm{1}_{|v_*-v|\sim |v_*|}),\eeno
 we have  
\beno&&\iint_{\sigma\in\SS^2, v,v_*\in\R^3} \mu_*b(\cos\theta)\chi_{A_j}^2\big((W_{\gamma/2}f)'-(W_{\gamma/2}f)\big)^2d\sigma dv_*dv\\&&
\lesssim C_5(\lambda,\delta)\big( \|f\|_{L^2_{\gamma/2}}^2+ \mathcal{E}^\gamma_{g}(f)\big). \eeno
From which together with \eqref{l5} and \eqref{l6}, we are led to the desired result for   hard potentials.
We complete the proof of the lemma. \end{proof}
\medskip

We are now in a position to complete the proof of Theorem \ref{thmlb}. 
\begin{proof}The desired results  are easily derived from \eqref{sharpdec}, Lemma \ref{grlimlem3}, Lemma \ref{coermaxw1} and Lemma \ref{reducmaxwellian}. \end{proof}

\subsection{Proof of  Theorem \ref{entropyproduction} }  Finally we give the proof to Theorem \ref{entropyproduction}.

\begin{proof} Following the computation in \cite{advw}, we first have if $\gamma\ge0$,
\beno  D_B(f)&=&\iint_{v,v_*\in\R^3,\sigma\in\SS^2} |v-v_*|^\gamma b(\cos\theta) f_*(f\ln\f{f}{f'}-f+f' )dvdv_*d\sigma\\
&&-\iint_{v,v_*\in\R^3,\sigma\in\SS^2} |v-v_*|^\gamma b(\cos\theta) (f-f') dvdv_*d\sigma\\
&\ge&\iint_{v,v_*\in\R^3,\sigma\in\SS^2} |v-v_*|^\gamma b(\cos\theta)f_*(\sqrt{f'}-\sqrt{f})^2dvdv_*d\sigma\\
&&-\iint_{v,v_*\in\R^3,\sigma\in\SS^2} |v-v_*|^\gamma b(\cos\theta) f_*(f-f') dvdv_*d\sigma\\
&\ge& \mathcal{E}^{\gamma}_f(\sqrt{f})-\|f\|_{L^1}\|f\|_{L^1_2}\gtrsim \mathcal{E}^{0}_\mu(W_{\gamma/2}\sqrt{f})-\|f\|_{L^1_2}^2, \eeno
where we use the inequality $x\ln \f{x}{y}-x+y\ge (\sqrt{x}-\sqrt{y})^2$ and Lemma \ref{reducmaxwellian}.   

For   soft potentials($\gamma<0$), we observe that 
\beno  D_B(f)&=&\f14\iint_{v,v_*\in\R^3,\sigma\in\SS^2} |v-v_*|^\gamma b(\cos\theta) (f'f'_*-ff_*)\ln\f{f_*'f'}{f_*f}dvdv_*d\sigma \\
&\ge&\f14\iint_{v,v_*\in\R^3,\sigma\in\SS^2}  \langle v-v_*\rangle^\gamma b(\cos\theta) (f'f'_*-ff_*)\ln\f{f_*'f'}{f_*f}dvdv_*d\sigma\\
&\ge&\iint_{v,v_*\in\R^3,\sigma\in\SS^2}  \langle v-v_*\rangle^\gamma b(\cos\theta) f_*(\sqrt{f'}-\sqrt{f})^2dvdv_*d\sigma\\
&&-\iint_{v,v_*\in\R^3,\sigma\in\SS^2}  \langle v-v_*\rangle^\gamma b(\cos\theta) f_*(f-f') dvdv_*d\sigma\\
&\gtrsim& \mathcal{E}^{0}_\mu(W_{\gamma/2}\sqrt{f})-\|f\|_{L^1}^2.  \eeno
The last inequality is deduced from the proof of Lemma \ref{reducmaxwellian}. We complete the proof of the theorem with the help of Lemma \ref{coermaxw1}. \end{proof}

\subsection{Proof of Theorem \ref{thumbanti}}  Now we are ready to give the    proof to Theorem \ref{thumbanti}. 
\begin{proof} To get sharp bounds for the Boltzmann collision operator in anisotropic spaces, we only need to give the new estimates to $\mathfrak{W}_{k,l}^2$ and  $\mathfrak{W}_{k,p}^3$ due to the geometric decomposition \eqref{geodecom} for $k\ge0$. Recalling that  $u=r\tau$ and $\varsigma=\f{\sigma+\tau}{|\sigma+\tau|}\in \SS^2$, we
have
\beno \mathfrak{W}_{k,p}^3
 &=& \mathfrak{W}_{k,p}^{3,1}+\mathfrak{W}_{k,p}^{3,2},\eeno
  where \beno
 \mathfrak{W}_{k,p}^{3,1}&=&\iint_{\sigma\in \SS^2,v_*,u\in \R^3} \Phi_k^\gamma(|u|)b(\sigma\cdot\tau)  (\mathfrak{F}_{p}g)_*(T_{v_*}\tilde{\mathfrak{F}}_{p}h)(r\tau)\\&&\quad\times\big(
 (T_{v_*}\tilde{\mathfrak{F}}_{p}f)(r\varsigma)-(T_{v_*}\tilde{\mathfrak{F}}_{p}f)(r\tau)\big) d\sigma dudv_*,\\
 \mathfrak{W}_{k,p}^{3,2}&=& \iint_{\sigma\in \SS^2,v_*,u\in \R^3} \Phi_k^\gamma(|v-v_*|)b(\cos\theta)  (\mathfrak{F}_{p}g)_*(\tilde{\mathfrak{F}}_{p}h) \\&&\quad\times\big(\big(\tilde{\mathfrak{F}}_{p}f)(v_*+u^+)-(\tilde{\mathfrak{F}}_{p}f)(v_*+|u|\f{u^+}{|u^+|})\big)d\sigma dv_* du.
 \eeno
For the term $\mathfrak{W}_{k,p}^{3,1}$, by change of the variable from $\sigma$ to $\varsigma$, we have
\beno \mathfrak{W}_{k,p}^{3,1}&=&\iint_{\varsigma,\tau\in \SS^2,v_*\in \R^3,r>0} \Phi_k^\gamma(r)b(2(\varsigma\cdot\tau)^2-1)  (\mathfrak{F}_{p}g)_*(T_{v_*}\tilde{\mathfrak{F}}_{p}h)(r\tau)\\&&\quad\times\big(
 (T_{v_*}\tilde{\mathfrak{F}}_{p}f)(r\varsigma)-(T_{v_*}\tilde{\mathfrak{F}}_{p}f)(r\tau)\big)r^24(\varsigma\cdot \tau)d\varsigma dv_* d\tau dr\\
 &=&\iint_{v_*\in \R^3,r>0} dr dv_*\Phi_k^\gamma(r) (\mathfrak{F}_{p}g)_* r^2\iint_{\varsigma, \tau \in \SS^2}b\big( 2(\varsigma\cdot\tau)^2-1\big)  4(\varsigma\cdot \tau) \\&&\quad\times (T_{v_*}\tilde{\mathfrak{F}}_{p}h)(r\tau)\big(
 (T_{v_*}\tilde{\mathfrak{F}}_{p}f) (r\varsigma)-(T_{v_*}\tilde{\mathfrak{F}}_{p}f)(r\tau)\big)d\varsigma d\tau .  \eeno
Thanks to Corrolary \ref{antin6} and Lemma \ref{antin3},  for $a,b\in[0,2s]$ with $a+b=2s$, we get
\beno |\mathfrak{W}_{k,p}^{3,1}|&\lesssim&2^{\gamma k} \int_{\R^3} |(\mathfrak{F}_{p}g)_*| \|(1-\triangle_{\SS^2})^{a/2}(T_{v_*}\tilde{\mathfrak{F}}_{p}h)\|_{L^2}
\|(1-\triangle_{\SS^2})^{b/2}(T_{v_*}\tilde{\mathfrak{F}}_{p}f)\|_{L^2}dv_*\\
&\lesssim& 2^{\gamma k}\|\mathfrak{F}_{p}g\|_{L^1_{2s}}  (\|(-\triangle_{\SS^2})^{a/2}\tilde{\mathfrak{F}}_{p}h\|_{L^2}+\|\tilde{\mathfrak{F}}_{p}h\|_{H^a})
(\|(-\triangle_{\SS^2})^{b/2}\tilde{\mathfrak{F}}_{p}f\|_{L^2}+\|\tilde{\mathfrak{F}}_{p}f\|_{H^b}).\eeno
From which together with Lemma \ref{antin7}, we deduce that
\beno \sum_{p=-1}^\infty |\mathfrak{W}_{k,p}^{3,1}|&\lesssim& 2^{\gamma k}\|g\|_{L^1_{2s}}(\|(-\triangle_{\SS^2})^{a/2} h\|_{L^2}+\| h\|_{H^a})
(\|(-\triangle_{\SS^2})^{b/2} f\|_{L^2}+\| f\|_{H^b}).  \eeno

For the term $\mathfrak{W}_{k,p}^{3,2}$, we may follow the argument used to bound $\mathcal{E}_{l,p}$(see \eqref{elp}) to get  for $k\ge0$,
\beno |\mathfrak{W}_{k,p}^{3,2}|\lesssim 2^{(\gamma+s) k}2^{sp} \|g\|_{L^1}\|\tilde{\mathfrak{F}}_{p}h\|_{L^2}\|\tilde{\mathfrak{F}}_{p}f\|_{L^2},  \eeno
which implies
\beno  \sum_{p=-1}^\infty|\mathfrak{W}_{k,p}^{3,2}|\lesssim 2^{(\gamma+s) k} \|g\|_{L^1}\|h\|_{H^{a_1}}\|f\|_{H^{b_1}},\eeno where  $a_1,b_1\in\R$ with $a_1+b_1=s$.

We finally arrive at for $k\ge0$,
\beno \sum_{p=-1}^\infty |\mathfrak{W}_{k,p}^{3}|+\sum_{l=-1}^\infty|\mathfrak{W}_{k,l}^{2}| &\lesssim& 2^{\gamma k}\|g\|_{L^1_{2s}}(\|(-\triangle_{\SS^2})^{a/2} h\|_{L^2}+\| h\|_{H^a})
(\|(-\triangle_{\SS^2})^{b/2} f\|_{L^2}+\| f\|_{H^b}) \\&&\quad+2^{(\gamma+s) k} \|g\|_{L^1}\|h\|_{H^{a_1}}\|f\|_{H^{b_1}}. \eeno
Thanks to Lemma \ref{lemub1} and Lemma \ref{lemub2}, we also have for $k\ge0$,
 \beno \sum_{l\le p-N_0} |\mathfrak{W}_{k,p,l}^1|+\sum_{m<p-N_0} |\mathfrak{W}_{k,p,m}^4|&\lesssim&   2^{\gamma k}\|g\|_{L^1}\|h\|_{H^{a }}\|f\|_{H^{b }}.\eeno

Now we are in a position to prove the sharp bounds.
We conclude that  for $k\ge0$, \ben\label{estqk3}  |\langle Q_k(g, h), f \rangle_v|  &\lesssim&   2^{(\gamma+s) k} \|g\|_{L^1}\|h\|_{H^{a_1}}\|f\|_{H^{b_1}}\notag\\&&+2^{\gamma k}\|g\|_{L^1_{2s}}(\|(-\triangle_{\SS^2})^{a/2} h\|_{L^2}+\| h\|_{H^a})
(\|(-\triangle_{\SS^2})^{b/2} f\|_{L^2}+\| f\|_{H^b}), \een
and
\ben\label{estqk4}  |\langle Q_{-1}(g, h), f \rangle_v|  &\lesssim&  (\|g\|_{L^1}+\|g\|_{L^2})\|h\|_{H^{a}}\|f\|_{H^{b}}. \een

Recalling \eqref{ubdecom}, we rewrite it by
\beno \langle Q(g, h), f \rangle_v
&=&\sum_{k\ge N_0-1}\langle Q_k(\mathcal{U}_{k-N_0} g, \tilde{\mathcal{P}}_kh), \tilde{\mathcal{P}}_kf \rangle_v +
\sum_{j\ge k+N_0}\langle Q_k(\mathcal{P}_{j} g, \tilde{\mathcal{P}}_jh), \tilde{\mathcal{P}}_jf \rangle_v\notag\\&&\quad+\sum_{|j-k|\le N_0}\langle Q_k( \mathcal{P}_{j} g, \mathcal{U}_{k+N_0}h), \mathcal{U}_{k+N_0}f \rangle_v\\&=&
 \mathfrak{U}_4+\mathfrak{U}_5+\mathfrak{U}_6. \eeno
Thanks to \eqref{estqk3} and \eqref{estqk4}, we can give the estimates term by term .

Suppose $w_1,w_2\in\R$ with $w_1+w_2=\gamma+s$. It is not difficult to check
\beno |\mathfrak{U}_4|&\lesssim& \sum_{k\ge N_0-1} \bigg(
   2^{(\gamma+s) k} \|\mathcal{U}_{k-N_0} g\|_{L^1}\|\tilde{\mathcal{P}}_kh \|_{H^{a_1}}\|\tilde{\mathcal{P}}_kf\|_{H^{b_1}}+2^{\gamma k}\|\mathcal{U}_{k-N_0} g\|_{L^1_{2s}}\notag\\&&\quad\times(\|(-\triangle_{\SS^2})^{a/2} \tilde{\mathcal{P}}_kh \|_{L^2}+\| \tilde{\mathcal{P}}_kh \|_{H^a})
(\|(-\triangle_{\SS^2})^{b/2} \tilde{\mathcal{P}}_kf)\|_{L^2}+\| \tilde{\mathcal{P}}_kf\|_{H^b})\bigg).\eeno
From which together with Theorem \ref{baslem3},  we get 
\beno  |\mathfrak{U}_4|&\lesssim&  \|g\|_{L^1_{2s}} \bigg((\|(-\triangle_{\SS^2})^{a/2}  h\|_{L^2_{\gamma/2}}+\|  h\|_{H^a_{\gamma/2}})\\&&\quad\times
(\|(-\triangle_{\SS^2})^{b/2} f)\|_{L^2_{\gamma/2}}+\| f\|_{H^b_{\gamma/2}})+ \|h\|_{H^{a_1}_{w_1}}\|f\|_{H^{b_1}_{w_2}}
\bigg). \eeno

For the term $\mathfrak{U}_5$, it holds
\beno |\mathfrak{U}_5|&\lesssim& \sum_{j\ge k+N_0,k\ge0} \bigg(
  2^{(\gamma+s) k} \|\mathcal{P}_{j} g\|_{L^1}\|\tilde{\mathcal{P}}_jh \|_{H^{a_1}}\|\tilde{\mathcal{P}}_jf\|_{H^{b_1}}+2^{\gamma k}\|\mathcal{P}_{j} g\|_{L^1_{2s}}\notag\\&&\quad\times(\|(-\triangle_{\SS^2})^{a/2} \tilde{\mathcal{P}}_jh \|_{L^2}+\| \tilde{\mathcal{P}}_jh \|_{H^a})
(\|(-\triangle_{\SS^2})^{b/2} \tilde{\mathcal{P}}_jf)\|_{L^2}+\| \tilde{\mathcal{P}}_jf\|_{H^b})\bigg)\\
&&+\sum_{j\ge N_0-1}(\|\mathcal{P}_{j}g\|_{L^1}+\|\mathcal{P}_{j} g\|_{L^2})\|\tilde{\mathcal{P}}_j h\|_{H^{a}}\| \tilde{\mathcal{P}}_j f\|_{H^{b}}.\eeno
Thanks to Theorem \ref{baslem3}, we obtain that
\begin{enumerate}
\item
if $\gamma>0$
\beno  |\mathfrak{U}_5|&\lesssim& (\|g\|_{L^1_{2s}}+\| g\|_{L^2})\bigg((\|(-\triangle_{\SS^2})^{a/2}  h\|_{L^2_{\gamma/2}}+\|  h\|_{H^a_{\gamma/2}})\\&&\quad\times
(\|(-\triangle_{\SS^2})^{b/2} f)\|_{L^2_{\gamma/2}}+\| f\|_{H^b_{\gamma/2}})+ \|h\|_{H^{a_1}_{w_1}}\|f\|_{H^{b_1}_{w_2}}
\bigg). \eeno
\item if $\gamma=0$, for any $\delta>0$,
\beno  |\mathfrak{U}_5|&\lesssim& (\|g\|_{L^1_{2s+\delta}} +\|g\|_{L^2})\bigg((\|(-\triangle_{\SS^2})^{a/2}  h\|_{L^2}+\|  h\|_{H^a})\\&&\quad\times
(\|(-\triangle_{\SS^2})^{b/2} f)\|_{L^2}+\| f\|_{H^b})+ \|h\|_{H^{a_1}_{w_1}}\|f\|_{H^{a_2}_{ w_2}}
\bigg). \eeno
\item if $\gamma<0$,
\beno  |\mathfrak{U}_5|&\lesssim& (\|g\|_{L^1_{-\gamma+2s}}+\| g\|_{L^2_{-\gamma}})\bigg((\|(-\triangle_{\SS^2})^{a/2}  h\|_{L^2_{\gamma/2}}+\|  h\|_{H^a_{\gamma/2}})\\&&\quad\times
(\|(-\triangle_{\SS^2})^{b/2} f)\|_{L^2_{\gamma/2}}+\| f\|_{H^b_{\gamma/2}})+ \|h\|_{H^{a_1}_{w_1}}\|f\|_{H^{b_1}_{w_2}}
\bigg). \eeno
\end{enumerate}

Finally we turn to   the estimate of $\mathfrak{U}_6$. One has
\beno |\mathfrak{U}_6|&\lesssim& \sum_{|j-k|\le N_0,k\ge0}\bigg(
  2^{(\gamma+s) k} \| \mathcal{P}_{j}  g\|_{L^1}\| \mathcal{U}_{k+N_0}h \|_{H^{a_1}}\| \mathcal{U}_{k+N_0}f\|_{H^{b_1}}+2^{\gamma k}\|\mathcal{P}_{j}  g\|_{L^1_{2s}}\notag\\&&\quad\times(\|(-\triangle_{\SS^2})^{a/2} \mathcal{U}_{k+N_0}h \|_{L^2}+\| \mathcal{U}_{k+N_0}h \|_{H^a})
(\|(-\triangle_{\SS^2})^{b/2}  \mathcal{U}_{k+N_0}f)\|_{L^2}+\| \mathcal{U}_{k+N_0}f\|_{H^b})\bigg)\\
&&+(\|\tilde{\mathcal{P}}_{-1} g\|_{L^1}+\|\tilde{\mathcal{P}}_{-1} g\|_{L^2})\|\mathcal{U}_{N_0} h\|_{H^{a}}\| \mathcal{U}_{N_0} f\|_{H^{b}}.\eeno
Then by Lemma \ref{antin7} and \eqref{u4}, we have
\begin{enumerate}
\item
if $\gamma>0$
\beno  |\mathfrak{U}_6|&\lesssim& (\|g\|_{L^1_{\gamma+2s}}+\|g\|_{L^1_{\gamma+s+(-w_1)^++(-w_2)^+}}+\| g\|_{L^2})\bigg((\|(-\triangle_{\SS^2})^{a/2}  h\|_{L^2_{\gamma/2}}+\|  h\|_{H^a_{\gamma/2}})\\&&\quad\times
(\|(-\triangle_{\SS^2})^{b/2} f)\|_{L^2_{\gamma/2}}+\| f\|_{H^b_{\gamma/2}})+ \|h\|_{H^{a_1}_{w_1}}\|f\|_{H^{b_1}_{w_2}}
\bigg). \eeno
\item if $\gamma=0$,
\beno  |\mathfrak{U}_6|&\lesssim& (\|g\|_{L^1_{2s}}+\|g\|_{L^1_{s+(-w_1)^++(-w_2)^+}}+\| g\|_{L^2})\bigg((\|(-\triangle_{\SS^2})^{a/2}  h\|_{L^2}+\|  h\|_{H^a})\\&&\quad\times
(\|(-\triangle_{\SS^2})^{b/2} f)\|_{L^2}+\| f\|_{H^b})+ \|h\|_{H^{a_1}_{w_1}}\|f\|_{H^{b_1}_{w_2}}
\bigg). \eeno
\item if $\gamma<0$,
\beno  |\mathfrak{U}_6|&\lesssim& (\|g\|_{L^1_{-\gamma+2s}}+\|g\|_{L^1_{\gamma+s+(-w_1)^++(-w_2)^+}}+\| g\|_{L^2_{}})\bigg((\|(-\triangle_{\SS^2})^{a/2}  h\|_{L^2_{\gamma/2}}+\|  h\|_{H^a_{\gamma/2}})\\&&\quad\times
(\|(-\triangle_{\SS^2})^{b/2} f)\|_{L^2_{\gamma/2}}+\| f\|_{H^b_{\gamma/2}})+ \|h\|_{H^{a_1}_{w_1}}\|f\|_{H^{b_1}_{w_2}}
\bigg). \eeno
\end{enumerate}

 The  theorem is obtained by patching together all the  estimates to $\mathfrak{U}_4, \mathfrak{U}_5$ and $\mathfrak{U}_6$. We complete the proof of the theorem.
 \end{proof}

\setcounter{equation}{0}

\section{Sharp bounds for the Landau collision operator via grazing collision limit}
In this section, we will show that the strategy used to handle the Boltzmann collision operator  is robust. It can be applied to capture the intrinsic structure of the collision operator in the process of the grazing collision limit. Before giving the estimates, we first introduce the special function $W^\epsilon$ defined by
\ben\label{symboloflimit} W^\epsilon(x)=\psi(\epsilon x)\langle x\rangle+\epsilon^{s-1}(1-\psi(\epsilon x))\langle x\rangle^{s},  \een
which characterizes the symbol of the collision operator in the process of the limit.  We emphasize that the function $\psi$ is defined in \eqref{defpsivarphi}.

We begin with a technical lemma which describes the behavior of the fractional Laplace-Beltrami operator in  the limit. We postpone the proof to the end of Section 5.4.

\begin{lem}\label{lbplim} Suppose  $0<s<1$. For any smooth function $f$ defined in $\SS^2$,  the following equivalences hold:
\beno &&\|f\|_{L^2(\SS^2)}^2+\epsilon^{2s-2}\int_{\sigma,\tau\in \SS^2} \f{|f(\sigma)-f(\tau)|^2}{|\sigma-\tau|^{2+2s}} \mathrm{1}_{|\sigma-\tau|\le \epsilon}d\sigma d\tau\\ &&\sim \|f\|^2_{L^2(\SS^2)}+\|(-\triangle_{\SS^2})^{1/2}\mathbb{P}_{\le\f1\epsilon}f\|^2_{L^2(\SS^2)}
+\epsilon^{2s-2}\|(-\triangle_{\SS^2})^{s/2}\mathbb{P}_{>\f1\epsilon}f\|^2_{L^2(\SS^2)}\\
&&\sim \|f\|^2_{L^2(\SS^2)}+\|\big((-\triangle_{\SS^2})^{1/2}\mathbb{P}_{\le\f1\epsilon}
+\epsilon^{s-1} (-\triangle_{\SS^2})^{s/2}\mathbb{P}_{>\f1\epsilon}\big)f\|^2_{L^2(\SS^2)},
\eeno
where the projection operators $\mathbb{P}_{\le\f1\epsilon}$ and $\mathbb{P}_{>\f1\epsilon}$ are defined as follows: if $f(\sigma)=\sum\limits_{l=0}^\infty \sum\limits_{m=-l}^lf_l^m Y_l^m(\sigma)$, then
\ben\label{prejector}  &&\big(\mathbb{P}_{\le\f1\epsilon} f\big) (\sigma)\eqdefa\sum_{[l(l+1)]^\f12\le \f1\epsilon} \sum_{m=-l}^l f_l^mY_l^m(\sigma),\nonumber\\&&
\big(\mathbb{P}_{>\f1\epsilon} f\big) (\sigma)\eqdefa\sum_{[l(l+1)]^\f12> \f1\epsilon} \sum_{m=-l}^l f_l^mY_l^m(\sigma).
\een
\end{lem}

\begin{rmk} We remark that the projection operators $\mathbb{P}_{\le\f1\epsilon}$ and $\mathbb{P}_{>\f1\epsilon}$ commutate with the fractional Laplace-Beltrami operator. Moreover, since the  Laplace-Beltrami operator is a self-adjoint operator with orthogonal basis of the eigenfunctions, the spectrum theorem yields that \beno  &&\|f\|^2_{L^2(\SS^2)}+\|\big((-\triangle_{\SS^2})^{1/2}\mathbb{P}_{\le\f1\epsilon}
+\epsilon^{s-1} (-\triangle_{\SS^2})^{s/2}\mathbb{P}_{>\f1\epsilon}\big)f\|^2_{L^2(\SS^2)}\\&&\sim
\|W^\epsilon((-\triangle_{\SS^2})^{1/2}) f\|^2_{L^2(\SS^2)}+\|f\|^2_{L^2(\SS^2)}.
\eeno
\end{rmk}

\subsection{Proof of Theorem \ref{thmblandau}} We are in a position to prove Theorem \ref{thmblandau}.

\begin{proof} In  \cite{he1} , it is proved that for any smooth functions $g, h$ and $f$, there holds
\beno  \lim_{\epsilon\rightarrow 0} \langle Q^\epsilon (g,h), f\rangle =\langle Q_L (g, h), f\rangle, \eeno
where $Q^\epsilon$ is a collision operator with the kernel $B^\epsilon$ under the assumption {\bf{(B1)}}.
  Then the bounds of the Landau operator can be reduced to the uniform bounds of the   operator $Q^\epsilon$ with respect to the parameter $\epsilon$.
Since $Q^\epsilon$ is still the Boltzmann collision operator, we may copy the argument  used in the proof of Theorem \ref{thmub} and Theorem \ref{thumbanti} to get the desired results.

Let us follow the same notations used in Theorem \ref{thmub} and Theorem \ref{thumbanti}.  In the next we only  point out the difference. In order to cancel the singularity caused by the kernel and  get the uniform estimates with respect to the parameter $\epsilon$, we have to make use of the fact:
\beno  \int_{0}^{\pi/2} b^\epsilon (\cos\theta) \sin\theta \theta^2d\theta \sim 1.\eeno  Therefore
there is no need to introduce the   function $\psi$ to make the decomposition for the term $\mathfrak{D}_k^1$ in the proof of Lemma \ref{lemub1},  the terms $\mathfrak {E}^1_k$ and $\mathfrak {E}^2_k$ in the proof of Lemma \ref{lemub2},  the term $\mathfrak{M}^3_{k,p}$ in the proof of Lemma \ref{lemub3}  and the term $\mathcal{E}_{l,p}$ in the proof of Lemma \ref{coermaxw1}. 
Keep it in mind  and follow almost the same calculation, then we get that  the  results stated in Lemma \ref{lemub1},  Lemma  \ref{lemub2} and  Lemma \ref{lemub3}  are valid for $Q^\epsilon$ with $s=1$. In particular, 
we have  \ben\label{ubl1}  |\mathcal{E}_{l,p}|\lesssim 2^{(l-p)/2}2^k\|g\|_{L^1}\|\mathfrak{F}_pf\|_{H^{1/2}}\|\mathfrak{F}_lf\|_{H^{1/2}},
\een where $\mathcal{E}_{l,p}$ is defined in \eqref{elp}.
 With these in hand, following the same argument used in the proof of Theorem \ref{thmub} will yield the desired results (\ref{u1L}-\ref{u3L}).

Next we turn to the  upper bounds of the Landau operator in   anisotropic spaces.  Let $a,b\in[0,2]$ with $a+b=2$ and $a_1,b_1\in\R$ with $a_1+b_1=1$. We first give the bounds to $\mathfrak{W}_{k,p,l}^1,  \mathfrak{W}_{k,p,m}^4,\mathfrak{W}_{k,p}^{3}$ and $ \mathfrak{W}_{k,l}^{2}$.  
Thanks to  the facts that Lemma \ref{lemub1}, Lemma \ref{lemub2}  and  Lemma \ref{lemub3}  are valid  for $Q^\epsilon$ with $s=1$, we deduce that for $k\ge0$,
 \beno \sum_{l\le p-N_0} |\mathfrak{W}_{k,p,l}^1|+\sum_{m<p-N_0} |\mathfrak{W}_{k,p,m}^4|&\lesssim&    \|g\|_{L^1}\|h\|_{L^2}\|f\|_{L^2}.\eeno
and   
\ben\label{thmgrf}|\langle Q^\epsilon_{-1}(g, h), f \rangle_v|  &\lesssim&  (\|g\|_{L^1}+\|g\|_{L^2})\|h\|_{H^{a}}\|f\|_{H^{b}}.\een

 Thanks to  Lemma \ref{lbplim}, we have
 \beno && \epsilon^{2s-2}\int_{\sigma,\tau\in \SS^2} \f{|f(\sigma)-f(\tau)|^2}{|\sigma-\tau|^{2+2s}} \mathrm{1}_{|\sigma-\tau|\le \epsilon}d\sigma d\tau\\ &&\lesssim \|f\|^2_{L^2(\SS^2)}+\|(-\triangle_{\SS^2})^{1/2} f\|^2_{L^2(\SS^2)}, \eeno
which implies that for smooth functions $g$ and $h$,  
\beno   &&\epsilon^{2s-2}\int_{\SS^2\times\SS^2} \big(g(\sigma)-g(\tau)\big)h(\tau) |\sigma-\tau|^{-(2+2s)}\mathrm{1}_{|\sigma-\tau|\le \epsilon} d\sigma d\tau\\&&\lesssim \sum_{l=0}^\infty\sum_{m=-l}^l  g_l^mh_l^m
(\|(-\triangle_{\SS^2})^{1/2}Y^m_l\|_{L^2(\SS^2)}+1)(\|(-\triangle_{\SS^2})^{1/2}Y^m_l\|_{L^2(\SS^2)}+1)
\\&&\lesssim \sum_{l=0}^\infty\sum_{m=-l}^l g_l^mh_l^m (l(l+1)+1)^{2}\\&&\lesssim \|(1-\triangle_{\SS^2})^{a/2}g\|_{L^2(\SS^2)}\|(1-\triangle_{\SS^2})^{b/2}h\|_{L^2(\SS^2)}.\eeno
 From which together with the strategy explained in Section 1.3 and the fact $b^\epsilon(\sigma\cdot\tau)\sim |\sigma-\tau|^{-(2+2s)}\mathrm{1}_{|\sigma-\tau|\le\epsilon}$, we have
\beno |\mathfrak{W}_{k,p}^{3,1}|&\lesssim&2^{\gamma k} \int_{\R^3} |(\mathfrak{F}_{p}g)_*| \|(1-\triangle_{\SS^2})^{a/2}(T_{v_*}\tilde{\mathfrak{F}}_{p}h)\|_{L^2}
\|(1-\triangle_{\SS^2})^{b/2}(T_{v_*}\tilde{\mathfrak{F}}_{p}f)\|_{L^2}dv_*\\
&\lesssim& 2^{\gamma k}\|\mathfrak{F}_{p}g\|_{L^1_{2}}  (\|(-\triangle_{\SS^2})^{a/2}\tilde{\mathfrak{F}}_{p}h\|_{L^2}+\|\tilde{\mathfrak{F}}_{p}h\|_{H^a})
(\|(-\triangle_{\SS^2})^{b/2}\tilde{\mathfrak{F}}_{p}f\|_{L^2}+\|\tilde{\mathfrak{F}}_{p}f\|_{H^b}).\eeno
From which together with Lemma \ref{antin7}, one has
\beno \sum_{p=-1}^\infty |\mathfrak{W}_{k,p}^{3,1}|&\lesssim& 2^{\gamma k}\|g\|_{L^1_{2}}(\|(-\triangle_{\SS^2})^{a/2} h\|_{L^2}+\| h\|_{H^a})
(\|(-\triangle_{\SS^2})^{b/2} f\|_{L^2}+\| f\|_{H^b}).  \eeno
Observing that the term $\mathfrak{W}_{k,p}^{3,2}$ enjoys the similar structure as that of the term $\mathcal{E}_{l,p}$,   \eqref{ubl1} indicates that it is not difficult to derive
\beno |\mathfrak{W}_{k,p}^{3,2}|\lesssim 2^{(\gamma+1) k}2^{p} \|g\|_{L^1}\|\tilde{\mathfrak{F}}_{p}h\|_{L^2}\|\tilde{\mathfrak{F}}_{p}f\|_{L^2},  \eeno
which implies
\beno  \sum_{p=-1}^\infty|\mathfrak{W}_{k,p}^{3,2}|\lesssim 2^{(\gamma+1) k} \|g\|_{L^1}\|h\|_{H^{a_1}}\|f\|_{H^{b_1}}. \eeno
Therefore we have 
\beno    \sum_{p=-1}^\infty |\mathfrak{W}_{k,p}^{3}| &\lesssim& 2^{\gamma k}\|g\|_{L^1_{2}}(\|(-\triangle_{\SS^2})^{a/2} h\|_{L^2}+\| h\|_{H^a})
(\|(-\triangle_{\SS^2})^{b/2} f\|_{L^2}+\| f\|_{H^b}) \\&&\quad+2^{(\gamma+1) k} \|g\|_{L^1}\|h\|_{H^{a_1}}\|f\|_{H^{b_1}}. \eeno 
Because $\mathfrak{W}_{k,l}^{2}$ enjoys the similar structure as that of $\mathfrak{W}_{k,p}^{3}$, we finally arrive at for $k\ge0$,
\beno \sum_{p=-1}^\infty |\mathfrak{W}_{k,p}^{3}|+\sum_{l=-1}^\infty|\mathfrak{W}_{k,l}^{2}| &\lesssim& 2^{\gamma k}\|g\|_{L^1_{2}}(\|(-\triangle_{\SS^2})^{a/2} h\|_{L^2}+\| h\|_{H^a})
(\|(-\triangle_{\SS^2})^{b/2} f\|_{L^2}+\| f\|_{H^b}) \\&&\quad+2^{(\gamma+1) k} \|g\|_{L^1}\|h\|_{H^{a_1}}\|f\|_{H^{b_1}}. \eeno

Due to \eqref{fsdecom}, we conclude that for for $k\ge0$, \beno |\langle Q^\epsilon_k(g, h), f \rangle_v|  &\lesssim&   2^{(\gamma+1) k} \|g\|_{L^1}\|h\|_{H^{a_1}}\|f\|_{H^{b_1}}\notag\\&&+2^{\gamma k}\|g\|_{L^1_{2}}(\|(-\triangle_{\SS^2})^{a/2} h\|_{L^2}+\| h\|_{H^a})
(\|(-\triangle_{\SS^2})^{b/2} f\|_{L^2}+\| f\|_{H^b}). \eeno
From which together with \eqref{thmgrf}, it is enough to derive the upper bounds in   anisotropic spaces by the argument used in the proof of Theorem \ref{thumbanti}. This completes the proof of the theorem. \end{proof}

\subsection{Proof of Theorem \ref{thmlblandau} } We   give the proof to Theorem \ref{thmlblandau}.
\begin{proof} 
 We first focus on the lower bound  of the functional $\langle -Q^\epsilon(g, f), f\rangle$ where $g$ satisfies the condition \eqref{lbc}.  Observe that
\beno  \langle -Q^\epsilon(g, f), f\rangle=\f12\mathcal{E}^{\gamma,\epsilon}_g(f)-\f12\mathcal{L}^{\epsilon}_g(f),\eeno
where
\beno \mathcal{E}^{\gamma,\epsilon}_g(f)&\eqdefa& \iint_{v,v_*\in\R^3,\sigma\in\SS^2} |v-v_*|^\gamma g_* b^\epsilon(\cos\theta)(f'-f)^2  d\sigma dv_*dv,\\
\mathcal{L}^{\epsilon}_g(f)&\eqdefa&\int_{\R^6}dv_*dv\int_{\SS^2}
B^\epsilon (|v-v_*|,\sigma)g_*(f'^2-f^2)d\sigma.
\eeno
By the cancellation lemma, it holds
\beno  |\mathcal{L}^{\epsilon}_g(f)|\lesssim \mathcal{R}.\eeno
We note that this term is controlled by Lemma \ref{grlimlem3}.

Next we concentrate on the term $\mathcal{E}^{\gamma,\epsilon}_g(f)$. Due to Lemma \ref{reducmaxwellian}, it suffices to consider the lower bound of  $\mathcal{E}^{0,\epsilon}_g(f)$.
Thanks to the geometric decomposition  \eqref{geodecom}, we  have \beno \mathcal{E}^{0,\epsilon}_g(f)\ge \mathcal{E}^{0,\epsilon}_{1,g}(f)-\mathcal{E}^{0,\epsilon}_{2,g}(f),\eeno where
\beno   \mathcal{E}^{0,\epsilon}_{1,g}(f)\eqdefa
\f12\iint_{u,v_*\in\R^3,\sigma\in\SS^2}  g_* b^\epsilon(\cos\theta)\big((T_{v_*}f)(r\varsigma)-(T_{v_*}f)(r\tau))^2 d\sigma dv_*du,\\
\mathcal{E}^{0,\epsilon}_{2,g}(f)\eqdefa 2
 \iint_{u,v_*\in\R^3,\sigma\in\SS^2}   g_* b^\epsilon(\cos\theta)\big(f(v_*+u^+)-f(v_*+|u|\f{u^+}{|u^+|})\big)^2 d\sigma dv_*du. \eeno
 \medskip

{\it Step 1: Estimate of $\mathcal{E}^{0,\epsilon}_{1,g}(f)$.}
By the strategy explained in Section 1.3, the fact $b^\epsilon(\sigma\cdot\tau)\sim |\sigma-\tau|^{-(2+2s)}\mathrm{1}_{|\sigma-\tau|\le\epsilon}$ and     Lemma \ref{lbplim}, we first get the lower bound of $\mathcal{E}^{0,\epsilon}_{1,g}(f)$, that is, 
if $g\ge0$, then \beno\mathcal{E}^{0,\epsilon}_{1,g}(f)+\|g\|_{L^1}\|f\|_{L^2}^2\sim \int_{\R^3} g_*\|W^\epsilon((-\triangle_{\SS^2})^{1/2})T_{v_*}f\|_{L^2}^2dv_*+\|g\|_{L^1}\|f\|_{L^2}^2.\eeno

\smallskip

{\it Step 2: Estimate of $\mathcal{E}^{0,\epsilon}_{2,g}(f)$:}
 By a slight modification of the estimate to $\mathcal{E}^{0}_2$(defined in \eqref{defe20}) in the proof of Lemma \ref{coermaxw1}, we can get   
\beno  |\mathcal{E}^{0,\epsilon}_{2,g}(f)|\lesssim \|g\|_{L^1_1}\|f\|_{H^{1/2}_{1/2}}^2.\eeno  We point out that it is  a consequence of   \eqref{ubl1}.

We finally arrive at for $\eta>0$,
\beno  \mathcal{E}^{0,\epsilon}_{g}(f)+\|g\|_{L^1}\|f\|_{L^2}^2&\gtrsim&  
\int_{\R^3} g_*\|W^\epsilon((-\triangle_{\SS^2})^{1/2})T_{v_*}f\|_{L^2}^2dv_*-\|g\|_{L^1_1}(\eta^{-1} \|f\|_{H^{1}}^2+\eta\|f\|_{L^2_1}^2)\\
&\gtrsim & \int_{\R^3} g_*\|W^\epsilon((-\triangle_{\SS^2})^{1/2})T_{v_*}f\|_{L^2}^2dv_*-\|g\|_{L^1_1}(\eta^{-1} \|\psi(\epsilon D)f\|_{H^1}^2\\&&+\eta^{-1} \|(1-\psi(\epsilon D))f\|_{H^1}^2+
\eta\|f\|_{L^2_1}^2).
\eeno
Thanks to the condition \eqref{lbc}, we also have
 \beno  \mathcal{E}^{0,\epsilon}_{g}(f)+\|g\|_{L^1}\|f\|_{L^2}^2\ge C(\lambda, \delta)\|W^\epsilon(D)f\|_{L^2}^2,\eeno
 which was proven in \cite{he1}.
 Combining these two inequalities, we obtain that
\beno &&\mathcal{E}^{0,\epsilon}_{g}(f)+\|g\|_{L^1_1}(\eta^{-1} \|(1-\psi(\epsilon D))f\|_{H^{1}}^2+\eta^{-1}\|f\|_{L^2}^2+\eta\|f\|_{L^2_1}^2)\\&&\ge  C(\lambda, \delta, \eta)\big(\|W^\epsilon(D)f\|_{L^2}^2+
\int_{\R^3} g_*\|W^\epsilon((-\triangle_{\SS^2})^{1/2})T_{v_*}f\|_{L^2}^2dv_*\big).  \eeno

Thanks to the condition \eqref{lbc} and Lemma \ref{reducmaxwellian}, we have
\ben\label{coerLandau} &&C(\lambda, \delta, \eta) \big( \int_{\R^3} \mu_*\|W^\epsilon((-\triangle_{\SS^2})^{1/2})T_{v_*}W_{\gamma/2}f\|_{L^2}^2dv_*+\|W^\epsilon(D)W_{\gamma/2}f\|_{L^2}^2\big)\notag\\
&&\lesssim C(\delta,\lambda)\big(\mathcal{E}^{\gamma,\epsilon}_{g}(f)+\eta^{-1}(\|f\|_{L^{2}_{\gamma/2}}^2+\|(1-\psi(\epsilon D))W_{\gamma/2}f\|_{H^{1}}^2)+\eta\|f\|_{L^2_{\gamma/2+1}}^2\big) .\een 
Noticing that for any  smooth function $f$, we have
\beno    \langle -Q_L(g,f),f\rangle_v=\lim_{\epsilon\rightarrow0} \langle -Q^\epsilon(g, f), f\rangle\ge 
\lim_{\epsilon\rightarrow0}\f12\mathcal{E}^{\gamma,\epsilon}_g(f)- C\mathcal{R}.\eeno
Then the results are easily obtained by Fatou Lemma, \eqref{coerLandau}, Lemma \ref{grlimlem3} and Lemma \ref{antin3}. \end{proof}

\subsection{Proof of Theorem \ref{thmeplandau}} Finally we give the proof to Theorem \ref{thmeplandau}.
\begin{proof}  It is proved in \cite{villani1} (see (55) in \cite{villani3}) that for any smooth function $f$,
$ \lim\limits_{\epsilon\rightarrow0}D_B^\epsilon(f)=D_L(f)$. By the proof of Theorem  \ref{entropyproduction} and Lemma  \ref{reducmaxwellian}, we have  \beno  D_B^\epsilon(f) 
&\gtrsim& \mathcal{E}_\mu^{0,\epsilon}(W_{\gamma/2}\sqrt{f})- \|f\|_{L^1_2}. \eeno
From the proof of Theorem \ref{thmlblandau}, we have that for any $\eta>0$,
\beno &&\mathcal{E}^{0,\epsilon}_{\mu}(W_{\gamma/2}\sqrt{f})+ (\eta^{-1} \|(1-\psi(\epsilon D))W_{\gamma/2}\sqrt{f}\|_{H^{1}}^2+\eta^{-1}\|W_{\gamma/2}\sqrt{f}\|_{L^2}^2+\eta\|W_{\gamma/2}\sqrt{f}\|_{L^2_1}^2)\\&&\gtrsim  C(\lambda, \delta, \eta)\big(\|W^\epsilon(D)W_{\gamma/2}\sqrt{f}\|_{L^2}^2+
\int_{\R^3} \mu_*\|W^\epsilon((-\triangle_{\SS^2})^{1/2})T_{v_*}W_{\gamma/2}\sqrt{f}\|_{L^2}^2dv_*\big).
\eeno
In other words,
\beno 
&&D_B^\epsilon(f)+\|f \|_{L^1_{\gamma+2}}+\|f\|_{L^1_2}+ \|(1-\psi(\epsilon D))W_{\gamma/2}\sqrt{f}\|_{H^{1}}^2\\
&&\gtrsim    C(\lambda, \delta)\big(\|W^\epsilon(D)W_{\gamma/2}\sqrt{f}\|_{L^2}^2+
\int_{\R^3} \mu_*\|W^\epsilon((-\triangle_{\SS^2})^{1/2})T_{v_*}W_{\gamma/2}\sqrt{f}\|_{L^2}^2dv_*\big).
\eeno
Thanks to Fatou Lemma and Lemma \ref{antin3},  the theorem is easily obtained by passing the limit $\epsilon\rightarrow0$.  
\end{proof}

\setcounter{equation}{0}
\section{Toolbox: weighted Sobolev spaces, interpolation theory and $L^2$ profile of the Laplace-Beltrami operator}

In this section,  we will first   give new profiles  of the weighted Sobolev Spaces.    Then we will state a new version of  interpolation theory which slightly relaxes the assumption that operators are needed to be commutated with each other. Then in the next we will list some basic properties of the real spherical harmonics and introduce the definition of the Laplace-Beltrami operator.  After giving the $L^2$ profile of the fractional Laplace-Beltrami operator,   we  will give a detailed proof to  \eqref{labelint} which are crucial to capture the anisotropic structure of the collision operator.  We address that  results in this section have independent interest.

\subsection{Weighted Sobolev spaces}

\medskip

Before stating the results, we list some basic facts which will be used in the proof of the new profiles of the weighted Sobolev spaces.
\begin{lem}[Bernstein inequalities]\label{Bernstein-Ineq}
There exists a  constant $C$ independent of $j$ and $f$ such that

\noindent 1) For any $s\in \mathbb{R}$ and $j\geq 0$,
\ben\label{bern1}
 C^{-1} 2^{js} \|\mathfrak{F}_j f\|_{L^2(\mathbb{R}^3)} \leq \|\mathfrak{F}_j f\|_{H^s(\mathbb{R}^3)} \leq C 2^{js}\|\mathfrak{F}_j f\|_{L^2(\mathbb{R}^3)}.
\een
2) For integers $j,k \geq 0$ and $
p, q \in [1, \infty]$, the Bernstein  inequalities are shown as
\begin{eqnarray}
\begin{array}{ccc}
\displaystyle{\sup_{|\alpha|=k} \|\partial^{\alpha} \mathfrak{F}_j f\|_{L^{q}(\mathbb{R}^3)} \lesssim 2^{jk} 2^{3j(\frac{1}{p}-\frac{1}{q})}\|\mathfrak{F}_j f \|_{L^{p}(\mathbb{R}^3)}},\\
  \displaystyle{ 2^{jk} \| \mathfrak{F}_j f\|_{L^p(\mathbb{R}^3)} \lesssim
  \sup_{|\alpha|=k} \|\partial^{\alpha} \mathfrak{F}_j f\|_{L^p(\mathbb{R}^3)} \lesssim
  2^{jk} \|\mathfrak{F}_j f \|_{L^{p}(\mathbb{R}^3)}}.
\end{array}\end{eqnarray}
3) For any $f \in H^s$, it holds that
\begin{eqnarray}\label{bern2}
 \|f\|_{H^s(\mathbb{R}^3)}^2 \sim \sum_{k=-1}^{\infty} 2^{2ks} \|\mathfrak{F}_k f\|^2_{L^2(\mathbb{R}^3)}.
\end{eqnarray} 
\end{lem}

\begin{lem}\label{baslem1}(see \cite{hmuy}) Let $s, r\in \R$ and $a(v), b(v)\in C^\infty$ satisfy for any $\alpha\in \Z^3_{+}$,
\beno  |\pa^\alpha_v a(v)|\le C_{1,\alpha} \langle v\rangle^{r-|\alpha|}, |\pa^\alpha_\xi b(\xi)|\le C_{2,\alpha} \langle \xi\rangle^{s-|\alpha|} \eeno for constants $C_{1,\alpha}, C_{2,\alpha}$. Then there exists a constant $C$ depending only on $s, r$ and finite numbers of  $C_{1,\alpha}, C_{2,\alpha}$ such that for any Schwartz function f,
\beno  \|a(\cdot)b(D)f\|_{L^2}\le C\|\langle D\rangle^sW_rf\|_{L^2},\\
\|b(D)a(\cdot)f\|_{L^2}\le C\|W_r\langle D\rangle^sf\|_{L^2}.
\eeno
\end{lem}

\begin{rmk}\label{rmwsp} As a direct consequence, we get \beno \|\langle D\rangle^m W_l f\|_{L^2}\sim \| W_l \langle D\rangle^mf\|_{L^2}\sim \|f\|_{H^m_l}. \eeno
\end{rmk}

\begin{defi} A smooth function $a(v,\xi)$ is said to be a symbol of type $S^{m}_{1,0}$ if   $a(v,\xi)$  verifies that for any multi-indices $\alpha$ and $\beta$,
\beno |(\pa^\alpha_\xi\pa^\beta_v a)(v,\xi)|\le C_{\alpha,\beta} \langle \xi\rangle^{m-|\alpha|}, \eeno
where $C_{\alpha,\beta}$ is a constant depending only on   $\alpha$ and $\beta$.
\end{defi}

\begin{lem}\label{baslem2} Let $l, s, r\in \R$, $M(\xi)\in S^{r}_{1,0}$ and $\Phi(v)\in S^{l}_{1,0}$. Then there exists a constant $C$ such that
\beno \|[M(D_v), \Phi]f \|_{H^s}\le C\|f\|_{H^{r+s-1}_{l-1}}. \eeno
\end{lem}

\begin{proof} We   prove it in the   spirit of \cite{hmuy}.
Thanks to  the expansion of the pseduo-differential operator, it holds that for any $N\in \N$,
\ben\label{psuedoexp} M(D_v)\Phi=\Phi M(D_v)+\sum_{1\le|\alpha|<N}\f1{\alpha !}\Phi_{\alpha}M^{\alpha}(D_v)+r_N(v,D_v),\een
where $\Phi_{\alpha}(v)=\pa_v^\alpha \Phi$, $M^\alpha(\xi)=\pa^\alpha_\xi M(\xi)$ and $$ r_{N}(v, \xi)=N\sum_{|\alpha|=N}\int_0^1\f{(1-\tau)^{N-1}}{\alpha !}r_{N,\tau, \alpha}(v,\xi)d\tau. $$
Here $$r_{N,\tau, \alpha}(v,\xi)=\int \big[(1-\triangle_y)^n\Phi_{\alpha}(v+y)\big]I(\xi;y)\langle y\rangle^{-2m}dy$$ with $2m>N-l+3$, $2n>N-r+3$ and \beno  I(\xi,y)=\f1{(2\pi)^3}\int e^{-iy\eta}(1-\triangle_\eta)^m\bigg[ \langle \eta\rangle^{-2n}M^{(\alpha)}(\xi+\tau\eta)\bigg]d\eta.\eeno
It is not difficult to check that  it holds uniformly with respect to $\tau\in [0,1]$,
\beno  \big| \pa_v^\beta\pa^{\beta'}_\xi r_{N,\tau, \alpha}(v,\xi)\big|\le C_{\beta,\beta'}\langle \xi\rangle^{r-N-|\beta'|}\langle v\rangle^{l-N-|\beta|}.\eeno
 Then \eqref{psuedoexp} and Lemma \ref{baslem1} imply the lemma with $s=0$. The case $s\neq0 $ can be treated similarly and we skip the proof here.
 \end{proof}

Now we are in a position to  give the new profiles of the weighted Sobolev spaces.

\begin{thm}\label{baslem3} Let $m,l\in \R$. Then for   $f\in H^m_l$, we have
\beno  \sum_{k=-1}^\infty 2^{2kl}\|\mathcal{P}_{k}f\|_{H^m}^2\sim \|f\|_{H^m_l}^2.\eeno
\end{thm}

\begin{proof} We first observe that $2^{k(l+1)}\varphi(\f{v}{2^k})$ verifies the condition in Lemma \ref{baslem1}. Then we have
\beno  2^{2kl}\|\mathcal{P}_{k}f\|_{H^m}^2&=&2^{-2k}\|2^{k(l+1)}\mathcal{P}_{k}f\|_{H^m}^2\\
&\lesssim& 2^{-2k} \|\langle D\rangle^m W_{l+1} \mathcal{P}_{k}f\|_{L^2}^2\\
&\lesssim& 2^{-2k} \big[\|W_{l+1} \mathcal{P}_{k}\langle D\rangle^m f\|_{L^2}^2+\|f\|_{H^{m-1}_{l}}^2\big], \eeno
where we use Lemma \ref{baslem2} in the last step. This implies
\beno \sum_{k=-1}^\infty 2^{2kl}\|\mathcal{P}_{k}f\|_{H^m}^2&\lesssim& \sum_{k=-1}^\infty  \|W_{l} \mathcal{P}_{k}\langle D\rangle^m f\|_{L^2}^2 +\|f\|_{H^m_l}^2\\&\lesssim&  \|f\|_{H^m_l}^2.\eeno

To prove the inverse inequality, we first treat with the case $m\ge0$. Thanks to Remark \ref{rmwsp}, we have \beno \|f\|_{H^m_l}^2&\sim& \sum_{k=-1}^\infty \|\mathcal{P}_{k}W_l\langle D\rangle^m f\|_{L^2}^2\sim\sum_{k=-1}^\infty 2^{2kl}\|\mathcal{P}_{k}\langle D\rangle^m f\|_{L^2}^2\\&\lesssim & \sum_{k=-1}^\infty \bigg(2^{2kl}\|\mathcal{P}_{k}f\|_{H^m}^2+2^{-k}\|[2^{k(l+\f12)}\mathcal{P}_{k}, \langle D
\rangle^m]f\|_{L^2}^2 \bigg)\\
&\lesssim & \sum_{k=-1}^\infty  2^{2kl}\|\mathcal{P}_{k}f\|_{H^m}^2+ \| f\|_{H^{m-1}_{l-1/2}}^2,  \eeno
where we use Lemma \ref{baslem2} in the last two steps.
Then by iterated argument, we obtain that for any $N\in\N$,
\beno \|f\|_{H^m_l}^2&\lesssim & \sum_{k=-1}^\infty  2^{2kl}\|\mathcal{P}_{k}f\|_{H^m}^2+ \| f\|_{H^{m-N}_{l-N/2}}^2. \eeno
 Thanks to the fact that for $m\ge0$,  it holds \beno  \sum_{k=-1}^\infty 2^{2kl}\|\mathcal{P}_{k}f\|_{H^m}^2\gtrsim \|f\|_{L^2_{l}}^2.\eeno
Choose $N$ sufficiently large, then we  
  get \beno \|f\|_{H^m_l}^2&\lesssim & \sum_{k=-1}^\infty  2^{2kl}\|\mathcal{P}_{k}f\|_{H^m}^2,\eeno
which gives the proof to  desired result with $m\ge0$.

Next we will use the duality method to deal with the case $m<0$.
 Notice that \beno \int_{\R^3} fgdv&=&\sum_{k=-1}^\infty \int_{\R^3}  \mathcal{P}_{k}f \tilde{ \mathcal{P}}_{k}gdv
\lesssim\sum_{k=-1}^\infty  \| \mathcal{P}_{k}f\|_{H^m}\| \tilde{ \mathcal{P}}_{k}g\|_{H^{-m}} \\ &\lesssim&\big(\sum_{k=-1}^\infty  2^{2kl}\| \mathcal{P}_{k}f\|_{H^m}^2\big)^{\f12}\|g\|_{H^{-m}_{-l}}.\eeno
Then for any Schwartz function $g$, \beno \big|\int_{\R^3} \langle v\rangle^l(\langle D\rangle^m f)gdv\big|&\lesssim& \big(\sum_{k=-1}^\infty  2^{2kl}\| \mathcal{P}_{k}f\|_{H^m}^2\big)^{\f12}\|\langle D\rangle^mW_lg\|_{H^{-m}_{-l}}\\
&\lesssim& \big(\sum_{k=-1}^\infty  2^{2kl}\| \mathcal{P}_{k}f\|_{H^m}^2\big)^{\f12}\| g\|_{L^2},\eeno
which implies \beno \|f\|_{H^m_l}^2&\lesssim & \sum_{k=-1}^\infty  2^{2kl}\| \mathcal{P}_{k}f\|_{H^m}^2.\eeno
We complete the proof of the lemma. \end{proof}

\subsection{Interpolation theory}  
The couple of Banach spaces $(X,Y)$ is said to be an interpolation couple if both $X$ and $Y$ are continuously embedding in a Hausdorff topological vector space. Let $(X,Y)$ be a real interpolation couple, then the real interpolation space $(X,Y)_{\theta, p}$ with $\theta\in (0,1)$ and $p\in [1,\infty]$ is defined as follows:
\beno  (X, Y)_{\theta, p}\eqdefa \bigg\{x\in X+Y\bigg| \|x\|_{\theta, p}\eqdefa \bigg\| t^{-\theta}K(t,x) \bigg\|_{L^p_*(0,\infty)}<\infty\bigg\},\eeno
where $K(t,x)=\inf\limits_{x=a+b, a\in X,b\in Y} (\|a\|_X+t\|b\|_Y)$ and $L^p_*(0,\infty)$ is a Lebesgue  space $L^p$ with respect to the measure $dt/t$. 

Let $X$ be a real  Banach space with norm $\|\cdot\|$.  Let $T$ be a closed operator:
$\mathcal{D}(T)\subset X\rightarrow X$  satisfying there exists a constant $M$ such that for any $\lambda>0$,
\ben (0,\infty)\subset\rho(T),  \|\lambda R(\lambda, T)\|_{L(X)}\le M,\label{clooper}
 \een 
 where $\rho(T)$ denotes the resolvent set of the operator $T$ and 
 \beno R(\lambda, T)\eqdefa (\lambda I-T)^{-1}, \|T\|_{L(X)}\eqdefa \sup_{\|x\|=1} \|Tx\|. \eeno
 Then $\mathcal{D}(T)$ is a Banach space with the graph norm $\|x\|_{\mathcal{D}(T)}=\|x\|+\|Tx\|$ for $x\in \mathcal{D}(T)$.

\begin{prop}[see \cite{al}]\label{interspace}
Let $A$ satisfy \eqref{clooper}. If we set $\mathcal{D}_A(\theta, p)\eqdefa(X, \mathcal{D}(A))_{\theta,p}$, then 
\beno \mathcal{D}_A(\theta, p)=\bigg\{x\in X\bigg| \bigg\|\lambda^{\theta}\|AR(\lambda, A)x\|\bigg\|_{L^p_*(0,\infty)}<\infty\bigg\}. \eeno
\end{prop}

Let $A, B$ be two closed operators satisfying \eqref{clooper}. We recall that $[A,B]=AB-BA.$
In general,  if $[A,B]\neq 0$, it is not easy to derive
 \ben\label{intpfunc}(X, \mathcal{D}(A)\cap \mathcal{D}(B) )_{\theta,2}=\mathcal{D}_A(\theta,2)\cap \mathcal{D}_B(\theta,2).\een 
 The aim of this subsection is to show that under some special conditions on the operators $A$ and $B$, the real interpolation space $(X, \mathcal{D}(A)\cap \mathcal{D}(B))_{\theta,2}$ still verifies \eqref{intpfunc}. We will use this fact to prove \eqref{labelint}.

Let us give the typical examples of the operators which verify the condition \eqref{clooper}.
 Let $\Omega_{ij}=x_i\pa_j-x_j\pa_i$ with  $1\le i<j\le3$ and the domain of the operator is defined as $$\mathcal{D}(\Omega_{ij})=\bigg\{f\in L^2(\R^3_x)| \exists \, g\in L^2(
\R^3_x),  \forall h\in   C^\infty_c(\R^3_x), \int_{\R^3} \big(\Omega_{ij} h\big) fdx=-\int_{\R^3}  hgdx \bigg \}.$$ From which, we give the definition: $g\eqdefa\Omega_{ij}f.$   Then $\Omega_{ij}$ is a closed operator and verifies the condition \eqref{clooper}. Another   example is the partial derivative operator $\pa_k$ with $1\le k\le 3$. We mention that in this case the domain of the operator $\pa_k$ is defined by
$$\mathcal{D}(\pa_k)=\bigg\{f\in L^2(\R^3_x)| \exists \, g\in L^2(
\R^3_x),  \forall h\in   C^\infty_c(\R^3_x), \int_{\R^3} \big(\pa_k h\big) fdx=-\int_{\R^3}  hgdx \bigg \}.$$

 \medskip

The new interpolation theory can be  stated as follows:

\begin{thm}\label{intp}
Let $A, B_1, B_2$ and $B_3$ be closed operators satisfying the  condition \eqref{clooper} and
\ben\label{comuope}[B_i, B_j]=0, [A, B_1]=-B_2, [A,B_2]=-B_1, [A, B_3]=0.\een
If we set $\mathcal{D}(B)=\bigcap\limits_{i=1}^{3} \mathcal{D}(B_i)$ and $\|x\|_{\mathcal{D}(B)}=\|x\|+\sum_{i=1}^3\|B_ix\|.$ Then $$(X, D)_{\theta,2}=\mathcal{D}_A(\theta,2)\cap \mathcal{D}_B(\theta,2),$$ where
$D=\mathcal{D}(A)\cap \mathcal{D}(B)$.
\end{thm}

\begin{proof} By the definition of the real interpolation space, it is easy to check \beno (X, D)_{\theta,2}\subset\mathcal{D}_A(\theta,2)\cap \mathcal{D}_B(\theta,2).
 \eeno Therefore we only need to prove the inverse conclusion.  In other words, we only need to prove
 \beno \mathcal{D}_A(\theta,2)\cap \mathcal{D}_B(\theta,2)\subset (X, D)_{\theta,2}. \eeno
 By the definition of real interpolation space $(X, D)_{\theta,2}$, it is reduced to prove that for $f\in  \mathcal{D}_A(\theta,2)\cap \mathcal{D}_B(\theta,2)$,
\ben\label{goalinter}  \|t^{-\theta}K(t,f)\|_{L^2_*(0,\infty)} <\infty,\een 
 where $K(t,f)=\inf\limits_{f=a+b, a\in X,b\in D} (\|a\|+t\|b\|_D)$ and $\|b\|_D\eqdefa \|b\|+ \|Ab\|+ \sum_{i=1}^3\|B_ib\|$.
 We remark that for $t\ge1$ and  $f\in \mathcal{D}_A(\theta,2)\cap \mathcal{D}_B(\theta,2)$, it holds
\beno  K(t,f)\le  \|f\|,\eeno
which yields
\beno \|t^{-\theta}K(t,f)\|_{L^2_*(1,\infty)}\lesssim \|f\|.\eeno Now it suffices to give the bound for $\|t^{-\theta}K(t,f)\|_{L^2_*(0,1)}$. 
 In order to do that, we perform the following decomposition:
\beno  f=f-V(\lambda)+V(\lambda),\eeno
where \beno  V(\lambda)=\lambda^8R(\lambda,B_3)[R(\lambda,B_1)R(\lambda,B_2)]^2 R(\lambda, A)R(\lambda,B_1)R(\lambda,B_2)f.\eeno
 
{\it Step 1: Estimate of  $\|f-V(\lambda)\|$.}
 From the fact \ben\label{fact5} \lambda R(\lambda,T)=I+TR(\lambda, T), \een it holds
\beno && V(\lambda)-f\\&&=-f+\lambda^7R(\lambda,B_3)R(\lambda,B_1)R(\lambda,B_2)R(\lambda,B_1)
R(\lambda,B_2)R(\lambda, A)R(\lambda,B_1)f\\
&&\quad+\lambda^7R(\lambda,B_3)R(\lambda,B_1)R(\lambda,B_2)R(\lambda,B_1)
R(\lambda,B_2)R(\lambda, A)R(\lambda,B_1)B_2R(\lambda, B_2)f.\eeno
Using  the condition \eqref{clooper}, we get
\beno &&\|\lambda^7R(\lambda,B_3)R(\lambda,B_1)R(\lambda,B_2)R(\lambda,B_1)
R(\lambda,B_2)R(\lambda, A)R(\lambda,B_1)B_2R(\lambda, B_2)f\|\\&&\lesssim \|B_2R(\lambda, B_2)f\|.\eeno
 It gives
 \beno  &&\|V(\lambda)-f\|\\&&\lesssim\|-f+\lambda^7R(\lambda,B_3)R(\lambda,B_1)R(\lambda,B_2)R(\lambda,B_1)
R(\lambda,B_2)R(\lambda, A)R(\lambda,B_1)f\|\\
 &&\quad+ \|B_2R(\lambda, B_2)f\|.\eeno
 
Using \eqref{fact5} again and following the similar argument, we derive  that
\beno  &&\|-f+\lambda^7R(\lambda,B_3)R(\lambda,B_1)R(\lambda,B_2)R(\lambda,B_1)
R(\lambda,B_2)R(\lambda, A)R(\lambda,B_1)f\|
\\&&\lesssim \|-f+\lambda^7R(\lambda,B_3)R(\lambda,B_1)R(\lambda,B_2)R(\lambda,B_1)
R(\lambda,B_2)R(\lambda, A)f\|+\|B_1R(\lambda,B_1)f\|.\eeno 
Then by the inductive method, we obtain that
\ben\label{comop2}  \|V(\lambda)-f\|\lesssim \|AR(\lambda, A)f\|+\sum_{i=1}^3 \|B_iR(\lambda, B_i)f\|.\een

{\it Step 2: Estimate of $\|V(\lambda)\|_D$.} Thanks to the condition \eqref{clooper}, we have $$\|V(\lambda)\|\lesssim \|f\|.$$
Observe that if  $[T_i, T_j]=0$, one has
\beno  R(\lambda, T_i)R(\lambda,T_j)=R(\lambda,T_i)R(\lambda, T_j), T_iR(\lambda,T_j)=R(\lambda,T_j)T_i .
\eeno
From which together with the condition \eqref{clooper} and $[B_i, B_3]=[A, B_3]=0$, we deduce that
\beno  \|B_3V(\lambda)\|\lesssim \lambda\|B_3R(\lambda, B_3)f\|.\eeno

Due to the condition \eqref{comuope},  the standard computation for the resolvent will give the following three facts:
 \ben&& [R(\lambda, B_1)R(\lambda, B_2), R(\lambda, A)]\label{Fact I}\\&&
= R(\lambda, A)R(\lambda, B_1)R(\lambda, B_2)[\lambda(B_1+B_2)-B_1^2-B_2^2]\notag\\&&\quad\times R(\lambda, B_1)R(\lambda, B_2)R(\lambda, A),\notag\\
&& [R(\lambda, B_1)R(\lambda, B_2)^2, R(\lambda, A)]\label{Fact II}\\&&
= R(\lambda, A)R(\lambda, B_1)R(\lambda, B_2)^2[-\lambda^2(B_2+2B_1)
+2\lambda(B_1B_2+B_1^2+B_2^2)\notag
\\&&\quad+B_2^3-2B_1^2B_2]R(\lambda, B_1)R(\lambda, B_2)^2R(\lambda, A),\notag\een
and  \ben&& [(R(\lambda, B_1)R(\lambda, B_2))^2, R(\lambda, A)]\label{Fact III}\\&&
= R(\lambda, A)[R(\lambda, B_1)R(\lambda, B_2)]^2[2\lambda^3(B_1+B_2)
-4\lambda^2(B_1B_2+B_1^2+B_2^2)\notag
\\&&\quad+2\lambda(B_1^3+B_2^3+2B_1^2B_3+2B_2^2B_1)-2B_1B_2^3-2B_2B_1^3][R(\lambda, B_1)R(\lambda, B_2)]^2 R(\lambda, A).\notag\een

Now we start to estimate  $\|AV(\lambda)\|$ and $\|B_1V(\lambda)\|$.
It is easy to check
\beno AV(\lambda)&=& \lambda^8A R(\lambda,B_3)[R(\lambda,B_1)R(\lambda,B_2)]^2 R(\lambda, A)R(\lambda,B_1)R(\lambda,B_2)f\\
&=& \lambda^8 AR(\lambda, A)[R(\lambda, B_1)R(\lambda, B_2)]^3R(\lambda, B_3)f\\&&
+\lambda^8A[(R(\lambda, B_1)R(\lambda, B_2))^2, R(\lambda, A)]R(\lambda,B_1)R(\lambda,B_2)R(\lambda, B_3)f\\
&\eqdefa & R_1+R_2.\eeno
By \eqref{fact5} and the condition \eqref{clooper}, we get \ben\label{comop1} \|TR(\lambda, T)\|_{L(X)}\lesssim 1. \een
From which together with \eqref{Fact III} and the condition \eqref{clooper}, we obtain that \beno \|R_2\|\lesssim \|f\|. \eeno
Notice that
\beno  R_1&=&\lambda^8(-I+\lambda R(\lambda, A))[R(\lambda, B_1)R(\lambda, B_2)]^3R(\lambda, B_3)f\\
&=&\lambda^8[R(\lambda, B_1)R(\lambda, B_2)]^3R(\lambda, B_3)(-I)f\\&&+\lambda^9[R(\lambda, A), (R(\lambda, B_1)R(\lambda, B_2))^2]
R(\lambda, B_1)R(\lambda, B_2)R(\lambda, B_3)f\\&&+\lambda^9[R(\lambda, B_1)R(\lambda, B_2)]^2[R(\lambda, A), R(\lambda, B_1)R(\lambda, B_2)]
R(\lambda, B_3)f\\&&+\lambda^9[R(\lambda, B_1)R(\lambda, B_2)]^3R(\lambda, B_3)R(\lambda, A)f\\
&=&\lambda^8[R(\lambda, B_1)R(\lambda, B_2)]^3R(\lambda, B_3)AR(\lambda, A)f\\&&+\lambda^9[R(\lambda, A), (R(\lambda, B_1)R(\lambda, B_2))^2]
R(\lambda, B_1)R(\lambda, B_2)R(\lambda, B_3)f\\&&+\lambda^9[R(\lambda, B_1)R(\lambda, B_2)]^2[R(\lambda, A), R(\lambda, B_1)R(\lambda, B_2)]
R(\lambda, B_3)f\\
&\eqdefa& R_3+R_4+R_5.\eeno
Thanks to \eqref{clooper}, \eqref{comop1} , \eqref{Fact I} and \eqref{Fact III}, we get 
\beno  \|R_3\|\lesssim \lambda \|AR(\lambda, A)f\|, \|R_4\|+\|R_5\|\lesssim \|f\|,\eeno
which implies that $$\|AV(\lambda)\|\lesssim \|f\|+\lambda\|AR(\lambda, A)f\|.$$

Similarly we have \beno
B_1V(\lambda)&=&\lambda^8(-I+\lambda R(\lambda, B_1))R(\lambda, B_1)R(\lambda, B_2)^2R(\lambda, A)
R(\lambda, B_1)R(\lambda, B_2)R(\lambda, B_3)\\
&=& \lambda^8R(\lambda, A)R(\lambda,B_2)^3R(\lambda, B_1)^2R(\lambda, B_3)B_1R(\lambda, B_1)f\\&&-\lambda^8[R(\lambda, B_1)R(\lambda, B_2)^2, R(\lambda, A)]R(\lambda, B_1)R(\lambda, B_2)R(\lambda, B_3)f\\&&+\lambda^9[R(\lambda, B_1)^2R(\lambda, B_2)^2, R(\lambda, A)]R(\lambda, B_1)R(\lambda, B_2)R(\lambda, B_3)f.\eeno
Thanks to \eqref{clooper}, \eqref{comop1} , \eqref{Fact II} and \eqref{Fact III}, we have $$\|B_1V(\lambda)\|\lesssim \|f\|+\lambda\|B_1R(\lambda, B_1)f\|.$$

By the same argument, we can get \beno \|B_2V(\lambda)\|\lesssim \|f\|+\lambda\|B_2R(\lambda, B_2)f\|.\eeno
Patching together all the estimates, we finally get
\ben\label{comop3}  \|V(\lambda)\|_D\lesssim  \|f\|+\lambda\|AR(\lambda, A)f\|+\sum_{i=1}^3\lambda\|B_iR(\lambda, B_i)f\|.\een
\smallskip

Then for $\lambda\ge1$, one has
\beno  &&\lambda^\theta (\|V(\lambda)-f\|+\lambda^{-1}\|V(\lambda)\|_D)\\&&
\lesssim \lambda^\theta(\|AR(\lambda, A)f\|+\sum_{i=1}^3\|B_iR(\lambda, B_i)f\|) +\lambda^{\theta-1}\|f\|.\eeno
 Thanks to the condition $[B_i,B_j]=0,$ by Proposition 3.1 in \cite{al}, one has \beno \mathcal{D}_B(\theta,2)=\bigcap_{i=1}^3 \mathcal{D}_{B_i}(\theta,2).\eeno
Then if $f\in \mathcal{D}_A(\theta,2)\cap \mathcal{D}_B(\theta,2)$, by Proposition \ref{interspace}, we have for $\lambda>0$,
\ben\label{Fact4} \lambda^\theta\|AR(\lambda, A)f\|, \lambda^\theta\|B_iR(\lambda, B_i)f\| \in L^2_*(0,\infty). \een It means
\beno \bigg\|\lambda^\theta (\|V(\lambda)-f\|+\lambda^{-1}\|V(\lambda)\|_D)\bigg\|_{L^2_*(1,\infty)}\lesssim \|f\|+\|f\|_{\mathcal{D}_A(\theta,2)}+\|f\|_{\mathcal{D}_B(\theta,2)}.\eeno
In other words, we get
\beno \|t^{-\theta}K(t,f)\|_{L^2_*(0,1)}\le\bigg\|t^{-\theta} (\|V(t^{-1})-f\|+t\|V(t^{-1})\|_D)\bigg\|_{L^2_*(0,1)}\lesssim \|f\|+\|f\|_{\mathcal{D}_A(\theta,2)}+\|f\|_{\mathcal{D}_B(\theta,2)}.\eeno
 We complete the proof to  \eqref{goalinter} and then ends the proof of the theorem.
 \end{proof}

\subsection{Spherical harmonics}
In this subsection, we introduce the definition and   basic properties of the real spherical harmonics.

 Let
$\sigma=(\cos\phi\sin \theta, \sin\phi\sin \theta, \cos\theta)\in \SS^2$ with $\theta\in[0,\pi]$ and $\phi\in [0,2\pi)$. The real spherical harmonics $Y_l^m(\sigma)$ with $l\in \N$, $-l\le m\le l$, are defined as $Y_0^0(\sigma)=(4\pi)^{-1/2}$ and for any $l\ge1$,
\beno  Y^m_l(\sigma)=\left\{\begin{aligned} &\bigg(\f{2l+1}{4\pi}\bigg)^{1/2}P_l(\cos\theta), \mbox{if}\, m=0,\\&\bigg(\f{2l+1}{2\pi}\f{(l-m)!}{(l+m)!}\bigg)^{1/2}P^m_l(\cos\theta)\cos(m\phi), \,\mbox{if}\,\,m=1,\ldots,\,l ,\\&\bigg(\f{2l+1}{2\pi}\f{(l+m)!}{(l-m)!}\bigg)^{1/2}P^{-m}_l(\cos\theta)\sin(-m\phi), \,\mbox{if}\,\,m=-l,\,\ldots,\,-1, \end{aligned}\right.\eeno
where $P_l$ denotes the $l$-th Legendre polynomial and $P^m_l$ denotes the associated Legendre functions of  order $l$ and degree $m$. It is well-known that
\beno (-\triangle_{\SS^2})Y_l^m=l(l+1)Y_l^m.  \eeno
We remark that the family $(Y_l^m)_{l,m}$ is an orthonormal basis of the space $L^2(\SS^2, d\sigma)$ with $d\sigma$ being the surface measure on $\SS^2$. Thus if $f\in L^2(\SS^2)$, then we have
\beno f(\sigma)=\sum_{l=0}^\infty\sum_{m=-l}^l f^m_lY^m_l(\sigma), \eeno
where $f^m_l=\int_{\SS^2} f(\sigma)Y^m_l(\sigma)d\sigma.$ Then for $s\in\R$, the fractional Laplace-Beltrami operator $(-\triangle_{\SS^2})^{s/2}$ is defined by 
\ben\label{defilb1}
\big((-\triangle_{\SS^2})^{s/2}f\big)(\sigma)\eqdefa \sum_{l=0}^\infty\sum_{m=-l}^l (l(l+1))^{s/2}f^m_lY^m_l(\sigma).
\een  Similarly we have 
\ben\label{defilb11}
\big((1-\triangle_{\SS^2})^{s/2}f\big)(\sigma)\eqdefa \sum_{l=0}^\infty\sum_{m=-l}^l (1+l(l+1))^{s/2}f^m_lY^m_l(\sigma).
 \een
\smallskip

Next we denote $\mathcal{A}_l$ by the space of solid spherical harmonics of degree $l$, that is, the set of  all  homogeneous polynomials of degree $l$ on $\R^3$ that are harmonic. Let $\mathcal{H}_l$  be the space of spherical harmonics of degree $l$. Then we define $\mathcal{D}_l$ to be a space of all linear combinations of functions of the form $f(r)P(x)$, where $f$ ranges  over the radial functions and $P$ over the solid spherical harmonics of degree $l$, in such a way that $f(r)P(x)$ belongs to $L^2(\R^3)$.
We have
\begin{thm} \label{l2decom} \beno L^2(\R^3)=\sum_{l=0}^\infty \bigoplus\mathcal{D}_l. \eeno
Moreover, for $Y\in \mathcal{H}_l$, there exists a function $\Psi$ defined on the $[0,\infty)$ such that for $w> 0$,
\ben\label{sphar} \int_{\SS^2} e^{-2\pi i w \sigma\cdot \tau} Y(\sigma)d\sigma= \Psi(w)Y(\tau),\een
which means that the Fourier transform maps $\mathcal{D}_l$ into itself.
\end{thm}

Suppose $f$ is a Schwartz function. Thanks to Theorem \ref{l2decom}, we have
\beno f(x)=\sum_{l=0}^\infty\sum_{m=-l}^l  Y^m_l(\sigma) f^m_l(r), \eeno where $x=r\sigma$ and   $\sigma\in \SS^2$.
Then for $s\in\R$,
\ben\label{defilb2} \big((-\triangle_{\SS^2})^{s/2}f\big)(x)\eqdefa\sum_{l=0}^\infty\sum_{m=-l}^l (l(l+1))^{s/2} Y^m_l(\sigma) f^m_l(r). \een
Similarly we have 
\ben\label{defilb21} \big((1-\triangle_{\SS^2})^{s/2}f\big)(x)\eqdefa\sum_{l=0}^\infty\sum_{m=-l}^l (1+l(l+1))^{s/2} Y^m_l(\sigma) f^m_l(r). \een

We recall the statement of the addition theorem:
\begin{thm}[Addition Theorem] Suppose that $\sigma$ and $\tau$ are two unit vectors. Then
\beno  P_l(\sigma\cdot \tau)= \f{4\pi}{2l+1}\sum_{m=-l}^l Y_l^m(\sigma) Y_l^m(\tau). \eeno
\end{thm}

Now we want to prove
\begin{lem}\label{antin5} Suppose $H(x)\in L^2([-1,1])$. Then we have
\beno &&\int_{\SS^2\times\SS^2} \big( g(\sigma)-g(\tau)\big)h(\sigma) H(\sigma\cdot \tau)d\sigma d\tau\\&&
=\sum_{l=0}^\infty\sum_{m=-l}^l g^m_l h^m_l \int_{\SS^2\times\SS^2} \big( Y^m_l(\sigma)-Y^m_l(\tau)\big)Y^m_l(\sigma) H(\sigma\cdot \tau)d\sigma d\tau.
 \eeno Here we use the notation: $f^m_l\eqdefa \int_{\SS^2} f(\sigma)Y^m_l(\sigma)d\sigma.$ 
\end{lem}

\begin{proof} Thanks to the fact the family $\{P_n\}_{n\ge0}$ is an orthogonal basis of the space $L^2[-1,1]$, we have
\beno  H(x)=\sum_{n\ge0} a_n P_n(x),\eeno
where $a_n(x)=(n+\f12)\int_{-1}^1 H(x)P_n(x)dx.$ In particular, it gives
\beno &&\int_{\SS^2\times\SS^2} g(\tau) h(\sigma) H(\sigma\cdot \tau)d\sigma d\tau\\
&&= \sum_{n=0}^\infty   a_n \int_{\SS^2\times\SS^2} g(\tau) h(\sigma) P_n(\sigma\cdot\tau) d\sigma d\tau\\
&&= \sum_{n=0}^\infty\sum_{q=-n}^n  a_n  \f{4\pi}{2n+1}\int_{\SS^2\times\SS^2} g(\tau) h(\sigma)Y_n^q(\sigma) Y_n^q(\tau) d\sigma d\tau\\
&&=\sum_{n=0}^\infty\sum_{q=-n}^n g_n^q h^{q}_{n}a_n  \f{4\pi}{2n+1}  \\
&&=\sum_{l=0}^\infty\sum_{m=-l}^l  g_l^m h^m_l\int_{\SS^2\times\SS^2} Y^m_l(\tau) Y^{m}_{l}(\sigma) H(\sigma\cdot \tau)d\sigma d\tau,\eeno
where we use Theorem \ref{antin5} in the second  and the last equalities.
On the other hand, \beno  \int_{\SS^2\times\SS^2} g(\sigma) h(\sigma) H(\sigma\cdot \tau)d\sigma d\tau
&=&  \int_{\SS^2} g(\sigma) h(\sigma)d\sigma \int_{\SS^2}H(\sigma\cdot \tau)d\tau\\
&=& (\sum_{l=0}^\infty\sum_{m=-l}^l  g^m_l h^m_l) \iint_{\theta\in[0,\pi], \phi\in[0,2\pi]}H(\cos\theta)\sin\theta d\theta d\phi\\
&=&\sum_{l=0}^\infty\sum_{m=-l}^l  g^m_l h^m_l \int_{\SS^2\times\SS^2} Y^m_l(\sigma) Y^m_l(\sigma) H(\sigma\cdot \tau)d\sigma d\tau.\eeno
Combine these two equalities and then we get the desired result. \end{proof}

\subsection{$L^2$ profile of the fractional Laplace-Beltrami operator}
   In this subsection, we first show    the   $L^2$ profile of the fractional Laplace-Beltrami operator. Then we  show that in the whole space the  fractional Laplace-Beltrami operator has strong connection to the rotation vector fields.

\begin{lem}\label{antin1} Suppose that $f$ is a smooth function defined in $\SS^2$. Then if $0<s<1$, it holds
\beno \|f\|_{L^2(\SS^2)}^2+\int_{\sigma,\tau\in \SS^2} \f{|f(\sigma)-f(\tau)|^2}{|\sigma-\tau|^{2+2s}}d\sigma d\tau\sim \|f\|^2_{L^2(\SS^2)}+\|(-\triangle_{\SS^2})^{s/2}f\|^2_{L^2(\SS^2)}.
 \eeno
\end{lem}
\begin{proof} Let $\omega_1, \omega_2, \omega_3\in C^\infty_c(\R)$ be non-negative functions. Assume that $\omega_1(x)=1$ in the Ball $B_{\f23}$ with compact support in the Ball $B_{\f34}$, $\omega_2(x)=1$ in the Ball $B_{\f34}$ with compact support in the Ball $B_{\f45}$ and $\omega_3(x)=1$ in the Ball $B_{\f45}$ with compact support in the Ball $B_{\f56}$. Let $\chi$ be a smooth function verifying
 \beno  \chi(x)=\left\{\begin{aligned} & 1, \quad\mbox{if}\quad x\ge0;\\
&  0,\quad\mbox{if}\quad x<-\f1{10}.\end{aligned}\right.\eeno
Suppose  $u=(u_1, u_2, u_3) \in \R^3$. Then it is easy to check for $k\in \NN$ with $1\le k\le 3$ and $u\neq 0$,
\beno  \sum_{i=1}^3 \omega_k(\sum_{j\neq i} \f{u_j^2}{|u|^2})\ge 1.\eeno
Then we suppose  that for   $1\le m\le 3$,
\beno \vartheta_{km+}(u)\eqdefa\f{\omega_k(\sum\limits_{j\neq m} \f{u_j^2}{|u|^2})}{\sum\limits_{i=1}^3 \omega_k(\sum\limits_{j\neq i} \f{u_j^2}{|u|^2})}\chi(\f{u_m}{|u|})\quad
\mbox{and}\quad \vartheta_{km-}(u)\eqdefa\f{\omega_k(\sum\limits_{j\neq m} \f{u_j^2}{|u|^2})}{\sum\limits_{i=1}^3 \omega_k(\sum\limits_{j\neq i} \f{u_j^2}{|u|^2})}\chi(-\f{u_m}{|u|}). \eeno
We conclude that for $u\in\SS^2$, \ben\label{lb6} \sum_{m=1}^3 \big[\vartheta_{km+}(u)+\vartheta_{km-}(u)\big]=1.\een

Observe that
\ben\label{lb1} &&\int_{\sigma,\tau\in \SS^2} \f{|f(\sigma)-f(\tau)|^2}{|\sigma-\tau|^{2+2s}}d\sigma d\tau\nonumber\\
&&\sim\sum_{m=1}^3\bigg[\int_{\sigma,\tau\in \SS^2} \f{|f(\sigma)-f(\tau)|^2}{|\sigma-\tau|^{2+2s}}[\vartheta_{1m+}^2(\sigma)+\vartheta_{1m-}^2(\sigma)] d\sigma d\tau. \een
Then due to the symmetric structure, we only need to focus on the estimate of
\beno \int_{\sigma,\tau\in \SS^2} \f{|f(\sigma)-f(\tau)|^2}{|\sigma-\tau|^{2+2s}}\vartheta_{13+}^2(\sigma) d\sigma d\tau. \eeno

Notice that
\beno  &&\int_{\sigma,\tau\in \SS^2} \f{|f(\sigma)-f(\tau)|^2}{|\sigma-\tau|^{2+2s}}\vartheta_{13+}^2(\sigma) d\sigma d\tau\\&&=
\int_{\sigma,\tau\in \SS^2} \f{|f(\sigma)-f(\tau)|^2}{|\sigma-\tau|^{2+2s}}\vartheta_{13+}^2(\sigma) \vartheta_{33+}(\sigma)\vartheta_{33+}(\tau)d\sigma d\tau\\
&&+
\int_{\sigma,\tau\in \SS^2} \f{|f(\sigma)-f(\tau)|^2}{|\sigma-\tau|^{2+2s}}\vartheta_{13+}^2(\sigma) \vartheta_{33+}(\sigma)[1-\vartheta_{33+}(\tau)]d\sigma d\tau. \eeno
From which together with the fact that $|\sigma-\tau|\ge \f12-\f1{\sqrt{5}}$ if $\sigma\in \mathrm{Supp}\,\vartheta_{13+}$ and $\tau\in \mathrm{Supp}\, (1-\vartheta_{33+})$, we deduce that
\beno &&\int_{\sigma,\tau\in \SS^2} \f{|f(\sigma)-f(\tau)|^2}{|\sigma-\tau|^{2+2s}}\vartheta_{13+}^2(\sigma) d\sigma d\tau+\|f\|_{L^2(\SS^2)}^2\\
&&\sim \int_{\sigma,\tau\in \SS^2} \f{|f(\sigma)-f(\tau)|^2}{|\sigma-\tau|^{2+2s}}\vartheta_{13+}^2(\sigma) \vartheta_{33+}(\sigma)\vartheta_{33+}(\tau)d\sigma d\tau+\|f\|_{L^2(\SS^2)}^2\\
\\
&&\sim \int_{\sigma,\tau\in \SS^2} \f{|(\vartheta_{13+}f)(\sigma)-(\vartheta_{13+}f)(\tau)|^2}{|\sigma-\tau|^{2+2s}} \vartheta_{33+}(\sigma)\vartheta_{33+}(\tau)d\sigma d\tau+\|f\|_{L^2(\SS^2)}^2.\eeno

Suppose  $\sigma=(x_1, x_2, x_3) \in \SS^2$ and  $x=(x_1,x_2)$. Let $F_{13}^+(x)\eqdefa(\vartheta_{13+}f)(x_1,x_2, \sqrt{1-x_1^2-x_2^2})$ and $\Theta_{i3}^+(x)=\vartheta_{i3+}(x_1,x_2, \sqrt{1-x_1^2-x_2^2})(i=1,2,3)$. Then by change of variables, we have
\beno  && \int_{\sigma,\tau\in \SS^2} \f{|(\vartheta_{13+}f)(\sigma)-(\vartheta_{13+}f)(\tau)|^2}{|\sigma-\tau|^{2+2s}} \vartheta_{33+}(\sigma)\vartheta_{33+}(\tau)d\sigma d\tau\\
&&=\int_{|x|,|y|\le \sqrt{\f56}} \f{|F_{13}^+(x)-F_{13}^+(y)|^2}{\big(|x-y|^2+|\sqrt{1-x_1^2-x_2^2}-\sqrt{1-y_1^2-y_2^2}|^2\big)^{1+s}} \\&&\quad\times\Theta_{33}^+(x)\Theta_{33}^+(y)\f1{\sqrt{1-x_1^2-x_2^2}}
\f1{\sqrt{1-y_1^2-y_2^2}}dxdy\\
&&\gtrsim \int_{|x|,|y|\le \sqrt{\f45}} \f{|F_{13}^+(x)-F_{13}^+(y)|^2}{ |x-y|^{2+2s}} dxdy,
 \eeno
which yields
\ben\label{lb2} &&\int_{\sigma,\tau\in \SS^2} \f{|f(\sigma)-f(\tau)|^2}{|\sigma-\tau|^{2+2s}}\vartheta_{13+}^2(\sigma) d\sigma d\tau+\|f\|_{L^2(\SS^2)}^2\gtrsim \|F_{13}^+\|_{H^s(B_{\f2{\sqrt{5}}})}^2.  \een
On the other hand,  one has \beno  && \int_{\sigma,\tau\in \SS^2} \f{|(\vartheta_{13+}f)(\sigma)-(\vartheta_{13+}f)(\tau)|^2}{|\sigma-\tau|^{2+2s}} \vartheta_{33+}(\sigma)\vartheta_{33+}(\tau)d\sigma d\tau\\
&&\lesssim \int_{|x-y|\le \eta} \f{|F_{13}^+(x)-F_{13}^+(y)|^2}{ |x-y|^{2+2s}} \Theta_{33}^+(x)\Theta_{33}^+(y) dxdy+C_\eta\|F_{13}^+\|_{L^2(B_{\f2{\sqrt{5}}})}^2.\eeno
 Choose $\eta$   sufficiently small such that  \beno \mathrm{1}_{|x-y|\le \eta}F_{13}^+(x)^2\Theta_{33}^+(x)\Theta_{33}^+(y)=\mathrm{1}_{|x-y|\le \eta}F_{13}^+(x)^2\Theta_{23}^+(x)\Theta_{23}^+(y),\\
\mbox{and}\quad\mathrm{1}_{|x-y|\le \eta} F_{13}^+(x)F_{13}^+(y) \Theta_{33}^+(x)\Theta_{33}^+(y)=\mathrm{1}_{|x-y|\le \eta}F_{13}^+(x)F_{13}^+(y)\Theta_{23}^+(x)\Theta_{23}^+(y).\eeno
Then we  get 
\beno&& \int_{\sigma,\tau\in \SS^2} \f{|(\vartheta_{13+}f)(\sigma)-(\vartheta_{13+}f)(\tau)|^2}{|\sigma-\tau|^{2+2s}} \vartheta_{33+}(\sigma)\vartheta_{33+}(\tau)d\sigma d\tau\\
&&\lesssim \int_{|x|,|y|\le \sqrt{\f45}}  \f{|F_{13}^+(x)-F_{13}^+(y)|^2}{ |x-y|^{2+2s}} \Theta_{23}^+(x)\Theta_{23}^+(y) dxdy+C_\eta\|F_{23}^+\|_{L^2(B_{\f2{\sqrt{5}}})}^2\\&&
\lesssim \|F_{13}^+\|_{H^s(B_{\f2{\sqrt{5}}})}^2, \eeno
which implies
\beno &&\int_{\sigma,\tau\in \SS^2} \f{|f(\sigma)-f(\tau)|^2}{|\sigma-\tau|^{2+2s}}\vartheta_{13+}^2(\sigma) d\sigma d\tau\lesssim  \|f\|_{L^2(\SS^2)}^2+ \|F_{13}^+\|_{H^s(B_{\f2{\sqrt{5}}})}^2.\eeno
From which together with \eqref{lb2}, we have 
\ben\label{lb3} &&\int_{\sigma,\tau\in \SS^2} \f{|f(\sigma)-f(\tau)|^2}{|\sigma-\tau|^{2+2s}}\vartheta_{13+}^2(\sigma) d\sigma d\tau+\|f\|_{L^2(\SS^2)}^2\sim \|F_{13}^+\|_{H^s(B_{\f2{\sqrt{5}}})}^2+\|f\|_{L^2(\SS^2)}^2.  \een

 Observe  that \beno \|(-\triangle_{\SS^2})^{\f12}  (\vartheta_{13+}f)\|_{L^2(\SS^2)}^2=\int_{ \SS^2} \big((-\triangle_{\SS^2})(\vartheta_{13+}f)\big)(\sigma) (\vartheta_{13+}f)(\sigma)d\sigma.\eeno
Thanks to the fact $(-\triangle_{\SS^2}f)(\sigma)=-\sum\limits_{1\le i<j\le 3} (\Omega_{ij}^2 f)(x_1, x_2, x_3)$ with $\sigma=(x_1, x_2, x_3)$ and $\Omega_{ij}=x_i\pa_j-x_j\pa_i$, by change of variables, we obtain that
\beno &&\|(-\triangle_{\SS^2})^{\f12}  (\vartheta_{13+}f)\|_{L^2(\SS^2)}^2\\&&=\int_{|x|\le \sqrt{\f45}} -\sum\limits_{1\le i<j\le 3} \bigg(\Omega_{ij}^2(\vartheta_{13+}f)\bigg)(x,\sqrt{1-|x|^2})(\vartheta_{13+}f)(x,\sqrt{1-|x|^2})\f1{\sqrt{1-|x|^2}}dx.\eeno
It is easy to   see that for $i=1,2$, \beno \pa_iF_{13}^+(x_1,x_2)=\f1{\sqrt{1-|x|^2}}\bigg(-\Omega_{i3}(\vartheta_{13+}f)\bigg)(x,\sqrt{1-|x|^2}), \eeno
which implies that \beno \bigg(\Omega_{i3}^2(\vartheta_{13+}f)\bigg)(x,\sqrt{1-|x|^2})=\bigg((\sqrt{1-|x|^2}\pa_i)^2F_{13}^+\bigg)(x).\eeno
Then by direct calculation, it yields
\beno  &&\|(-\triangle_{\SS^2})^{\f12}  (\vartheta_{13+}f)\|_{L^2(\SS^2)}^2\\
&&=-\int_{|x|\le \sqrt{\f45}}  \big(\pa_1(\sqrt{1-|x|^2}\pa_1)F_{13}^+\big)F_{13}^+dx -\int_{|x|\le \sqrt{\f45}}  \big(\pa_2(\sqrt{1-|x|^2}\pa_2) F_{13}^+\big)F_{13}^+ dx\\
&&\quad-\int_{|x|\le \sqrt{\f45}}  (\Omega_{12})^2F_{13}^+F_{13}^+\f1{\sqrt{1-|x|^2}}dx.\eeno
Thus we have
\ben\label{lb12} &&\|(-\triangle_{\SS^2})^{\f12}  (\vartheta_{13+}f)\|_{L^2(\SS^2)}^2+\|\vartheta_{13+}f\|_{L^2(\SS^2)}^2\nonumber\\
&& \sim \|F_{13}^+\|_{H^1(B_{\f2{\sqrt{5}}})}^2+\|\Omega_{12}F_{13}^+\|_{L^2(B_{\f2{\sqrt{5}}})}^2\sim \|F_{13}^+\|_{H^1(B_{\f2{\sqrt{5}}})}^2,  \een
where we use the fact $\|\vartheta_{13+}f\|_{L^2(\SS^2)}\sim \|F_{13}^+\|_{L^2(B_{\f2{\sqrt{5}}})}.$

By  the real interpolation method, we obtain that for $0\le s\le 1$,
\ben\label{lb4}\|(-\triangle_{\SS^2})^{s/2}(\vartheta_{13+}f)\|_{L^2(\SS^2)}+\|\vartheta_{13+}f\|_{L^2(\SS^2)}\sim \|F_{13}^+\|_{H^s(B_{\f2{\sqrt{5}}})}. \een
Next we claim that for $0\le s\le 2$,
 \ben\label{lb5} &&\|(-\triangle_{\SS^2})^{s/2}(\vartheta_{1m+}f)\|_{L^2(\SS^2)}+ \|\vartheta_{1m+}f\|_{L^2(\SS^2)}\nonumber\\&&\lesssim \|(-\triangle_{\SS^2})^{s/2} f\|_{L^2(\SS^2)}+\|f\|_{L^2(\SS^2)}. \een
This is easily  followed by the real interpolation method since
\beno \|(-\triangle_{\SS^2})(\vartheta_{1m+}f)\|_{L^2(\SS^2)}\lesssim  \|(-\triangle_{\SS^2}) f\|_{L^2(\SS^2)}+\|  f\|_{L^2(\SS^2)},\\
\mbox{and} \qquad \| \vartheta_{1m+}f\|_{L^2(\SS^2)}\lesssim  \|  f\|_{L^2(\SS^2)}. \eeno

 Thanks to \eqref{lb6} and \eqref{lb5}, we  deduce that  for $0\le s\le 2$,
 \ben \label{lb7} &&\sum_{m=1}^3\bigg[\|(-\triangle_{\SS^2})^{s/2}(\vartheta_{1m+}f)\|_{L^2(\SS^2)}+\|(-\triangle_{\SS^2})^{s/2}(\vartheta_{1m-}f)\|_{L^2(\SS^2)}\nonumber\\&&+\|\vartheta_{1m+}f\|_{L^2(\SS^2)}+\|\vartheta_{1m-}f\|_{L^2(\SS^2)}\bigg]
\sim \|(-\triangle_{\SS^2})^{s/2} f\|_{L^2(\SS^2)}+\|f\|_{L^2(\SS^2)}. \een
Then (\ref{lb3}-\ref{lb4})  and  \eqref{lb7} imply the lemma. 
\end{proof}

As a consequence of Lemma \ref{antin5} and Lemma \ref{antin1}, we get the following estimate:
\begin{col}\label{antin6} Suppose that  $g$ and $h$ are smooth functions defined in $\SS^2$. Then for $a,b\in \R$ with $a+b=2s$,
\beno&& \bigg|\int_{\SS^2\times\SS^2} \big(g(\sigma)-g(\tau)\big)h(\tau) H(\sigma\cdot\tau) d\sigma d\tau\bigg|\\
&&\lesssim
\|(1-\triangle_{\SS^2})^{a/2}g\|_{L^2(\SS^2)}\|(1-\triangle_{\SS^2})^{b/2}h\|_{L^2(\SS^2)},\eeno
where $H(\sigma\cdot \tau)=|\sigma-\tau|^{-(2+2s)}$.
\end{col}
\begin{proof}
Let  $\lambda>0$. Then by Lemma \ref{antin5} and the notation: $f^m_l\eqdefa \int_{\SS^2} f(\sigma)Y^m_l(\sigma)d\sigma$,  we have
\beno &&\int_{\SS^2\times\SS^2} \big(g(\sigma)-g(\tau)\big)h(\tau) H(\sigma\cdot\tau) d\sigma d\tau\\&&=
\lim_{\lambda\rightarrow0}\int_{\SS^2\times\SS^2}  \big(g(\sigma)-g(\tau)\big)h(\tau) H(\sigma\cdot\tau)\mathrm{1}_{|\sigma\cdot \tau|\ge\lambda} d\sigma d\tau\\&&=
\lim_{\lambda\rightarrow0} \sum_{l=0}^\infty\sum_{m=-l}^l g_l^mh_l^m \int_{\SS^2\times\SS^2} (Y^m_l(\sigma)-Y^m_l(\tau))Y^m_l(\tau) H(\sigma\cdot\tau)\mathrm{1}_{|\sigma\cdot \tau|\ge\lambda} d\sigma d\tau\\
&&=\f12
\lim_{\lambda\rightarrow0}\sum_{l=0}^\infty\sum_{m=-l}^lg_l^mh_l^m \int_{\SS^2\times\SS^2} (Y^m_l(\sigma)-Y^m_l(\tau))(Y^m_l(\sigma)-Y^m_l(\tau)) H(\sigma\cdot\tau)\mathrm{1}_{|\sigma\cdot \tau|\ge\lambda} d\sigma d\tau,
   \eeno where we use the symmetric property of the integral in the last step. Applying the Cauchy-Schwartz inequality and Lemma \ref{antin1}, we obtain that
\beno   &&\int_{\SS^2\times\SS^2} \big(g(\sigma)-g(\tau)\big)h(\tau) H(\sigma\cdot\tau) d\sigma d\tau\\&&\lesssim \sum_{l=0}^\infty\sum_{m=-l}^l g_l^mh_l^m
(\|(-\triangle_{\SS^2})^{s/2}Y^m_l\|_{L^2(\SS^2)}+1)(\|(-\triangle_{\SS^2})^{s/2}Y^m_l\|_{L^2(\SS^2)}+1)
\\&&\lesssim \sum_{l=0}^\infty\sum_{m=-l}^l g_l^mh_l^m (l(l+1)+1)^{s} \lesssim \|(1-\triangle_{\SS^2})^{a/2}g\|_{L^2(\SS^2)}\|(1-\triangle_{\SS^2})^{b/2}h\|_{L^2(\SS^2)}, \eeno which completes the proof of the lemma.
\end{proof}

Next we show the strong connection between the Laplace-Beltrami operator and the rotation vector fields.

\begin{lem}\label{antin2}  Let $f$ be a smooth function defined in $\R^3$. Suppose $f(u)=f(u_1,u_2,u_3)$ with $u=(u_1,u_2,u_3)\in \R^3$ and
$\Omega_{ij} f\eqdefa (u_i\pa_{u_j}-u_j\pa_{u_i})f$. Then if $0<s<1$, it holds
\beno &&\int_{\sigma,\tau\in \SS^2, r>0} \f{|f(r\sigma)-f(r\tau)|^2}{|\sigma-\tau|^{2+2s}}r^2d\sigma d\tau dr+\|f\|^2_{L^2(\R^3)}\\
&&\sim \|(-\triangle_{\SS^2})^{s/2}f\|_{L^2}^2+\|f\|_{L^2}^2\sim \sum_{1\le i<j\le 3} \|f\|^2_{\mathcal{D}_{\Omega_{ij}}(s,2)}. \eeno
Moreover  for $s\in [0,2]$, we have
\ben\label{lapbel}  \|(-\triangle_{\SS^2})^{s/2}f\|_{L^2}^2+\|f\|_{L^2}^2\sim \sum_{1\le i<j\le 3} \|f\|^2_{\mathcal{D}_{\Omega_{ij}^2}(s/2,2)}.\een
Here  we use notations $\mathcal{D}_{\Omega_{ij}}(s,2)\eqdefa (L^2, \mathcal{D}(\Omega_{ij}))_{s, 2}$ and $\mathcal{D}_{\Omega_{ij}^2}(s/2,2)
\eqdefa (L^2, \mathcal{D}(\Omega_{ij}^2))_{s/2, 2}$.

\end{lem}
\begin{proof} For  $r>0$ and $x=(x_1,x_2)$, we set $$\bar{F}_{13}^+(r,x)\eqdefa r (\vartheta_{13+}f) (rx_1,rx_2, r\sqrt{1-x_1^2-x_2^2})$$ and
\beno \widetilde{F_{13}^+}(r,x)\eqdefa \left\{\begin{aligned} & r( \vartheta_{13+}f)(rx_1,rx_2, r\sqrt{1-x_1^2-x_2^2}), \quad\mbox{if}\quad |x| \le\sqrt{\f45};\\
&  0,\quad\mbox{if}\quad |x| \ge \sqrt{\f45},\end{aligned}\right. . \eeno where we use the fact $\vartheta_{13+} (rx_1,rx_2, r\sqrt{1-x_1^2-x_2^2})=\vartheta_{13+} (x_1,x_2, \sqrt{1-x_1^2-x_2^2})$.
Thanks to \eqref{lb4}, one has
\beno\|(-\triangle_{\SS^2})^{s/2}(\vartheta_{13+}f)\|_{L^2(\R^3)}+\|\vartheta_{13+}f\|_{L^2(\R^3)}\sim \|\bar{F}_{13}^+\|_{L^2((0,\infty);H^s(B_{\f2{\sqrt{5}}}))}. \eeno

Let $T_1: L^2((0,\infty)\times B_{\f2{\sqrt{5}}})\mapsto L^2((0,\infty)\times \R^2)$ be a linear operator defined by
\ben\label{OpT1} (T_1f)(r,x)\eqdefa \left\{\begin{aligned} &  \vartheta_{23+}(x_1,x_2, \sqrt{1-x_1^2-x_2^2})f(r,x), \quad\mbox{if}\quad |x| \le\sqrt{\f45};\\
&  0,\quad\mbox{if}\quad |x| \ge \sqrt{\f45}.\end{aligned}\right. . \een
Then we have \beno \|T_1f\|_{L^2((0,\infty);H^1(\R^2))}\lesssim \|f\|_{L^2((0,\infty);H^1(B_{\f2{\sqrt{5}}}))}, \\ \|T_1f\|_{L^2((0,\infty);L^2(\R^2))}\lesssim \|f\|_{L^2((0,\infty);L^2(B_{\f2{\sqrt{5}}}))}.
 \eeno
Then by real interpolation, we obtain that \beno \|T_1f\|_{L^2((0,\infty);H^s(\R^2))}\lesssim \|f\|_{L^2((0,\infty);H^s(B_{\f2{\sqrt{5}}}))}. \eeno
By the definition of $\bar{F}_{13}^+$, we have $\mathrm{Supp}\bar{F}_{13}^+(r,x) \subset (0,\infty)\times B_{\f{\sqrt{3}}2} $.
Thus if we take $f=\bar{F}_{13}^+ $, then we get
\beno \| \widetilde{F_{13}^+}\|_{L^2((0,\infty);H^s(\R^2))}\lesssim \|\bar{F}_{13}^+(r,x)\|_{L^2((0,\infty);H^s(B_{\f2{\sqrt{5}}}))}.
 \eeno

  Let $T_2:L^2((0,\infty)\times \R^2)\mapsto  L^2((0,\infty)\times B_{\f2{\sqrt{5}}}) $ be a linear operator defined by
\ben\label{OpT2} (T_2f)(r,x)\eqdefa    \vartheta_{23+}(x_1,x_2, \sqrt{1-x_1^2-x_2^2})f(r,x).  \een
Then by the similar argument, we may obtain that
\beno \|T_2f\|_{L^2((0,\infty); H^s(B_{\f2{\sqrt{5}}}))}\lesssim \|f\|_{L^2((0,\infty);H^s(\R^2))}.\eeno
Thus if we  take $f=\widetilde{F_{13}^+}$, then we get
\beno \|\bar{F}_{13}^+ \|_{L^2((0,\infty);H^s(B_{\f2{\sqrt{5}}}))}\lesssim \| \widetilde{F_{13}^+}\|_{L^2((0,\infty);H^s(\R^2))}. \eeno

Therefore, we are led to \beno \|\bar{F}_{13}^+ \|_{L^2((0,\infty);H^s(B_{\f2{\sqrt{5}}}))}\sim \| \widetilde{F_{13}^+}\|_{L^2((0,\infty);H^s(\R^2))}, \eeno
which implies that
\ben\label{lb8}\|(-\triangle_{\SS^2})^{s/2}(\vartheta_{13+}f)\|_{L^2(\R^3)}+\|\vartheta_{13+}f\|_{L^2(\R^3)}\sim \|\widetilde{F_{13}^+}\|_{L^2((0,\infty);H^s(\R^2))}. \een
In what follows, we use the notation $ \|\widetilde{F_{13}^+}\|_{L^2_rH^s_{x}}\eqdefa \|\widetilde{F_{13}^+}\|_{L^2((0,\infty);H^s(\R^2))}.$ It is easy to see
\beno \|\widetilde{F_{13}^+}\|_{L^2_rH^s_{x}}\sim\|\widetilde{F_{13}^+}\|_{L^2_rL^2_{x_1}H^s_{x_2}}+\|\widetilde{F_{13}^+}\|_{L^2_rL^2_{x_2}H^s_{x_1}}. \eeno
Notice that
\beno \|\widetilde{F_{13}^+}\|_{L^2_rL^2_{x_2}H^1_{x_1}}&\sim& \|\Omega_{13} (\vartheta_{13+}f)\|_{L^2(\R^3)}+\| \vartheta_{13+}f\|_{L^2(\R^3)},\\
\|\widetilde{F_{13}^+}\|_{L^2_rL^2_{x_2}L^2_{x_1}}&\sim & \|\vartheta_{13+}f\|_{L^2(\R^3)}.
\eeno
By  real interpolation, one has
\beno  \|\widetilde{F_{13}^+}\|_{L^2_rL^2_{x_1}H^s_{x_2}}\sim \| \vartheta_{13+}f\|_{\mathcal{D}_{\Omega_{13}}(s,2)}. \eeno
Similarly  \beno  \|\widetilde{F_{13}^+}\|_{L^2_rL^2_{x_2}H^s_{x_1}}\sim \| \vartheta_{13+}f\|_{\mathcal{D}_{\Omega_{23}}(s,2)}. \eeno
From which together with \eqref{lb8}, we have
\ben\label{lb9}&&\|(-\triangle_{\SS^2})^{s/2}(\vartheta_{13+}f)\|_{L^2(\R^3)}+\|\vartheta_{13+}f\|_{L^2(\R^3)}\nonumber\\&&\sim\|\widetilde{F_{13}^+}\|_{L^2((0,\infty);H^s(\R^2))}\sim \
 \| \vartheta_{13+}f\|_{\mathcal{D}_{\Omega_{13}}(s,2)}+\| \vartheta_{13+}f\|_{\mathcal{D}_{\Omega_{23}}(s,2)}. \een

Observe that \beno \big(\Omega_{12}(\vartheta_{13+}f)\big)(u)=(x_1\pa_{x_2}-x_2\pa_{x_1})\big(r^{-1}\bar{F}_{13}^+(r,x)\big),\eeno where
$u=(rx_1,rx_2,r\sqrt{1-|x|^2})$.  It yields
\beno   \|\Omega_{12}(\vartheta_{13+}f)\|_{L^2(\R^3)}+\| \vartheta_{13+}f\|_{L^2(\R^3)} \lesssim  \|\widetilde{F_{13}^+}\|_{L^2_rH^1_{x}}. \eeno
Therefore by real interpolation, we deduce that
\beno \| \vartheta_{13+}f\|_{\mathcal{D}_{\Omega_{12}}(s,2)}\lesssim \|\widetilde{F_{13}^+}\|_{L^2_rH^s_{x}}\lesssim  \| \vartheta_{13+}f\|_{\mathcal{D}_{\Omega_{13}}(s,2)}+\| \vartheta_{13+}f\|_{\mathcal{D}_{\Omega_{23}}(s,2)} \eeno
which yields
\ben\label{lb9}\|(-\triangle_{\SS^2})^{s/2}(\vartheta_{13+}f)\|_{L^2(\R^3)}+\|\vartheta_{13+}f\|_{L^2(\R^3)}\sim \sum_{1\le i<j\le 3} \| \vartheta_{13+}f\|_{\mathcal{D}_{\Omega_{ij}}(s,2)}. \een
With the help of \eqref{lb7}, we are led to
 \ben \label{lb10} &&\sum_{m=1}^3 \sum_{1\le i<j\le 3} [\| \vartheta_{1m+}f\|_{\mathcal{D}_{\Omega_{ij}}(s,2)}+\| \vartheta_{1m-}f\|_{\mathcal{D}_{\Omega_{ij}}(s,2)} ]\nonumber\\&&
\sim \|(-\triangle_{\SS^2})^{s/2} f\|_{L^2(\R^3)}+\|f\|_{L^2(\R^3)}. \een

Due to the fact \beno \|\Omega_{ij}(\vartheta_{1m+}f)\|_{L^2(\R^3)}\lesssim \|\Omega_{ij}f\|_{L^2(\R^3)}+\|f\|_{L^2(\R^3)},\eeno
we have \beno \| \vartheta_{1m+}f\|_{\mathcal{D}_{\Omega_{ij}}(s,2)}\lesssim \|  f\|_{\mathcal{D}_{\Omega_{ij}}(s,2)}. \eeno
From which  together with \eqref{lb10}, we derive that \beno   \sum_{1\le i<j\le 3}  \|  f\|_{\mathcal{D}_{\Omega_{ij}}(s,2)}
\sim \|(-\triangle_{\SS^2})^{s/2} f\|_{L^2(\R^3)}+\|f\|_{L^2(\R^3)}.\eeno
We complete the proof to the first equivalence.

 The  interpolation theory indicates \beno    \sum_{1\le i<j\le 3} \|f\|^2_{\mathcal{D}_{\Omega_{ij}}(s,2)}\sim  \sum_{1\le i<j\le 3} \|f\|^2_{\mathcal{D}_{\Omega^2_{ij}}(s/2,2)},\eeno which implies the second equivalence in the case of $0\le s\le1$.
 Next we want to prove that the result still holds for $1<s\le 2$.
 
  We first show \ben\label{lb13}  && \|(-\triangle_{\SS^2})  (\vartheta_{13+}f)\|_{L^2(\SS^2)}^2+\|\vartheta_{13+}f\|_{L^2(\SS^2)}^2\nonumber\\
&&\sim \|F_{13}^+\|_{H^2(B_{\f2{\sqrt{5}}})}^2.\een
It derives from the fact that 
\ben\label{elipitic lb}  (-\triangle_{\SS^2})  (\vartheta_{13+}f)=\mathbf{L} F_{13}^+, \een
where $\mathbf{L}=-(1-x_1^2)\pa^2_1-(1-x_2^2)\pa_2^2+2x_1\pa_1+2x_2\pa_2$. Since  $\mathbf{L}$ is a uniformly elliptic in $B_{\f2{\sqrt{5}}}$ and $F_{13}^+$ vanishes in the boundary of 
$B_{\f2{\sqrt{5}}}$, the standard elliptic estimate implies that 
\beno \|F_{13}^+\|_{H^2(B_{\f2{\sqrt{5}}})}&\lesssim &\|F_{13}^+\|_{L^2(B_{\f2{\sqrt{5}}})}+\|\mathbf{L} F_{13}^+\|_{L^2(B_{\f2{\sqrt{5}}})}\\
&\lesssim & \|(-\triangle_{\SS^2})  (\vartheta_{13+}f)\|_{L^2(\SS^2)}+\|\vartheta_{13+}f\|_{L^2(\SS^2)}, \eeno 
which gives the proof to \eqref{lb13} since the inverse inequality is obviously valid recalling the definition of $-\triangle_{\SS^2}$.
 By real interpolation, \eqref{lb13} yields  that \eqref{lb8} holds for $0\le s\le 2$.

Due to the fact \beno (\sqrt{1-|x|^2}\pa_{x_1})^2\bigg( r^{-1} \bar{F}^+_{13} (r,x)\bigg)= \big(\Omega_{13}^2 (\vartheta_{13+}f)\big)(u), \eeno
where $u=(rx_1,rx_2,r\sqrt{1-|x|^2})$, we derive
\beno \|\widetilde{F_{13}^+}\|_{L^2_rL^2_{x_2}H^2_{x_1}}&\sim& \|\Omega^2_{13} (\vartheta_{13+}f)\|_{L^2(\R^3)}+\| \vartheta_{13+}f\|_{L^2(\R^3)}. \eeno By real interpolation, one has that for  $0\le s\le 2$
\beno  \|\widetilde{F_{13}^+}\|_{L^2_rL^2_{x_1}H^s_{x_2}}\sim \| \vartheta_{13+}f\|_{\mathcal{D}_{\Omega^2_{13}}(s/2,2)}. \eeno
Similarly  for $0\le s\le 2$,  it holds \beno  \|\widetilde{F_{13}^+}\|_{L^2_rL^2_{x_1}H^s_{x_2}}\sim \| \vartheta_{13+}f\|_{\mathcal{D}_{\Omega^2_{23}}(s/2,2)}. \eeno
From which together with \eqref{lb8},  for  $0\le s\le 2$, we have
\ben\label{lb11}&&\|(-\triangle_{\SS^2})^{s/2}(\vartheta_{13+}f)\|_{L^2(\R^3)}+\|\vartheta_{13+}f\|_{L^2(\R^3)}
 \\&&\sim\|\widetilde{F_{13}^+}\|_{L^2((0,\infty);H^s(\R^2))}\sim \
 \| \vartheta_{13+}f\|_{\mathcal{D}_{\Omega^2_{13}}(s/2,2)}+\| \vartheta_{13+}f\|_{\mathcal{D}_{\Omega^2_{23}}(s/2,2)}\notag. \een

Observe that \beno \big(\Omega_{12}(\vartheta_{13+}f)\big)(u)=(x_1\pa_{x_2}-x_2\pa_{x_1})\big(r^{-1}\bar{F}_{13}^+(r,x)\big),\eeno where
$u=(rx_1,rx_2,r\sqrt{1-|x|^2})$.  It yields
\beno   \|\Omega_{12}^2(\vartheta_{13+}f)\|_{L^2(\R^3)}+\| \vartheta_{13+}f\|_{L^2(\R^3)} \lesssim  \|\widetilde{F_{13}^+}\|_{L^2_rH^2_{x}}. \eeno
Therefore by real interpolation, we deduce that
\beno \| \vartheta_{13+}f\|_{\mathcal{D}_{\Omega^2_{12}}(s/2,2)}\lesssim \|\widetilde{F_{13}^+}\|_{L^2_rH^s_{x}}\lesssim  \| \vartheta_{13+}f\|_{\mathcal{D}_{\Omega^2_{13}}(s/2,2)}+\| \vartheta_{13+}f\|_{\mathcal{D}_{\Omega^2_{23}}(s/2,2)} \eeno
which yields
\beno\|(-\triangle_{\SS^2})^{s/2}(\vartheta_{13+}f)\|_{L^2(\R^3)}+\|\vartheta_{13+}f\|_{L^2(\R^3)}\sim \sum_{1\le i<j\le 3} \| \vartheta_{13+}f\|_{\mathcal{D}_{\Omega^2_{ij}}(s/2,2)}. \eeno
From which together with \eqref{lb7}, we get the equivalence \eqref{lapbel}.
\end{proof}

As a consequence,  we show that the $L^2$ norm of  the fractional Laplace-Beltrami operator can be bounded by the weighted Sobolev norm. It explains why the additional weights are needed in  Theorem \ref{thmub}.

\begin{lem}\label{antin4} Suppose $f\in H^s_s(\R^3)$ with $s\ge0$. Then it holds  \beno \|(-\triangle_{\SS^2})^{s/2}f\|_{L^2}\lesssim \|f\|_{H^s_s}.\eeno
\end{lem}

\begin{proof}
Suppose $0\le s\le 2m$ with $m\in \N$. Then we have
 \beno \|(-\triangle_{\SS^2})^{m}f\|_{L^2}^2&=&\sum_{k=-1}^\infty \|(-\triangle_{\SS^2})^{m}\mathcal{P}_kf\|_{L^2}^2\\
 &=& \sum_{k=-1}^\infty \|(\sum_{1\le i<j\le 3} \Omega_{ij}^2)^{m}\mathcal{P}_kf\|_{L^2}^2\\
 &\lesssim& \sum_{k=-1}^\infty 2^{4mk}\|\mathcal{P}_k f\|_{H^{2m}}^2.\eeno
 Since it holds \beno \|(-\triangle_{\SS^2})^{m} \tilde{\mathcal{P}}_kf\|_{L^2}\lesssim  2^{2mk}\| f\|_{H^{2m}}, \|\tilde{\mathcal{P}}_k f\|_{L^2}\lesssim   \|  f\|_{L^2}, \eeno
by real interpolation, we get
\beno \|(-\triangle_{\SS^2})^{s/2}\tilde{\mathcal{P}}_kf\|_{L^2}\lesssim 2^{ks}\| f\|_{H^s}. \eeno
In particular, it yields
\beno \|(-\triangle_{\SS^2})^{s/2} \mathcal{P}_kf\|_{L^2}\lesssim 2^{ks}\|\mathcal{P}_k f\|_{H^s}. \eeno
We finally get
\beno \|(-\triangle_{\SS^2})^{s/2}f\|_{L^2}^2&=&\sum_{k=-1}^\infty \|(-\triangle_{\SS^2})^{s/2}\mathcal{P}_kf\|_{L^2}^2\\&\lesssim&  \sum_{k=-1}^\infty2^{2ks}\|\mathcal{P}_kf\|_{H^s}^2\lesssim \|f\|_{H^s_s}^2. \eeno  This completes the proof of the Lemma.
\end{proof}

Now we are in a position to give the proof to Lemma \ref{lbplim}.

\begin{proof} The result can be reduced to prove 
\beno &&\epsilon^{2s-2}\int_{\sigma,\tau\in \SS^2} \f{|f(\sigma)-f(\tau)|^2}{|\sigma-\tau|^{2+2s}} \mathrm{1}_{|\sigma-\tau|\le \epsilon} d\sigma d\tau+\|f\|_{L^2(\SS^2)}^2\\
&&\sim \sum_{[l(l+1)]^{1/2}\le \epsilon^{-1}} \sum_{m=-l}^l l(l+1)|f_l^m|^2 + \epsilon^{2s-2}\sum_{[l(l+1)]^{1/2}> \epsilon^{-1}} \sum_{m=-l}^l [l(l+1)]^s|f_l^m|^2 +\|f\|_{L^2(\SS^2)}^2,\eeno
where $f_l^m=\int_{\SS^2} f(\sigma)Y_l^m(\sigma)d\sigma.$
Thanks to Lemma \ref{antin5}, we have
\beno  &&\epsilon^{2s-2}\int_{\sigma,\tau\in \SS^2} \f{|f(\sigma)-f(\tau)|^2}{|\sigma-\tau|^{2+2s}} \mathrm{1}_{|\sigma-\tau|\le \epsilon} d\sigma d\tau\\
&&= \epsilon^{2s-2}\sum_{l=0}^\infty\sum_{m=-l}^l |f_l^m|^2 \int_{\sigma,\tau\in \SS^2} \f{|Y_l^m(\sigma)-Y_l^m(\tau)|^2}{|\sigma-\tau|^{2+2s}} \mathrm{1}_{|\sigma-\tau|\le \epsilon} d\sigma d\tau.\eeno
To prove the result, it suffices to estimate the quantity $A_l$ defined by 
\beno A_l\eqdefa \epsilon^{2s-2}\int_{\sigma,\tau\in \SS^2} \f{|Y_l^m(\sigma)-Y_l^m(\tau)|^2}{|\sigma-\tau|^{2+2s}}  \mathrm{1}_{|\sigma-\tau|\le \epsilon} d\sigma. d\tau\eeno
We divide the estimate of $A_l$ into three cases. We will follow the notations used in Lemma \ref{antin1}.
\smallskip

{\it Case 1: $l$ is small.}
We first claim that
\beno  &&\|(-\triangle_{\SS^2})^{1/2} f \|_{L^2(\SS^2)}^2+\| f\|_{L^2(\SS^2)}^2
-\epsilon^2\|(-\triangle_{\SS^2})  f  \|_{L^2(\SS^2)}^2\\ &&\lesssim \epsilon^{2s-2}\int_{\sigma,\tau\in \SS^2} \f{|f(\sigma)-f(\tau)|^2}{|\sigma-\tau|^{2+2s}} \mathrm{1}_{|\sigma-\tau|\le \epsilon} d\sigma d\tau+\|f\|_{L^2(\SS^2)}^2\\&&
\lesssim  \|(-\triangle_{\SS^2})^{1/2} f \|_{L^2(\SS^2)}^2+\| f\|_{L^2(\SS^2)}^2
+\epsilon^2\|(-\triangle_{\SS^2})  f \|_{L^2(\SS^2)}^2.\eeno
In particular, if we choose $f=Y_l^m$, then there exists a universal constant $c_1$ and $c_2$ such that
\ben\label{grlb5} &&(1-c_1[l(l+1)]\epsilon^2)[l(l+1)]+1 
\lesssim A_l+1\lesssim (1+c_2[l(l+1)]\epsilon^2)[l(l+1)]+1.\een
 Then we arrive at that if $[l(l+1)]^{1/2}\le (2c_1)^{-1/2}\epsilon^{-1}$,
\ben\label{grlb3}  A_l\sim l(l+1)+1.
 \een

To prove the claim,  we set \beno I\eqdefa \epsilon^{2s-2}\int_{\sigma,\tau\in \SS^2} \f{|f(\sigma)-f(\tau)|^2}{|\sigma-\tau|^{2+2s}}\vartheta_{13+}^2(\sigma)\mathrm{1}_{|\sigma-\tau|\le \epsilon} d\sigma d\tau+\|f\|_{L^2(\SS^2)}^2.\eeno
Then it is easy to check
\beno  I
&\sim& \epsilon^{2s-2} \int_{\sigma,\tau\in \SS^2} \f{|(\vartheta_{13+}f)(\sigma)-(\vartheta_{13+}f)(\tau)|^2}{|\sigma-\tau|^{2+2s}} \vartheta_{33+}(\sigma)\vartheta_{33+}(\tau)\mathrm{1}_{|\sigma-\tau|\le \epsilon}d\sigma d\tau+\|f\|_{L^2(\SS^2)}^2.\eeno

Suppose  $\sigma=(x_1, x_2, x_3) \in \SS^2$ and  $x=(x_1,x_2)$. Let $F_{13}^+(x)\eqdefa(\vartheta_{13+}f)(x_1,x_2, \sqrt{1-x_1^2-x_2^2})$ and $\Theta_{33}^+(x)=\vartheta_{33+}(x_1,x_2, \sqrt{1-x_1^2-x_2^2})$. Then by change of variables, we have
\beno  && \epsilon^{2s-2}\int_{\sigma,\tau\in \SS^2} \f{|(\vartheta_{13+}f)(\sigma)-(\vartheta_{13+}f)(\tau)|^2}{|\sigma-\tau|^{2+2s}} \vartheta_{33+}(\sigma)\vartheta_{33+}(\tau)\mathrm{1}_{|\sigma-\tau|\le \epsilon}d\sigma d\tau\\
&&\sim\epsilon^{2s-2}\int_{|x|,|y|\le \sqrt{\f56}} \f{|F_{13}^+(x)-F_{13}^+(y)|^2}{\big(|x-y|^2+|\sqrt{1-x_1^2-x_2^2}-\sqrt{1-y_1^2-y_2^2}|^2\big)^{1+s}} \\&&\quad\times\Theta_{33}^+(x)\Theta_{33}^+(y)\f1{\sqrt{1-x_1^2-x_2^2}}
\f1{\sqrt{1-y_1^2-y_2^2}}\mathrm{1}_{|x-y|\lesssim \epsilon}dxdy\\
&&\gtrsim \epsilon^{2s-2}\int_{|x|,|y|\le \sqrt{\f45}} \f{|F_{13}^+(x)-F_{13}^+(y)|^2}{ |x-y|^{2+2s}} \mathrm{1}_{|x-y|\le \epsilon}dxdy\\
&&\gtrsim \epsilon^{2s-2}\int_{|x|\le \sqrt{\f34},|x-y|\le \epsilon} \f{|F_{13}^+(x)-F_{13}^+(y)|^2}{ |x-y|^{2+2s}}dxdy.
 \eeno
Thanks to the Taylor expansion, it yields that
\beno I&\gtrsim& \epsilon^{2s-2}\int_{|x|\le \sqrt{\f34},|x-y|\le \epsilon} \f{|\na F_{13}^+(x)\cdot (y-x)|^2}{ |x-y|^{2+2s}}dxdy\\
&&-\epsilon^{2s-2}\int_0^1\int_{|x|\le \sqrt{\f34},|x-y|\le \epsilon} \f{| \na^2 F_{13}^+(x+\kappa(y-x))|^2||x-y|^4}{ |x-y|^{2+2s}}dxdyd\kappa. \eeno
Note that   \beno &&\epsilon^{2s-2}\int_{|x|\le \sqrt{\f34},|x-y|\le \epsilon} \f{|\na F_{13}^+(x)\cdot (y-x)|^2}{ |x-y|^{2+2s}}dxdy\\
&&=\int_{|x|\le \sqrt{\f34}}|\na F_{13}^+(x)|^2\int_{|h|\le \epsilon}\epsilon^{2s-2} \f{\bigg|\f{\na F_{13}^+(x)}{|\na F_{13}^+(x)|} \cdot \f{h}{|h|}\bigg|^2}{ |h|^{2s}}dhdx \eeno
 and \beno &&\epsilon^{2s-2}\int_0^1\int_{|x|\le \sqrt{\f34},|x-y|\le \epsilon} \f{| \na^2 F_{13}^+(x+\kappa(y-x))|^2||x-y|^4}{ |x-y|^{2+2s}}dxdyd\kappa\\
 &&\lesssim \epsilon^2\| F_{13}^+\|_{H^2(B_{\f2{\sqrt{5}}})}^2.\eeno
 Then by the definition of $\vartheta_{13+}$, we derive that
 \ben\label{grlb1} I&\gtrsim& \|   F_{13}^+\|_{H^1(B_{\f2{\sqrt{5}}})}^2-\epsilon^2 \| F_{13}^+\|_{H^2(B_{\f2{\sqrt{5}}})}^2. \een

On the other hand, following the similar argument, we may get \ben\label{grlb2} I&\lesssim& \|   F_{13}^+\|_{H^1(B_{\f2{\sqrt{5}}})}^2+\epsilon^2 \| F_{13}^+\|_{H^2(B_{\f2{\sqrt{5}}})}^2. \een  Thanks to the facts \eqref{lb12} and \eqref{lb13}, \eqref{grlb1} and \eqref{grlb2} imply that
\beno  &&\|(-\triangle_{\SS^2})^{1/2}(\vartheta_{13+}f)\|_{L^2(\SS^2)}^2+\|\vartheta_{13+}f\|_{L^2(\SS^2)}^2
-\epsilon^2\|(-\triangle_{\SS^2}) (\vartheta_{13+}f)\|_{L^2(\SS^2)}^2\\ &&\lesssim I
\lesssim  \|(-\triangle_{\SS^2})^{1/2}(\vartheta_{13+}f)\|_{L^2(\SS^2)}^2+\|\vartheta_{13+}f\|_{L^2(\SS^2)}^2
+\epsilon^2\|(-\triangle_{\SS^2}) (\vartheta_{13+}f)\|_{L^2(\SS^2)}^2.\eeno
Due to the decomposition \eqref{lb6}, we finally conclude the claim.

{\it Case 2: $l$ is sufficiently large.}  Observe that 
\beno &&A_l=  \epsilon^{2s-2} \int_{\sigma,\tau\in \SS^2} \f{|Y_l^m(\sigma)-Y_l^m(\tau)|^2}{|\sigma-\tau|^{2+2s}}   d\sigma d\tau
-\epsilon^{2s-2} \int_{\sigma,\tau\in \SS^2} \f{|Y_l^m(\sigma)-Y_l^m(\tau)|^2}{|\sigma-\tau|^{2+2s}} \mathrm{1}_{|\sigma-\tau|\ge \epsilon} d\sigma d\tau.
\eeno
Thanks to Lemma \ref{antin1}, there exits a universal constants $c_3$ and $c_4$ such that
\beno   \epsilon^{2s-2}\big([l(l+1)]^s-c_3\epsilon^{-2s}\big)\lesssim A_l\lesssim \epsilon^{2s-2}\big([l(l+1)]^s+c_4\epsilon^{-2s}\big).\eeno
It implies that if $[l(l+1)]^{1/2}\ge 2c_3^{1/2s}\epsilon^{-1}$,
\ben\label{grlb4} A_l
 \sim \epsilon^{2s-2}[l(l+1)]^s+1.  \een

{\it Case 3: $[l(l+1)]^{1/2}\sim \epsilon^{-1}$.}
We claim that in this case, $ A_l\sim l(l+1)+1$. Observe that for any $N\in\N$,
\beno A_l&\ge& N^{2s-2}(\epsilon/N)^{2s-2} \int_{\sigma,\tau\in \SS^2} \f{|Y_l^m(\sigma)-Y_l^m(\tau)|^2}{|\sigma-\tau|^{2+2s}} \mathrm{1}_{|\sigma-\tau|\le \epsilon/N} d\sigma d\tau \\
&\ge& N^{2s-2}\bigg( (\epsilon/N)^{2s-2} \int_{\sigma,\tau\in \SS^2} \f{|Y_l^m(\sigma)-Y_l^m(\tau)|^2}{|\sigma-\tau|^{2+2s}} \mathrm{1}_{|\sigma-\tau|\le \epsilon/N} d\sigma d\tau\bigg).\eeno
From which together with \eqref{grlb5} and \eqref{grlb3}, we derive that if $[l(l+1)]^{1/2}\le (2c_1)^{-1/2}N\epsilon^{-1}$,
\ben\label{grlb6} N^{2s-2}\big( l(l+1)+1\big)\lesssim A_l\lesssim  \big( l(l+1)+1\big). \een
Choose $N\ge 2\sqrt{2}c_3^{1/(2s)}c_1^{1/2}$, then \eqref{grlb3}, \eqref{grlb4} and \eqref{grlb6} yield the claim.

We are in a position to prove the lemma. It is easily obtained from the behavior of $A_l$.
We complete the proof of the lemma. \end{proof}

\subsection{Proof of \eqref{labelint}}
To prove \eqref{labelint},  we  first give several estimates to the  commutator between the Laplace-Beltrami operator and the standard derivatives. 

\begin{lem}\label{antin7}    Suppose $a, b\in\R$ and $\phi$ to be a radial function. Then we have \beno  \mathcal{F}(-\triangle_{\SS^2})^{a/2}&=&(-\triangle_{\SS^2})^{a/2}\mathcal{F}\eeno and \beno
\phi(|D|)(-\triangle_{\SS^2})^{a/2}=(-\triangle_{\SS^2})^{a/2}\phi(|D|).
\eeno  In particular, it holds
\beno \|(1-\triangle_{\SS^2})^{a/2} f\|_{H^b}^2\sim\sum_{p\ge-1}2^{2pb}\|(1-\triangle_{\SS^2})^{a/2} \mathfrak{F}_p f\|_{L^2}^2. \eeno
\end{lem}
\begin{proof} Suppose $f$ is a smooth function. Thanks to Theorem \ref{l2decom}, we have
\beno \big((-\triangle_{\SS^2})^{a/2}f\big)(x)=\sum_{l=0}^\infty\sum_{m=-l}^l (l(l+1))^{a/2} Y^m_l(\sigma) f^m_l(r), \eeno where $x=r\sigma$ and   $\sigma\in \SS^2$.
Then if $\xi=\rho\tau$ with $\tau\in \SS^2$, then \beno  \mathcal{F}\big((-\triangle_{\SS^2})^{a/2}f\big)(\xi)&=&\sum_{l=0}^\infty\sum_{m=-l}^l (l(l+1))^{a/2} \mathcal{F}(Y^m_l  f^m_l)(\xi)\\
&=&\sum_{l=0}^\infty\sum_{m=-l}^l  (l(l+1))^{a/2} Y_l^m(\tau)W_l^m(\rho),
\eeno where we use \eqref{sphar} to assume that $\mathcal{F}(Y^m_l f^m_l)(\xi)=Y_l^m(\tau)W_l^m(\rho)$.

On the other hand, using the same notation, we have
\beno (\mathcal{F}f)(\xi)=\sum_{l=0}^\infty\sum_{m=-l}^l Y_l^m(\tau)W_l^m(\rho),  \eeno which implies
\beno  (-\triangle_{\SS^2})^{a/2}(\mathcal{F}f)(\xi)=\sum_{l=0}^\infty\sum_{m=-l}^l (l(l+1))^{a/2} Y_l^m(\tau)W_l^m(\rho)=\mathcal{F}\big((-\triangle_{\SS^2})^{a/2}f\big)(\xi). \eeno
This gives the first equality.

Observe that \beno  \mathcal{F}\phi(|D|)(-\triangle_{\SS^2})^{a/2}&=&\phi\mathcal{F}(-\triangle_{\SS^2})^{a/2}
\\&=&\phi(-\triangle_{\SS^2})^{a/2} \mathcal{F}= (-\triangle_{\SS^2})^{a/2}\phi \mathcal{F},
\eeno where we use the fact that $\phi$ is a radial function in the last equality.

On the other hand, we have
\beno \mathcal{F}(-\triangle_{\SS^2})^{a/2}\phi(|D|)=(-\triangle_{\SS^2})^{a/2}\mathcal{F}\phi(|D|)=(-\triangle_{\SS^2})^{a/2}\phi\mathcal{F}, \eeno
which is enough to yield the second equality in the lemma. 

Finally we give the proof to the last equivalence.  It is derived  from the facts
\beno \|(1-\triangle_{\SS^2})^{a/2} f\|_{H^b}^2\sim\sum_{p\ge-1}2^{2pb}\|\mathfrak{F}_p(1-\triangle_{\SS^2})^{a/2}  f\|_{L^2}^2\eeno 
and \beno  &&\mathcal{F}(\mathfrak{F}_p(1-\triangle_{\SS^2})^{a/2} )=\varphi(2^{-p}\cdot)\mathcal{F}(1-\triangle_{\SS^2})^{a/2}=\varphi(2^{-p}\cdot)(1-\triangle_{\SS^2})^{a/2}\mathcal{F}\\&&=
(1-\triangle_{\SS^2})^{a/2}\varphi(2^{-p}\cdot)\mathcal{F}=(1-\triangle_{\SS^2})^{a/2}\mathcal{F}\mathfrak{F}_p=\mathcal{F}(1-\triangle_{\SS^2})^{a/2}\mathfrak{F}_p,\eeno
where $\varphi(2^{-p}\cdot)(\xi)\eqdefa \varphi(2^{-p}\xi)=\varphi(2^{-p}|\xi|)$, the multiplier of the operator $\mathfrak{F}_p$.
We complete the proof of the lemma. \end{proof}

\begin{lem}\label{antin8} Suppose  $a, b \ge 0$, $m\in\N$  and $f$ is a smooth function. Then
we have \ben\label{equ1}  \sum_{1\le i<j\le 3}\|\Omega_{ij} f\|_{H^a}&\sim& \|(-\triangle_{\SS^2})^{1/2} f\|_{H^a},\\
\sum_{1\le i<j\le 3}\|(-\triangle_{\SS^2})^{a/2}\Omega_{ij} f\|_{H^m}+\|f\|_{H^m}&\sim& \|(1-\triangle_{\SS^2})^{(a+1)/2} f\|_{H^m},\label{equ2}\een
and
\ben\label{equ3}\|(1-\triangle_{\SS^2})^{a/2} f\|_{H^m}\sim \sum_{|\alpha|\le m} \|(1-\triangle_{\SS^2})^{a/2}\pa^\alpha f\|_{L^2}.\een
Moreover, it holds\beno \|(-\triangle_{\SS^2})^{a/2} f\|_{H^b}\lesssim  \|(-\triangle_{\SS^2})^{(a+b)/2} f\|_{L^2}+\| f\|_{H^{a+b}}. \eeno
\end{lem}

\begin{proof}
  (i). We first give the proof to the last inequality.
Thanks to Theorem \ref{l2decom}, we have for $f\in L^2$,
\beno f(x)=\sum_{l=0}^\infty\sum_{m=-l}^l (\mathcal{B}^m_l f)(x),\eeno where   $x=r\sigma$ with $r\ge0$ and $\sigma\in\SS^2$ and  $\mathcal{B}^m_l(x)\eqdefa f^m_l(r) Y^m_l(\sigma)$. 
Then one has
\beno  \|(-\triangle_{\SS^2})^{a/2} f\|_{H^b}^2&\sim& \sum_{k\ge-1}2^{2kb}\|\mathfrak{F}_k(-\triangle_{\SS^2})^{a/2} f\|_{L^2}^2\\
&\sim& \sum_{k\ge -1}2^{2kb}\|(-\triangle_{\SS^2})^{a/2}\mathfrak{F}_k f\|_{L^2}^2\\
&\sim& \sum_{k\ge -1}\sum_{l=0}^\infty\sum_{m=-l}^l 2^{2kb}(l(l+1))^a\| \mathcal{B}^m_l(\mathfrak{F}_k f)\|_{L^2}^2\\
&\lesssim& \sum_{k\ge -1}\sum_{l=0}^\infty\sum_{m=-l}^l  (2^{2k(a+b)}+(l(l+1))^{a+b})\| \mathcal{B}^m_l(\mathfrak{F}_k f)\|_{L^2}^2\\
&\lesssim& \|(-\triangle_{\SS^2})^{(a+b)/2} f\|_{L^2}^2+\| f\|_{H^{a+b}}^2.
 \eeno
\medskip

(ii). Now we turn to proof of \eqref{equ1}.
Thanks to the fact $\mathcal{F} \Omega_{ij}=-\Omega_{ij}\mathcal{F}$ if $i\neq j$, we deduce that if $i\neq j$,
\beno \mathfrak{F}_k \Omega_{ij}=-\Omega_{ij} \mathfrak{F}_k. \eeno Then we have
\beno  \sum_{1\le i<j\le 3}\|\Omega_{ij}  f\|_{H^a(\R^3)}^2\sim \sum_{k\ge-1}  \sum_{1\le i<j\le 3} 2^{2ka}\|\Omega_{ij} \mathfrak{F}_k f\|_{L^2}^2,\eeno
which yields
\beno  \sum_{1\le i<j\le 3}\|\Omega_{ij}  f\|_{H^a(\R^3)}^2&\sim& \sum_{k\ge-1}  2^{2ka}\|(-\triangle_{\SS^2})^{1/2}\mathfrak{F}_k f\|_{L^2}^2\\
&\sim& \sum_{k\ge-1}  2^{2ka}\|\mathfrak{F}_k (-\triangle_{\SS^2})^{1/2} f\|_{L^2}^2\sim \|(-\triangle_{\SS^2})^{1/2}f\|_{H^a}^2. \eeno
This gives \eqref{equ1}.
\medskip

(iii).  We divide the proof of \eqref{equ2} into two steps.

{\it Step 1: $m=0$.} We want to prove 
\ben\label{equ2m0} \sum_{1\le i<j\le 3} \|(-\triangle_{\SS^2})^{a/2}\Omega_{ij} f\|_{L^2}+\|f\|_{L^2}&\sim& \|(1-\triangle_{\SS^2})^{(a+1)/2} f\|_{L^2}. \een
We begin with the case $0\le a\le 1$. 
Observing that \ben\label{e1}&&\langle \Omega_{mn}\Omega_{ij}f, \Omega_{mn}\Omega_{ij}f\rangle- \langle \Omega_{mn}\Omega_{mn}f, \Omega_{ij}\Omega_{ij}f\rangle\notag\\
&&=\langle [\Omega_{mn}, \Omega_{ij}f], \Omega_{mn}\Omega_{ij}f\rangle+\langle \big[[\Omega_{ij}, \Omega_{mn}], \Omega_{mn}\big]f, \Omega_{ij}f\rangle-\langle [\Omega_{ij}, \Omega_{mn}]f, \Omega_{mn}\Omega_{ij}f\rangle\een
and \beno [\Omega_{mn}, \Omega_{ij}]f=\delta_{ni}\Omega_{mj}+\delta_{jn}\Omega_{im}-\delta_{jm}\Omega_{in}-\delta_{mi}\Omega_{nj},\eeno
we deduce that
\ben\label{e4}  \sum_{1\le m<n \le 3}\sum_{1\le i<j \le 3} \|\Omega_{mn}\Omega_{ij}f\|_{L^2} \lesssim \|(-\triangle_{\SS^2})^{1/2} f\|_{L^2}+\|(-\triangle_{\SS^2}) f\|_{L^2}.\een
Due to the fact  \beno \|(-\triangle_{\SS^2})^{1/2}\Omega_{ij}f\|_{L^2}
\sim \sum_{1\le m<n \le 3}\|\Omega_{mn}\Omega_{ij}f\|_{L^2},\eeno
 we get
\ben\label{e2} \sum_{1\le i<j\le 3} \|(-\triangle_{\SS^2})^{1/2}\Omega_{ij}f\|_{L^2}\lesssim \| \triangle_{\SS^2} f\|_{L^2}+\|f\|_{L^2}. \een
On the other hand, by \eqref{e1}, we  obtain \beno \|(-\triangle_{\SS^2}) f\|_{L^2}&\lesssim& \sum_{1\le m<n \le 3}\sum_{1\le i<j \le 3}
 \|\Omega_{mn}\Omega_{ij}f\|_{L^2} + \|(-\triangle_{\SS^2})^{1/2} f\|_{L^2}
\\&\lesssim & \sum_{1\le i<j \le 3}\|(-\triangle_{\SS^2})^{1/2}\Omega_{ij}f\|_{L^2}+ \eta\|(-\triangle_{\SS^2}) f\|_{L^2}+C_\eta\|f\|_{L^2}, \eeno
where $\eta$ is a small constant. From which together with \eqref{e2}, we obtain  \eqref{equ2m0} with $a=1$.

We turn to the case $0<a<1$.
Due to Lemma \ref{antin2},  for smooth functions $g$ and $f$, we have \beno  \|(-\triangle_{\SS^2})^{a/2} (fg)\|_{L^2(\SS^2)}
\lesssim (\|\na_{\SS^2}g\|_{L^\infty(\SS^2)}+\|g\|_{L^\infty(\SS^2)})\|(1-\triangle_{\SS^2})^{a/2}f\|_{L^2(\SS^2)}.\eeno
It  implies \ben \label{e6} \|(-\triangle_{\SS^2})^{a/2}\vartheta_{13+}\Omega_{ij}f\|_{L^2}&=&\|(-\triangle_{\SS^2})^{a/2}[\Omega_{ij}(\vartheta_{13+}f)-
(\Omega_{ij}\vartheta_{13+})f]\|_{L^2}\nonumber\\
&\gtrsim&
\|(-\triangle_{\SS^2})^{a/2}\Omega_{ij}(\vartheta_{13+}f)\|_{L^2}-\|(1-\triangle_{\SS^2})^{a/2}f\|_{L^2}.
\een
Using the notations introduced in Lemma \ref{antin2}, we have
\beno \Omega_{ij}(\vartheta_{13+}f)(u)=r^{-1}\mathcal{A}_k  \bar{F}^+_{13}(r,x_1,x_2), \eeno
with $u=(rx_1,rx_2, r\sqrt{1-|x|^2})$ and $1\le k\le 3$. Here the operator $\mathcal{A}_k$ is defined by  \beno \mathcal{A}_1\eqdefa \sqrt{1-|x|^2}\pa_{x_1}, \mathcal{A}_2\eqdefa \sqrt{1-|x|^2}\pa_{x_2}, \mathcal{A}_3\eqdefa x_1\pa_{x_2}-x_2\pa_{x_1}. \eeno
Therefore, by \eqref{lb8}, we obtain that
\beno &&\sum_{1\le i<j\le 3}\|(-\triangle_{\SS^2})^{a/2}\Omega_{ij}(\vartheta_{13+}f)\|_{L^2}+\|(1-\triangle_{\SS^2})^{1/2}(\vartheta_{13+}f)\|_{L^2}
\\&&\sim  \sum_{1\le k\le 3}\|\mathcal{A}_k  \widetilde{F^+_{13}}\|_{L^2_rH^s_x}+\|\widetilde{F^+_{13}}\|_{L^2_rH^{1}_x}
\\&&\sim   \| \widetilde{F^+_{13}}\|_{L^2_rH^{1+s}_x}\sim \|(1-\triangle_{\SS^2})^{(a+1)/2}(\vartheta_{13+}f)\|_{L^2},\eeno
where \eqref{lb11} is used in the last equivalence.  From which together with \eqref{e6} and
\beno  \|\Omega_{ij}f\|_{L^2}+ \|(-\triangle_{\SS^2})^{a/2}\Omega_{ij}f\|_{L^2}\sim
\sum_{m=1}^3 \big( \|(-\triangle_{\SS^2})^{a/2}\vartheta_{1m+}\Omega_{ij}f\|_{L^2}+ \|(-\triangle_{\SS^2})^{a/2}\vartheta_{1m-}\Omega_{ij}f\|_{L^2}\big)+\|\Omega_{ij}f\|_{L^2},\eeno
we have
\beno \sum_{1\le i<j\le 3 } \|(-\triangle_{\SS^2})^{a/2}\Omega_{ij}f\|_{L^2} +\|(1-\triangle_{\SS^2})^{1/2}f\|_{L^2}\sim
\|(1-\triangle_{\SS^2})^{(a+1)/2}f\|_{L^2},\eeno
which implies \eqref{equ2m0} with $0<a<1$. It completes the proof to \eqref{equ2m0} for $a\in [0,1]$.

Next we prove that \eqref{equ2m0} holds for $1<a\le 2$. Suppose $a=1+s$ with $0<s\le 1$. Thanks to the fact that \eqref{equ2m0} holds for $0\le  a\le 1$, we have
\beno  \sum_{1\le i<j\le 3 }\|(-\triangle_{\SS^2})^{a/2} \Omega_{ij}f\|_{L^2}+\|f\|_{L^2}&\sim&\sum_{1\le i<j\le 3 } \|(-\triangle_{\SS^2})^{s+1/2}  \Omega_{ij}f\|_{L^2}+\|f\|_{L^2}\\
&\sim&  \sum_{1\le i<j\le 3 }\big(\sum_{1\le m<n\le 3 }\|(-\triangle_{\SS^2})^{s/2}\Omega_{mn} \Omega_{ij}f\|_{L^2}+\|\Omega_{ij}f\|_{L^2}\big)+\|f\|_{L^2}\\
&\sim&  \sum_{1\le i<j\le 3 } \sum_{1\le m<n\le 3 }\sum_{p=1}^3\big(\|(-\triangle_{\SS^2})^{s/2}\big(\vartheta_{1p+}\Omega_{mn} \Omega_{ij} f)\|_{L^2}\\&&
+\|(-\triangle_{\SS^2})^{s/2}\big(\vartheta_{1p-}\Omega_{mn} \Omega_{ij} f)\|_{L^2}\big)+\|(1-\triangle_{\SS^2})^{1/2}f\|_{L^2} .
\eeno
 Notice that
\beno  (-\triangle_{\SS^2})^{s/2}(\vartheta_{13+}\Omega_{mn} \Omega_{ij}f)&=&(-\triangle_{\SS^2})^{s/2}\Omega_{mn} \Omega_{ij}(\vartheta_{13+}f)-
 (-\triangle_{\SS^2})^{s/2}\big((\Omega_{mn}\vartheta_{13+} )(\Omega_{ij}f)\big)\\&& -
 (-\triangle_{\SS^2})^{s/2}\big((\Omega_{ij}\vartheta_{13+} )(\Omega_{mn}f)\big)-  (-\triangle_{\SS^2})^{s/2}\big((\Omega_{mn}\Omega_{ij}\vartheta_{13+}) f\big),\eeno
  \beno  \| (-\triangle_{\SS^2})^{s/2}\big((\Omega_{mn}\vartheta_{13+} )(\Omega_{ij}f)\big)\|_{L^2}&\lesssim& \|(1-\triangle_{\SS^2})^{s/2} \Omega_{ij}f\|_{L^2}\\
&\lesssim&\|(1-\triangle_{\SS^2})^{a/2}f\|_{L^2}
\eeno
and \beno \|(-\triangle_{\SS^2})^{s/2}\big((\Omega_{mn}\Omega_{ij}\vartheta_{13+}) f\big)\|_{L^2}\lesssim \|(1-\triangle_{\SS^2})^{s/2} f\|_{L^2},\eeno
then we get
\beno \|(-\triangle_{\SS^2})^{s/2}(\vartheta_{13+}\Omega_{mn} \Omega_{ij}f)\|_{L^2}+\|(1-\triangle_{\SS^2})^{a/2}f\|_{L^2}\sim \|(-\triangle_{\SS^2})^{s/2}\Omega_{mn} \Omega_{ij}(\vartheta_{13+}f)\|_{L^2} +\|(1-\triangle_{\SS^2})^{a/2}f\|_{L^2}.\eeno
Thus we have 
\beno  &&\sum_{1\le i<j\le 3 }\|(-\triangle_{\SS^2})^{a/2} \Omega_{ij}f\|_{L^2}+\|f\|_{L^2}+\|(1-\triangle_{\SS^2})^{a/2}f\|_{L^2}
\\ &&\sim \sum_{1\le i<j\le 3 } \sum_{1\le m<n\le 3 }\sum_{p=1}^3\big(\|(-\triangle_{\SS^2})^{s/2}\big(\Omega_{mn} \Omega_{ij}\vartheta_{1p+}f\big)\|_{L^2} 
+\|(-\triangle_{\SS^2})^{s/2}\big(\Omega_{mn} \Omega_{ij}\vartheta_{1p-} f)\|_{L^2}\big)+\|(1-\triangle_{\SS^2})^{a/2}f\|_{L^2}. \eeno
If we show 
\ben\label{e7} \sum_{1\le i<j\le 3 } \sum_{1\le m<n\le 3 }  \|(-\triangle_{\SS^2})^{s/2}\Omega_{mn} \Omega_{ij}(\vartheta_{13+}f)\|_{L^2}+\|\vartheta_{13+}f\|_{L^2} \sim \|(1-\triangle_{\SS^2})^{(s+2)/2}(\vartheta_{13+}f)\|_{L^2},\een
then by Young inequality, \beno  \|(1-\triangle_{\SS^2})^{a/2}f\|_{L^2}\le \eta \|(1-\triangle_{\SS^2})^{(a+1)/2}f\|_{L^2}+C_\eta \|f\|_{L^2},\eeno we conclude the equivalence \eqref{equ2m0} with $1\le a\le 2$.
It remains to prove \eqref{e7}. On one hand, we   observe that 
\beno &&\sum_{1\le i<j\le 3 } \sum_{1\le m<n\le 3 }\big( \|(-\triangle_{\SS^2})^{s/2}\Omega_{mn} \Omega_{ij}(\vartheta_{13+}f)\|_{L^2}+\|\Omega_{mn} \Omega_{ij}(\vartheta_{13+}f)\|_{L^2}\big)+\|\vartheta_{13+}f\|_{L^2}\\
&&\sim \sum_{k=1}^3 \sum_{p=1}^3 \|\mathcal{A}_k  \mathcal{A}_p\widetilde{F}_{13+}\|_{L^2_rH^s_x}+\|\widetilde{F}_{13+}\|_{L^2_rH^{2}_x}
\\&&\sim   \| \widetilde{F}_{13+}\|_{L^2_rH^{2+s}_x}. 
\eeno
On the other hand, it is easy to check that
\beno && \|(-\triangle_{\SS^2})^{(s+2)/2}(\vartheta_{13+}f)\|_{L^2}+\|\vartheta_{13+}f\|_{L^2}\\&&
\sim  \|(-\triangle_{\SS^2})^{s/2}(-\triangle_{\SS^2})(\vartheta_{13+}f)\|_{L^2}+\|\vartheta_{13+}f\|_{L^2}\\
&&\sim \|\mathbf{L} \widetilde{F}_{13+}\|_{L^2_rH^{s}_x}+\|\widetilde{F}_{13+}\|_{L^2_rL^2_x}\\
&& \sim\|\widetilde{F}_{13+}\|_{L^2_rH^{2+s}_x},
\eeno
where we use  \eqref{lb13} and \eqref{elipitic lb}. We end the proof to \eqref{e7} by these two equivalences. Then we get \eqref{equ2m0} with $a\in[0,2]$.

 To complete the proof,  we first use inductive method to  show that for $a\ge0$,
\ben\label{e3} \|(-\triangle_{\SS^2})^{a/2}\Omega_{ij}f\|_{L^2}\lesssim \|(1-\triangle_{\SS^2})^{a/2}(-\triangle_{\SS^2})^{1/2} f\|_{L^2}. \een
Since \eqref{e3} holds for $0\le a\le2$,
 we assume \eqref{e3} holds for $a\le m$ with $m\ge2$. Suppose $a\in [m, m+1]$, then we have
 \beno \|(-\triangle_{\SS^2})^{a/2}\Omega_{ij}f\|_{L^2}&=&\|(-\triangle_{\SS^2})^{a/2-1}\sum_{1\le m<n \le 3}\Omega_{mn}^2\Omega_{ij}f\|_{L^2}\\
 &\lesssim& \|(-\triangle_{\SS^2})^{a/2-1} \Omega_{ij}(-\triangle_{\SS^2})f\|_{L^2}+ \|(-\triangle_{\SS^2})^{a/2-1}[\sum_{1\le m<n \le 3}\Omega_{mn}^2,\Omega_{ij}]f\|_{L^2}\\
 &\lesssim&\|(-\triangle_{\SS^2})^{a/2-1} \Omega_{ij}(-\triangle_{\SS^2})f\|_{L^2}+\sum_{1\le m<n \le 3}\sum_{1\le i<j \le 3}
 \| (-\triangle_{\SS^2})^{a/2-1} \Omega_{mn}\Omega_{ij}f\|_{L^2}
 \\ &\lesssim & \|(1-\triangle_{\SS^2})^{a/2}(-\triangle_{\SS^2})^{1/2} f\|_{L^2} ,\eeno
where we use the inductive assumption in the last inequality and the fact
\beno  [\sum_{1\le m<n \le 3} \Omega^2_{mn},\Omega_{ij}]=\sum_{1\le m<n \le 3} \big(\Omega_{mn}[\Omega_{mn},\Omega_{ij}]+ [\Omega_{mn},\Omega_{ij}]\Omega_{mn}\big).\eeno We complete the inductive argument to derive \eqref{e3}.

Now we are ready to prove \eqref{equ2m0}. We may assume that \eqref{equ2m0} holds for $a\le m$ with $m\ge1$.  Suppose $a\in [m,m+1]$. We have
\beno && \sum_{1\le i<j\le 3} \|(1-\triangle_{\SS^2})^{a/2}\Omega_{ij}f\|_{L^2} +\|(1-\triangle_{\SS^2})^{1/2}f\|_{L^2}
\\&&\sim \sum_{1\le i<j\le 3} \|(1-\triangle_{\SS^2})^{(a-1)/2+1/2}\Omega_{ij}f\|_{L^2} +\|(1-\triangle_{\SS^2})^{1/2}f\|_{L^2}
\\&&\sim \sum_{1\le i<j\le 3} \sum_{1\le m<n\le 3} \|(-\triangle_{\SS^2})^{(a-1)/2}\Omega_{mn}\Omega_{ij}f\|_{L^2} +\|(1-\triangle_{\SS^2})^{1/2}\Omega_{ij}f\|_{L^2} +\|(1-\triangle_{\SS^2})^{1/2}f\|_{L^2}.\eeno
Due the fact \beno
&&\|(-\triangle_{\SS^2})^{(a+1)/2}  f\|_{L^2}=\|(-\triangle_{\SS^2})^{(a-1)/2}(-\triangle_{\SS^2}) f\|_{L^2}\\&&\le\sum_{1\le i<j\le 3} \sum_{1\le m<n\le 3} \|(-\triangle_{\SS^2})^{(a-1)/2}\Omega_{mn}\Omega_{ij}f\|_{L^2}\\
&&\lesssim \|(1-\triangle_{\SS^2})^{a/2}(-\triangle_{\SS^2})^{1/2} f\|_{L^2}\le\|(1-\triangle_{\SS^2})^{(a+1)/2}  f\|_{L^2},
 \eeno where we use \eqref{e3} in the second inequality, we finally derive the desired result  and ends the inductive argument to the equivalence \eqref{equ2m0}.

{\it Step 2: $m\in\N$.} By Lemma \ref{antin7} and  \eqref{equ2m0}, we have that for $a\ge1$,
 \beno \|(1-\triangle_{\SS^2})^{a/2} f\|_{H^m}^2 &\sim&  \sum_{k\ge-1} 2^{2mk} \|(1-\triangle_{\SS^2})^{a/2}\mathfrak{F}_kf\|_{L^2}^2\\
 &\sim&   \sum_{k\ge-1} 2^{2mk} \big(\sum_{1\le i<j\le3}\|( -\triangle_{\SS^2})^{(a-1)/2}\Omega_{ij}\mathfrak{F}_kf\|_{L^2}^2+ \|\mathfrak{F}_kf\|_{L^2}^2)
 \\
 &\sim&   \sum_{k\ge-1} 2^{2mk} \big(\sum_{1\le i<j\le3}\|( -\triangle_{\SS^2})^{(a-1)/2}\mathfrak{F}_k\Omega_{ij}f\|_{L^2}^2+ \| \mathfrak{F}_kf\|_{L^2}^2)\\
 &\sim&\sum_{1\le i<j\le3}\|( -\triangle_{\SS^2})^{(a-1)/2}\Omega_{ij}f\|_{H^m}^2+ \| f\|_{H^m}^2,
 \eeno where we use  the fact
 $\Omega_{ij}\mathcal{F}=-\mathcal{F}\Omega_{ij}$ if $i\neq j$. We complete the proof to \eqref{equ2}.

\medskip
(iv). The proof of \eqref{equ3} falls in three steps. 

{\it Step 1: Proof of  \eqref{equ3} with $a=0$ and $m=1$.} We want to prove that for $0\le a\le 1$,   \ben\label{mx1} \|(1-\triangle_{\SS^2})^{a/2} f\|_{H^1} \sim \sum_{|\alpha|\le 1} \|(1-\triangle_{\SS^2})^{a/2}\pa^\alpha f\|_{L^2}. \een
Obviously \eqref{mx1} holds for $a=0$.  Then we separate the proof of \eqref{mx1} into two cases.

{\it Case1:  $0<a<1$.} Indeed, by Plancherel theorem and Lemma \ref{antin7}, we have
\beno && \sum_{1\le i\le 3}\big(\|(-\triangle_{\SS^2})^{a/2} \pa_i f\|_{L^2}+\|\pa_i f\|_{L^2}\big)
= \sum_{1\le i\le 3}\big(\|(-\triangle_{\SS^2})^{a/2} \xi_i \mathcal{F}f\|_{L^2}+\|\xi_i \mathcal{F}f\|_{L^2}\big). \eeno
Due to Lemma \ref{antin2} and by  setting $\xi=r\sigma=(r\sigma_1, r\sigma_2, r\sigma_3)$, we have
\beno && \sum_{1\le i\le 3}\|(-\triangle_{\SS^2})^{a/2} \pa_i f\|_{L^2}^2+\| f\|_{H^1}^2
\\&&\sim
\sum_{1\le i\le 3}\int_{\sigma,\tau\in \SS^2, r>0} \f{|r\sigma_i (\mathcal{F}f)(r\sigma)-r\tau_i(\mathcal{F}f)(r\tau)|^2}{|\sigma-\tau|^{2+2a}}r^2d\sigma d\tau dr +\|f\|_{H^1}^2. \eeno
Thanks to the observation \beno &&\f12|\tau_i|^2|r (\mathcal{F}f)(r\sigma)-r(\mathcal{F}f)(r\tau)|^2-2|\sigma_i-\tau_i|^2|r(\mathcal{F}f)(r\sigma)|^2\\&&\le|r\sigma_i (\mathcal{F}f)(r\sigma)-r\tau_i(\mathcal{F}f)(r\tau)|^2\\
&&\lesssim |\tau_i|^2|r (\mathcal{F}f)(r\sigma)-r(\mathcal{F}f)(r\tau)|^2+|\sigma_i-\tau_i|^2|r(\mathcal{F}f)(r\sigma)|^2,\eeno
we deduce that
\beno && \sum_{1\le i\le 3}\|(-\triangle_{\SS^2})^{a/2} \pa_i f\|_{L^2}^2+\|f\|_{H^1}^2\\&&\sim
\int_{\sigma,\tau\in \SS^2, r>0} \f{|r (\mathcal{F}f)(r\sigma)-r(\mathcal{F}f)(r\tau)|^2}{|\sigma-\tau|^{2+2a}}r^2d\sigma d\tau dr +\|f\|_{H^1}^2,
\eeno
which implies
\ben\label{mx2}  \sum_{1\le i\le 3} \|(-\triangle_{\SS^2})^{a/2} \pa_i f\|_{L^2}+\|  f\|_{H^1} \sim \|(-\triangle_{\SS^2})^{a/2}(|D|f)\|_{L^2} +\|f\|_{H^1} .\een
This is enough to get \eqref{mx1} for $0<a<1$. 

{\it Case 2: $a=1$.} It is not difficult to check
 \beno &&\sum_{|\alpha|\le1} \|(-\triangle_{\SS^2})^{1/2}\pa^\alpha f\|_{L^2}+\|f\|_{H^1}\\
 &&\sim  \sum_{|\alpha|\le1}\sum_{1\le m<n\le 3}\|\Omega_{mn}\pa^\alpha f\|_{L^2}+\|f\|_{H^1}\\
 &&\sim  \sum_{|\alpha|\le1}\sum_{1\le m<n\le 3}\|\pa^\alpha \Omega_{mn}f\|_{L^2}+\|f\|_{H^1}
 \\
 &&\sim  \sum_{1\le m<n\le 3}\| \Omega_{mn}f\|_{H^1}+\|f\|_{H^1} \\
 &&\sim   \|(1-\triangle_{\SS^2})^{1/2}f\|_{H^1},\eeno
 where we use \eqref{equ1} in the last equivalence.
 We complete the proof to \eqref{mx1}.

 {\it Step 2: Proof of \eqref{equ3} with $a\in[0,1]$ and $m\in\N$.}   Thanks to \eqref{mx1},
we assume \eqref{equ1} holds for $m\le N-1$ with $N\ge2$.
Recall that \beno  &&\sum_{|\alpha|\le N} \|(1-\triangle_{\SS^2})^{a/2}\pa^\alpha f\|_{L^2}\\
&&\sim\sum_{|\alpha|\le N-1} \|(1-\triangle_{\SS^2})^{a/2}\pa^\alpha f\|_{L^2}+\sum_{|\alpha|=N-1}\sum_{i=1}^3 \|(1-\triangle_{\SS^2})^{a/2}\pa_i\pa^\alpha f\|_{L^2}. \eeno
Due to the result in {\it Step 1}, we have \beno  &&\sum_{|\alpha|\le N-1} \|(1-\triangle_{\SS^2})^{a/2}\pa^\alpha f\|_{L^2}+\sum_{|\alpha|=N-1}\sum_{i=1}^3 \|(1-\triangle_{\SS^2})^{a/2}\pa_i\pa^\alpha f\|_{L^2}\\
&&\sim \sum_{|\alpha|\le N-1} \|(1-\triangle_{\SS^2})^{a/2}\pa^\alpha f\|_{L^2}+\sum_{|\alpha|=N-1} \|(1-\triangle_{\SS^2})^{a/2}\langle D\rangle\pa^\alpha f\|_{L^2}.
\eeno  From which together with  the assumption that \eqref{equ1} holds for $N-1$,
we have  \beno \sum_{|\alpha|\le N} \|(1-\triangle_{\SS^2})^{a/2}\pa^\alpha f\|_{L^2}&\sim&   \|(1-\triangle_{\SS^2})^{a/2}\langle D\rangle^{N-2}\langle D\rangle f\|_{L^2}
+\sum_{|\alpha|=N-1} \|(1-\triangle_{\SS^2})^{a/2}\pa^\alpha\langle D\rangle f\|_{L^2}.\eeno
 We deduce that
\beno &&\sum_{|\alpha|\le N} \|(1-\triangle_{\SS^2})^{a/2}\pa^\alpha f\|_{L^2}\\
&&\sim \sum_{|\alpha|\le N-2}  \|(1-\triangle_{\SS^2})^{a/2}\pa^\alpha \langle D\rangle f\|_{L^2}
+\sum_{|\alpha|=N-1} \|(1-\triangle_{\SS^2})^{a/2}\pa^\alpha\langle D\rangle f\|_{L^2}
\\&&\sim \|(1-\triangle_{\SS^2})^{a/2} f\|_{H^N},
\eeno which completes the inductive argument to derive \eqref{equ3} with $a\in[0,1]$.

{\it Step 3: Proof of  \eqref{equ3} with $a\ge0$ and $m\in\N$.} We still use the inductive method. Suppose \eqref{equ3} holds for $a\le N$ with $N\ge1$. Suppose now $a\in [N, N+1]$. Due to \eqref{equ2} and the inductive assumption,  we have \beno  
\|( 1 -\triangle_{\SS^2})^{a/2}f\|_{H^m} &\sim&    \sum_{1\le i<j\le3}\|( -\triangle_{\SS^2})^{(a-1)/2} \Omega_{ij}f\|_{H^m}+ \| f\|_{H^m} \\
 &\sim&    \sum_{1\le i<j\le3}\sum_{|\alpha|\le m}\|( 1-\triangle_{\SS^2})^{(a-1)/2} \pa^\alpha\Omega_{ij}f\|_{L^2}+ \|f\|_{H^m}.\eeno
Thanks to \eqref{equ2}, we also have \beno && \sum_{|\alpha|\le m} \sum_{1\le i<j\le3}\|( 1-\triangle_{\SS^2})^{(a-1)/2} \Omega_{ij}\pa^\alpha f\|_{L^2} +\|   f\|_{H^m} \\ 
&&\sim  \sum_{|\alpha|\le m}\|( 1-\triangle_{\SS^2})^{ a/2}\pa^\alpha f\|_{L^2}.  \eeno
Notice the fact
\beno  &&\sum_{|\alpha|\le m} \sum_{1\le i<j\le3}\|( 1-\triangle_{\SS^2})^{(a-1)/2} \Omega_{ij}  \pa^\alpha f\|_{L^2}-\sum_{|\beta|\le m} \|( 1-\triangle_{\SS^2})^{(a-1)/2}  \pa^\beta f\|_{L^2}
 \\&&\lesssim \sum_{|\alpha|\le m} \sum_{1\le i<j\le3}\|( 1-\triangle_{\SS^2})^{(a-1)/2}  \pa^\alpha \Omega_{ij}   f\|_{L^2}\\&&\lesssim \sum_{|\alpha|\le m} \sum_{1\le i<j\le3}\|( 1-\triangle_{\SS^2})^{(a-1)/2} \Omega_{ij} \pa^\alpha f\|_{L^2}+\sum_{|\beta|\le m} \|( 1-\triangle_{\SS^2})^{(a-1)/2}  \pa^\beta f\|_{L^2},\eeno
we derive that
\ben\label{e5}  &&\sum_{|\alpha|\le m}\|( 1-\triangle_{\SS^2})^{ a/2}\pa^\alpha f\|_{L^2}-C_1\|(1-\triangle_{\SS^2})^{(a-1)/2} f\|_{H^{m}}^2\notag\\ &&\lesssim\|(1-\triangle_{\SS^2})^{a/2} f\|_{H^m}^2\\&&
\lesssim  \sum_{|\alpha|\le m}\|( 1-\triangle_{\SS^2})^{ a/2}\pa^\alpha f\|_{L^2}+C_2\|(1-\triangle_{\SS^2})^{(a-1)/2} f\|_{H^{m}}^2. \notag\een
Observe that \beno &&\|(1-\triangle_{\SS^2})^{(a-1)/2} f\|_{H^{m}}\sim \|(1-\triangle_{\SS^2})^{(a-1)/2}\langle D\rangle^m f\|_{L^2} \\&&
\le \eta\|(1-\triangle_{\SS^2})^{a/2} \langle D\rangle^mf\|_{L^2}+C_\eta\| \langle D\rangle^m f\|_{L^2}\\&&
\le \eta\|(1-\triangle_{\SS^2})^{a/2}  f\|_{H^m}+C_\eta\|  f\|_{H^m}. \eeno
Then \eqref{e5} yields the desired result and we complete the inductive argument to \eqref{equ3}. We end the proof of the lemma.
\end{proof}
\smallskip

We are in a position to prove \eqref{labelint}.

  \begin{lem}\label{antin3}  If $T_h f(v)\eqdefa f(v+h)$, then for $s\ge0$, it holds
\beno  \|(-\triangle_{\SS^2})^{s/2} T_h f\|_{L^2(\R^3)}\lesssim \langle h\rangle^s\big(\|(-\triangle_{\SS^2})^{s/2} f\|_{L^2(\R^3)}+\|f\|_{H^s(\R^3)}\big).\eeno
\end{lem} 
\begin{proof}We begin with the case $0\le s\le1$. Thanks to Lemma \ref{antin2}, we have
\beno\|(-\triangle_{\SS^2})^{s/2} T_h f\|_{L^2(\R^3)}&\lesssim& \sum_{1\le i<j\le 3}  \|  T_h f\|_{\mathcal{D}_{\Omega_{ij}}(s,2)}. \eeno
Since \beno \| \Omega_{ij}T_h f\|_{L^2}\lesssim \langle h\rangle(\|f\|_{H^1}+\| \Omega_{ij}  f\|_{L^2}),\qquad  \|  T_h f\|_{L^2}= \|f\|_{L^2},\eeno
then applying  Lemma \ref{intp} with $A=\Omega_{ij}$ and $B_k=\pa_k$, we get
\ben\label{transes}   \|  T_h f\|_{\mathcal{D}_{\Omega_{ij}}(s,2)}\lesssim  \langle h\rangle^s(\|f\|_{H^s}+\| f\|_{\mathcal{D}_{\Omega_{ij}}(s,2)}).\een
From which together with Lemma \ref{antin2}, we obtain the desired result.

Next we turn to the case $1<s\le 2$. Suppose $s=1+a$. Then by Lemma \ref{antin8}, we have
\beno \| (-\triangle_{\SS^2})^{s/2}f  \|_{L^2(\R^3)}&\lesssim& \sum_{1\le i<j\le 3}(\|f\|_{L^2}+ \|(-\triangle_{\SS^2})^{a/2}\Omega_{ij}f\|_{L^2}),\\
&\lesssim& \sum_{1\le i<j\le 3} \big( \|f\|_{L^2}+\|\Omega_{ij}  f\|_{\mathcal{D}_{\Omega_{ij}}(a,2)}\big).  \eeno
Therefore,
\beno \|(-\triangle_{\SS^2})^{s/2} T_h f\|_{L^2(\R^3)}&\lesssim& \|f\|+\sum_{1\le i<j\le 3}  \|\Omega_{ij}  T_hf\|_{\mathcal{D}_{\Omega_{ij}}(a,2)}
\\ &\lesssim& \|f\|+\sum_{1\le i<j\le 3} \big( \| T_h \Omega_{ij} f\|_{\mathcal{D}_{\Omega_{ij}}(a,2)}\\&&\quad+\langle h\rangle (\|T_h\pa_i f\|_{\mathcal{D}_{\Omega_{ij}}(a,2)}+\|T_h\pa_j f\|_{\mathcal{D}_{\Omega_{ij}}(a,2)})\big). \eeno
Thanks to \eqref{transes} and Lemma \ref{antin2}, we are led to
\beno &&\|(-\triangle_{\SS^2})^{s/2} T_h f\|_{L^2(\R^3)}\\
&&\lesssim \langle h\rangle^{s} \sum_{1\le i<j\le 3} \big( \|(-\triangle_{\SS^2})^{a/2} \Omega_{ij} f\|_{L^2(\R^3)}+\|\Omega_{ij}  f\|_{H^a(\R^3)}+\|f\|_{H^s}\\
&&\quad+
 \|(-\triangle_{\SS^2})^{a/2} \pa_{i} f\|_{L^2(\R^3)}+ \|(-\triangle_{\SS^2})^{a/2} \pa_{j} f\|_{L^2(\R^3)}\big).  \eeno
Due to Lemma \ref{antin8}, we deduce
\beno \|(-\triangle_{\SS^2})^{s/2} T_h f\|_{L^2(\R^3)}\lesssim \langle h\rangle^s (\|(-\triangle_{\SS^2})^{s/2} f\|_{L^2(\R^3)}+\|f\|_{H^s}). \eeno

Finally we use the inductive method to handle the case $s>2$.   We assume the result holds for $s\le 2N$. Suppose $s\in(2N,2N+2]$. Then
\beno \|(-\triangle_{\SS^2})^{s/2} T_h f\|_{L^2(\R^3)}&=&\sum_{1\le i<j\le 3}\|(-\triangle_{\SS^2})^{s/2-1} \Omega_{ij}^2 T_h f\|_{L^2(\R^3)}. \eeno
It is easy to check that
\beno  \Omega_{ij}^2T_h f&=&T_h(\Omega_{ij}^2f)-h_iT_h(\pa_j\Omega_{ij}f)+h_jT_h(\pa_i \Omega_{ij}f)-h_i
T_h(\Omega_{ij}\pa_jf)\\&&+h_i^2T_h(\pa^2_jf)-h_j^2T_h(\pa_i^2 f)+h_jT_h (\Omega_{ij} \pa_if).\eeno
Then by the inductive assumption and Lemma \ref{antin8}, we deduce that
\beno  &&\|(-\triangle_{\SS^2})^{s/2} T_h f\|_{L^2(\R^3)}\\
&&\lesssim \langle h\rangle^s\sum_{1\le i,j\le 3} \big( \|(-\triangle_{\SS^2})^{s/2-1}\Omega_{ij}^2 f\|_{L^2(\R^3)}
+\|\Omega_{ij}^2f\|_{H^{s-2}}+\|(-\triangle_{\SS^2})^{s/2-1}\pa_j\Omega_{ij} f\|_{L^2(\R^3)}
+\|\pa_j\Omega_{ij}f\|_{H^{s-2}}\\&&\quad+\|(-\triangle_{\SS^2})^{(s-1)/2}\pa_{i} f\|_{L^2(\R^3)}
+\|\Omega_{ij}\pa_{i}f\|_{H^{s-2}}\big) +\|(-\triangle_{\SS^2})^{(s-2)/2}\pa_{i}^2 f\|_{L^2(\R^3)}
+\| \pa^2_{i}f\|_{H^{s-2}}\big)\\
&&\lesssim \langle h\rangle^s (\|(-\triangle_{\SS^2})^{s/2} f\|_{L^2(\R^3)}+\|f\|_{H^s}), \eeno
which completes the inductive argument to get the result.
\end{proof}

\setcounter{equation}{0}
\section{Conclusions and Perspectives}
In this paper, by making full use of two types of the decomposition performed in phase and frequency spaces and the geometric decomposition, we successfully establish several lower and upper bounds for Boltzmann collision operator in weighted Sobolev spaces and in anisotropic spaces. By comparing with the behavior of the linearized operator, we show that all the bounds are sharp. We further show that  the strategy of the proof is  so robust that we can apply it to the rescaled Boltzmann collision operator (see assumption {\bf (B1)}). Finally we obtain sharp bounds for the Landau collision operator by so-called grazing collision limit. 
 \medskip
 
 It is very interesting to see whether our method used here can be applied or not  to capture the exact behavior of the operator under the assumption  \eqref{abc2} or  \eqref{ab2}. In Section 4, we make an attempt to analyze the Boltzmann collision operator in the process of the grazing collision limit (see Lemma \ref{lbplim}). We conjecture that if $\mathcal{L}^\epsilon_B$ is the linearized Boltzmann collision operator with the rescaled kernel under the assumption {\bf (B1)}, then 
\ben\label{conject} \langle\mathcal{L}^\epsilon_B f, f\rangle_v+\|f\|_{L^2_{\gamma/2}}^2\sim \|W^\epsilon(D)f\|_{L^2_{\gamma/2}}^2+ \|W^\epsilon((-\triangle_{\SS^2})^{\f12})f\|_{L^2_{\gamma/2}}^2+\|W^\epsilon f\|_{L^2_{\gamma/2}}^2,\een
where the symbol function $W^\epsilon$ is defined in \eqref{symboloflimit}.  Based on the conjecture, we may ask:
\begin{enumerate}
\item How to establish a unified framework to solve the Boltzmann and Landau equations in the perturbation regime and   prove the asymptotic formula \eqref{asymformula};
\item How to describe the behavior of the spectrum of the operator $\mathcal{L}^\epsilon_B$ in the limit $\epsilon\rightarrow0$ for $\gamma\in [-2,-2s)$; we  recall  that the spectrum gap exists for $\mathcal{L}_B$ if and  only if $\gamma\ge -2s$ while it exists for $\mathcal{L}_L$ if and only if $\gamma\ge-2$.
\end{enumerate}
The similar conjecture can be questioned on the operator with the assumption \eqref{abc2} or with the Coulomb potential. Then  the asympototics of the Boltzmann equation from short-range interactions to long-range interactions and  the Landau approximation for Coulomb potential can be investigated.
 
   \bigskip

{\bf Acknowledgments:}  Ling-Bing He is supported by NSF of China under Grant 11001149 and 11171173 and the Importation and Development of High-Caliber Talents Project of Beijing Municipal Institutions.  The author would like to express his gratitude to colleagues Xu-Guang Lu and Pin Yu for  profitable discussions.

\end{document}